\documentclass[11pt]{amsbook}

\usepackage{amsmath}
\usepackage{amsthm}
\usepackage{amssymb}
\usepackage{amscd}
\usepackage{enumerate}
\usepackage{textcomp}

\theoremstyle{plain}
\newtheorem{theorem}{Theorem}
\newtheorem{lemma}[theorem]{Lemma}
\newtheorem{corollary}[theorem]{Corollary}
\newtheorem{proposition}[theorem]{Proposition}

\newtheorem{example}[theorem]{Example}
\newtheorem*{claim}{Claim}
\newtheorem*{case1}{Case 1}
\newtheorem*{case2}{Case 2}

\theoremstyle{definition}
\newtheorem{definition}[theorem]{Definition}
\newtheorem{setting}[theorem]{Setting}

\numberwithin{theorem}{section}

\def\forces{\mathrel {||}\joinrel \relbar}

\DeclareMathOperator{\Add}{Add}
\DeclareMathOperator{\cf}{cf}
\DeclareMathOperator{\cof}{cof}
\DeclareMathOperator{\Coll}{Coll}

\DeclareMathOperator{\dcl}{dcl}
\DeclareMathOperator{\dom}{dom}
\DeclareMathOperator{\Fil}{Fil}
\DeclareMathOperator{\HOD}{HOD}
\DeclareMathOperator{\im}{im}
\DeclareMathOperator{\Lev}{Lev}
\DeclareMathOperator{\lh}{lh}
\DeclareMathOperator{\OR}{OR}
\DeclareMathOperator{\ro}{ro}
\DeclareMathOperator{\ssup}{ssup}
\DeclareMathOperator{\TC}{TC}
\DeclareMathOperator{\Ult}{Ult}
\DeclareMathOperator{\ZF}{ZF}

\author{Jacob Davis}
\title{Universal Graphs at $\aleph_{\omega_1+1}$ and Set-theoretic Geology}

\begin{document}

\frontmatter
\maketitle

\section*{Acknowledgements}

I would like to thank my advisor James Cummings for all his explanations, suggestions and enthusiasm throughout my time in graduate school. I am grateful to Joel Hamkins for his assistance with the writing of this thesis. I would like to thank Daniel Rodr\'iguez and many other graduate students for their help and support over the past six years.

Furthermore I would like to thank Professors Cummings, Schimmerling and Grossberg for their classes in logic; the mathematics department of Carnegie Mellon for its support; and all the undergraduates I have had the pleasure of teaching.

\section*{Abstract}

This thesis consists of two parts: the construction of a jointly universal family of graphs, and then an exploration of set-theoretic geology.

Firstly we shall construct a model in which $2^{\aleph_{\omega_1}}=2^{\aleph_{\omega_1+1}}=\aleph_{\omega_1+3}$ but there is a jointly universal family of size $\aleph_{\omega_1+2}$ of graphs on $\aleph_{\omega_1+1}$. We take a supercompact cardinal $\kappa$ and will use Radin forcing with interleaved collapses to change $\kappa$ into $\aleph_{\omega_1}$. Prior to the Radin forcing we perform a preparatory iteration to add functions from $\kappa^+$ into Radin names for what will become members of the jointly universal family on $\kappa^+$. The same technique can be used with any uncountable cardinal in place of $\omega_1$.

Secondly we explore various topics in set-theoretic geology. We begin by showing that a class Easton support iteration of $\Add(\kappa,1)$ at $\kappa$ regular results in a universe that is its own generic mantle. We then consider set forcings $\mathbb{P}$, $\mathbb{Q}$, $\mathbb{R}$ and $\mathbb{S}$ with respective generics $G$, $H$, $I$ and $J$ such that $V[G][I]=V[H][J]$ and show that $V[G]$ and $V[H]$ must have a shared ground via $(|\mathbb{R}|+|\mathbb{S}|)^+$-cc forcing. This allows a similar analysis of the related situation when $\mathbb{P}$ is replaced by a class iteration and $V[H]$ by a generic ground of $V[G]$. We conclude with a simple characterisation of the mantle of a class forcing extension, and an investigation of the possibilities for a version of the intermediate model theorem that applies to class forcing.

\tableofcontents

\mainmatter

\chapter{Introduction}

We write $x:=y$ to mean $x$ is defined to equal $y$, and $x=:y$ to mean $y$ is defined to equal $x$. We write $f:A\rightharpoonup B$ for a partial function from $A$ to $B$. Our forcing convention is that $p\leq q$ means $p$ is stronger (more informative) than $q$. For forcing conditions $p$ and $q$ we write $p\parallel q$ to mean $p$ is compatible with $q$; for a formula $\varphi$ we write $p\parallel\varphi$ to mean $p$ decides whether or not $\varphi$ is true. Given an ultrafilter $u$, the quantification $\forall_u x:\varphi(x)$ will signify that $\{x\mid\varphi(x)\}\in u$.

\section{Universal graphs at $\aleph_{\omega_1+1}$}

For a cardinal $\mu$, a {\em universal graph on $\mu$} is a graph on $\mu$ into which every graph on $\mu$ can be embedded as an induced subgraph. A family of graphs on $\mu$ is {\em jointly universal on $\mu$} if every graph on $\mu$ can be embedded into at least one of them. We are interested in obtaining jointly universal families of small cardinality for $\mu$ a successor cardinal of the form $\kappa^+$.

If $2^{\kappa}=\kappa^+$ then by a standard model-theoretic construction there is a saturated and hence universal graph on $\kappa^+$. This holds even if $2^{\kappa^+}$ is large. So we are interested in cases when $2^{\kappa}>\kappa^+$. If $\kappa$ is regular then as shown by D{\v z}amonja and Shelah in \cite{universalModels} it is consistent to have a jointly universal family on $\kappa^+$ of size $\kappa^{++}$ whilst $2^{\kappa^+}$ is arbitrarily large. If $\kappa$ is singular than matters are generally more problematic. D{\v z}amonja and Shelah introduce a new approach in \cite{2author} that begins with $\kappa$ supercompact and performs a preparatory iteration to add functions that after Prikry forcing will become embeddings into a family of  jointly universal graphs, whilst preserving some of the supercompactness of $\kappa$, followed by Prikry forcing to change the cofinality of $\kappa$. This enables them to build a model where $\cf(\kappa)=\omega$, $2^{\kappa^+}>\kappa^{++}$ and there is a jointly universal family on $\kappa^{+}$ of size $\kappa^{++}$. In \cite{5author} Cummings, D{\v z}amonja, Magidor, Morgan and Shelah modify this construction to use Radin forcing and achieve $\cf(\kappa)>\omega$ and $2^{\kappa^+}>\kappa^{++}$ with a jointly universal family on $\kappa^+$ of size $\kappa^{++}$. Then in \cite{3author} Cummings, D{\v z}amonja and Morgan employ Prikry forcing with interleaved collapses to build a model with $2^{\aleph_{\omega+1}}>\aleph_{\omega+2}$ and a jointly universal family on $\aleph_{\omega+1}$ of size $\aleph_{\omega+2}$. We will use a preparatory forcing followed by Radin forcing with interleaved collapses to prove the following theorem.

\begin{theorem}
  Let $\kappa$ be supercompact and $\lambda<\kappa$ regular uncountable. Then there is a forcing extension in which $\kappa=\aleph_{\lambda}$, $2^{\aleph_{\lambda}}=2^{\aleph_{\lambda+1}}=\aleph_{\lambda+3}$ and there is a jointly universal family of graphs on $\aleph_{\lambda+1}$ of size $\aleph_{\lambda+2}$.
\end{theorem}

In section \ref{ultrafilter_sequences} we consider sequences of ultrafilters $\vec{u}$ from which it is possible to derive a version $\mathbb{R}_{\vec{u}}$ of Radin forcing with interleaved collapses. The forcing is similar to the one used by Foreman and Woodin in \cite{ForemanWoodin} but differs in the forcing interleaved and some technical details. Also we will show the desired properties of the forcing directly rather than proving that a supercompact Radin forcing has these properties and then projecting them.

We identify certain useful properties of sequences of ultrafilters that have been derived from supercompactness embeddings, and denote the class of sequences possessing these properties by $\mathcal{U}$. In section \ref{properties_of_radin} we prove some results about the forcing $\mathbb{R}_{\vec{u}}$ when $\vec{u}\in\mathcal{U}$; in particular that it has the Prikry property and that its generic filters can be conveniently characterised. In section \ref{preparatory_forcing} we define a preparatory forcing $\mathbb{Q}_{\vec{u}}$ that adds functions which, after Radin forcing, will become embeddings from graphs on $\kappa^+$ into a graph on $\kappa^+$ that we intend to make a member of our jointly universal family. We also prove that this preparatory forcing has properties including $\kappa$-directed closure and the $\kappa^+$-cc.

In section \ref{construction_of_model} we begin with $\kappa$ supercompact and perform a Laver preparation forcing. We then use a diamond sequence to identify ultrafilter sequences $\vec{u}^{\gamma}$ for $\gamma<\kappa^{+4}$, and carry out an iteration of the $\mathbb{Q}_{\vec{u}^{\gamma}}$ forcings. This allows us to extend a supercompactness embedding $j$ from $V$ to the generic extension, and from this $j$ we derive an ultrafilter sequence $\vec{u}$ in $\mathcal{U}$ and take $J$ that is $\mathbb{R}_{\vec{u}}$-generic over the universe resulting from the $\mathbb{Q}_{\vec{u}^{\gamma}}$-iteration. We show that there is a stationary set $S$ of points $\gamma$ in $\kappa^{+4}$ where $\vec{u}$ restricts to $\vec{u}^{\gamma}$ and $\vec{u}^{\gamma}\in\mathcal{U}$; then the characterisation of generic filters will show that $J$ is also generic for $\mathbb{R}_{\vec{u}^{\gamma}}$ over the same universe. Our final model will be built by stopping the iteration at a point in $S$ that is also a limit of $\kappa^{++}$-many members of $S$ and adjoining $J$; we will then have constructed $\kappa^{++}$-many graphs to use as members of our joint universal family, together with embeddings of every graph on $\kappa^+$ into them.

\section{Set-theoretic geology}

A \textit{ground} of the universe $V$ is a model $W\subseteq V$ of ZFC such that there is a forcing $\mathbb{P}\in W$ with a generic $G$ such that $W[G]=V$. Laver in \cite{laver_geology} and independently Woodin in \cite{woodin} proved that in this case $W$ will be a class of $V$, using parameters from $W$. Hamkins and Reitz in \cite{hamkins} and \cite{Reitz} used these ideas to formulate the \textit{Ground Axiom} which asserts that there are no grounds of $V$ other than $V$ itself; in particular they showed that this axiom is first-order expressible. In \cite{FHR} these ideas were developed by Fuchs, Hamkins and Reitz to define the \textit{mantle} of $V$, written $M^V$, as the intersection of all of the grounds of $V$ and to show that it is a class of $V$.

The associated notion of a \textit{generic ground} of $V$, which is defined to be the ground of some forcing extension of $V$, is also given in \cite{FHR}. From this we can build the \textit{generic mantle}, written $gM^V$, which is the intersection of all of the generic grounds of $V$. It is clear that every ground is a generic ground (since $V$ is trivially a generic extension of itself) so $gM^V\subseteq M^V$. Fuchs, Hamkins and Reitz ask whether the mantle is in fact always equal to the generic mantle. Recent work by Usuba in \cite{usuba} answers this question in the affirmative.

It is possible to generalise the idea of forcing with set partial orders to the use of class partial orders, for which we obtain class generics; see \cite[Chapter 8]{FK} for a detailed exposition. This presents new challenges in ensuring that the generic universe will be a model of ZFC, but has the potential to make changes to the entire structure of the universe and so obtain interesting behaviours of the mantle and generic mantle. In particular in \cite{FHR} class forcing is used to show that every model of ZFC is the mantle of another model of ZFC.

In Section \ref{easton_generic_mantle} we consider the class Easton support iteration of $\Add(\kappa,1)$ at $\kappa$ regular, which was used in \cite{FHR}, and show that the generic mantle of the resulting universe $V[G]$ is equal to $V[G]$ itself; this answers Question 69 of that paper. We then look for ways to generalise to a wider range of class forcing extensions.

We begin in Section \ref{set_forcing_intersection} by considering the situation where $\mathbb{P}$ and $\mathbb{Q}$ are set-sized forcings with generics $G$ and $H$ respectively and there are further forcings $\mathbb{R}\in V[G]$ and $\mathbb{S}\in V[H]$ with respective generics $I$ and $J$ such that $V[G][I]=V[H][J]$. This situation is a reduced version of one in which $\mathbb{P}$ is a class forcing and $V[H]$ is a generic ground of $V[G]$, so we are interested in the possibilities for $V[G]\cap V[H]$. In Theorem \ref{main_sets} we show that, regardless of the properties of $\mathbb{P}$ and $\mathbb{Q}$, there will be an inner model $U\subseteq V[G]\cap V[H]$ from which $V[G]$ and $V[H]$ can each be recovered by $((|\mathbb{R}|+|\mathbb{S}|)^+)^{V[G]}$-cc forcings. This contrasts with Proposition \ref{intersection_not_model_of_ZFC} where we see that $V[G]\cap V[H]$ itself may not be a model of ZFC, and with Proposition \ref{coll_add_example} which gives an example where $V[G]\cap V[H]=V$ but $\mathbb{P}$ is not $((|\mathbb{R}|+|\mathbb{S}|)^+)^V$-cc.

We use these results in Section \ref{intersections_part_1} to show that if $\mathbb{P}$ is a class forcing with generic $G$ that preserves a sufficient number of weakly compact cardinals and $W$ is a generic ground of $V[G]$ via forcings $\mathbb{R}\in V[G]$ and $\mathbb{S}\in W$, with respective generics $I$ and $J$ such that $V[G][I]=W[J]$, then there is a common ground of $V[G]$ and $W$ via $((|\mathbb{R}|+|\mathbb{S}|)^+)^{V[G]}$-cc forcings.

In Section \ref{characterising_the_mantle} we consider a universe $V[G]$ formed by class forcing and give a simple characterisation of $M^{V[G]}$ that avoids any reference to the posets which are naturally involved in the construction of the mantle. Then in Section \ref{intersections_part_2} we present an alternative analysis of the intersection of a universe resulting from a class forcing extension with one of its generic grounds; this approach relies on covering rather than weakly compact cardinals, and shows that the intersection cannot just be a set-sized extension of the starting model.

Whenever we have a set forcing $\mathbb{P}$ with generic $G$ and an intermediate model $V\subseteq W\subseteq V[G]$ then we are able to form a complete sub-algebra $\mathbb{A}$ of $\ro(\mathbb{P})$ such that $V[G\cap\mathbb{A}]=W$, and this is extremely useful in the analysis of such intermediate models. In Section \ref{class_intermediate_models} we explore the possibilities for a similar result when $\mathbb{P}$ is a class forcing. The naive $\ro(\mathbb{P})$ would be a collection of classes and so not itself a class, but we are able to form an different Boolean algebra forcing-equivalent to $\mathbb{P}$ that is complete under set-sized supremums and infimums, though not class-sized ones. We discuss the difficulties this entails and conclude with a theorem that constructs the intermediate model $W$ as a forcing extension of $V$ via a weakened notion of class forcing.

\chapter{Universal graphs at $\aleph_{\omega_1+1}$}

\section{Ultrafilter sequences and the definition of $\mathbb{R}_{\vec{u}}$} \label{ultrafilter_sequences}

\subsection{The nature of ultrafilter sequences} We will be building sequences of the following form.

\begin{definition}
  A sequence $\vec{u}=\langle \kappa, u_i, \mathcal{F}_i\mid i<\lambda \rangle$ (which means that there is a single $\kappa$ together with $\lambda$-many each of the $u_i$ and $\mathcal{F}_i$) is a {\em proto ultrafilter sequence} if $\lambda<\kappa$, the $u_i$ are $\kappa$-complete ultrafilters on $V_{\kappa}$ and the $\mathcal{F}_i$ are sets of partial functions from $V_{\kappa}$ to $V_{\kappa}$. We will write $\kappa(\vec{u})$ for $\kappa$ and $\lh\vec{u}$ for $\lambda$ which we also call the {\em length} of $\vec{u}$. We stress that our use of the term ``length'' here differs from the usual convention.
\end{definition}

For $\beta$ a strongly inaccessible cardinal we define $\mathbb{C}(\alpha,\beta)$ to be the poset $\Coll(\alpha^{+5},<\beta)$ and $\mathbb{B}(\alpha,\beta)$ to be the regular open algebra derived from this poset. Note that $\mathbb{C}(\alpha,\beta)$ is contained in $V_{\beta}$ and has the $\beta$-cc so we are free to regard conditions in $\mathbb{B}(\alpha,\beta)$ as members of $V_{\beta}$. Given sequences $\vec{v}$ and $\vec{w}$ with $\kappa(\vec{v})<\kappa(\vec{w})$ we will also write $\mathbb{B}(\vec{v},\vec{w})$ for $\mathbb{B}(\kappa(\vec{v}),\kappa(\vec{w}))$. This is the forcing that we will interleave into our Radin generic sequence.

\begin{definition}
  Let $\kappa$ be strongly inaccessible, $i<\kappa$ and $u$ a $\kappa$-complete ultrafilter on $V_{\kappa}$ concentrating on proto ultrafilter sequences of length $i$. Then a {\em $u$-constraint} is a partial function $h:V_{\kappa}\rightharpoonup V_{\kappa}$ such that:
  \begin{itemize}
    \item $\dom h$ is in $u$ and consists of proto ultrafilter sequences of length $i$.
    \item For all $\vec{w}$ in $\dom h$, $h(\vec{w})\in \mathbb{B}(\kappa(\vec{w}),\kappa)-\{0\}$.
  \end{itemize}
  An {\em ultrafilter sequence} is defined by recursion on $\kappa(\vec{u})$ to be a proto ultrafilter sequence $\vec{u}=\langle \kappa, u_i, \mathcal{F}_i\mid i<\lambda\rangle$ such that each $\mathcal{F}_i$ is a non-empty set of $u_i$-constraints, and each $u_i$ concentrates on ultrafilter sequences of length $i$.
\end{definition}

Observe that if we form the ultrapower $j_u:V\rightarrow\Ult(V,u)$ we can regard $u$-constraints (modulo $u$) as representing members of the Boolean algebra $\mathbb{B}(\kappa,j_u(\kappa))^{\Ult(V,u)}$.

\begin{definition}
  We will need an auxiliary notion of {\em supercompact ultrafilter sequences}. Such sequences will be recursively defined to have the form $\vec{u}^*=\langle z, u^*_i, H^*_i \mid i<\lambda\rangle$ where there is some $\kappa(\vec{u}^*):=\kappa>\lambda$ with $z$ a set of ordinals that is a superset of $\kappa$, each $u^*_i$ is an ultrafilter on $[\kappa^{+4}]^{<\kappa}\times V_{\kappa}^{2i}$ that concentrates on supercompact ultrafilter sequences of length $i$, and each $H^*_i$ is a {\em $u^*_i$-constraint}. This last means that $\dom H^*_i \in u^*_i$, and for $\vec{w}^*\in\dom H^*_i$ we have $H^*_i(\vec{w}^*)\in\mathbb{B}(\vec{w}^*,\vec{u}^*)-\{0\}$.

  We also define an ordering on $u^*_i$-constraints by $L^*\leq K^*$ if $\dom L^*\subseteq\dom K^*$ and $L^*(\vec{w}^*)\leq K^*(\vec{w}^*)$ for all $\vec{w}^*$ in $\dom L^*$. We shall use similar orderings for other functions whose domains are required to lie in some ultrafilter.
\end{definition}

Observe that if we form the ultrapower $j_{u^*}:V\rightarrow\Ult(V,u^*)$ then we can regard $u^*$-constraints (modulo $u^*$) as representing members of the Boolean algebra $\mathbb{B}(\kappa, j_{u^*}(\kappa))^{\Ult(V,u^*)}$.

\subsection{Constructing ultrafilter sequences}

For the remainder of this section we work in the following context.

\begin{setting}
  Let $2^{\kappa} = \kappa^{+4}$ with $j:V\rightarrow M$ witnessing that $\kappa$ is $\kappa^{+4}$-supercompact. Let $\lambda<\kappa$ be regular uncountable.
\end{setting}

We will use $j$ to inductively build an ultrafilter sequence $\vec{u}=\langle \kappa, u_i, \mathcal{F}_i\mid i<\lambda \rangle$ with $\kappa(\vec{u})=\kappa$. In doing so we will need to construct an auxiliary supercompact ultrafilter sequence $\vec{u}^*=\langle j``\kappa^{+4}, u^*_i, H^*_i\mid i<\lambda \rangle$.

We will also define a function $\pi$ from supercompact ultrafilter sequences to ultrafilter sequences, given by
$$\pi(\langle z^*, w^*_i, K^*_i \mid i<\bar{\lambda}\rangle) := \langle z^*\cap\bar{\kappa}, \pi_i(w^*_i), \pi'_i(w^*_i,K^*_i)\mid i<\bar{\lambda}\rangle$$
with $\pi_i$ and $\pi'_i$ to be built as part of the induction and $\bar{\kappa}:=\kappa(\langle z^*, w^*_i, K^*_i \mid i<\bar{\lambda}\rangle)$. Note that the $u^*_i$ concentrate on sequences where $z^*\cap\bar{\kappa}$ is inaccessible. We will ensure as we induct on $\bar{\lambda}\leq\lambda$ that
\begin{equation}\tag{*}
  j(\pi)(\langle j` `\kappa^{+4}, u^*_i, H^*_i \mid i<\bar{\lambda}\rangle) = \langle \kappa, u_i, \mathcal{F}_i \mid i<\bar{\lambda}\rangle.
\end{equation}
Suppose we have defined $u_i$, $u^*_i$, $\mathcal{F}_i$, $H^*_i$, $\pi_i$ and $\pi'_i$ for $i<\bar{\lambda}$; this gives us the definition of $\pi$ on sequences of length up to $\bar{\lambda}$. Define
$$u^*_{\bar{\lambda}} := \{X\subseteq[\kappa^{+4}]^{<\kappa}\times V_{\kappa}^{2\bar{\lambda}} \mid \langle j``\kappa^{+4}, u^*_i, H^*_i \mid i<\bar{\lambda}\rangle \in j(X)\}$$
and
$$u_{\bar{\lambda}} := \{Y\subseteq V_{\kappa}\mid \langle \kappa, u^i, \mathcal{F}^i \mid i<\bar{\lambda}\rangle \in j(Y)\}.$$
For $w^*$ an ultrafilter on $[\bar{\kappa}^{+4}]^{<\bar{\kappa}}\times V_{\bar{\kappa}}^{2\bar{\lambda}}$ define
$$\pi_{\bar{\lambda}}(w^*):= \{Y\subseteq V_{\bar{\kappa}}\mid \pi^{-1}``Y \in w^* \}.$$
Note by (*) that $Y\in u_{\bar{\lambda}}$ is equivalent to $\pi^{-1}``Y \in u^*_{\bar{\lambda}}$ which, since $\pi$ and $j(\pi)$ agree on $V_{\kappa}$, is equivalent to $j(\pi)^{-1}``Y\in u^*_{\bar{\lambda}}$ and so to $Y\in j(\pi_{\bar{\lambda}})(u^*_{\bar{\lambda}})$. Therefore $u_{\bar{\lambda}}=j(\pi_{\bar{\lambda}})(u^*_{\bar{\lambda}})$. We now pause the construction to make some definitions.

\begin{definition}
  Let $w^*$ be an ultrafilter on $[\bar{\kappa}^{+4}]^{<\bar{\kappa}}\times V_{\bar{\kappa}}^{2\bar{\lambda}}$ that concentrates on supercompact ultrafilter sequences of length $\bar{\lambda}$, and $K^*$ a $w^*$-constraint. Then for $A\in w^*$ and $\vec{x}$ an ultrafilter sequence we define
  $$b(K^*, A)(\vec{x}):=\bigvee\{K^*(\vec{x}^*)\mid \pi(\vec{x}^*)=\vec{x}, \vec{x}^*\in A\}\in\mathbb{B}(\kappa(\vec{x}),\bar{\kappa})-\{0\}.$$
  Observe that the domain of $b(K^*,A)$ is the projection of $A$ under $\pi_{\bar{\lambda}}$, so it is in $\pi_{\bar{\lambda}}(w^*)$. Observe also that for $A'\subseteq A$ we have $b(K^*,A')\leq b(K^*,A)$ pointwise, so as $A$ ranges over $w^*$ the equivalence classes generated by the $b(K^*,A)$ yield a non-trivial filter base in $\mathbb{B}(\bar{\kappa},j_{\pi_{\bar{\lambda}}(w^*)}(\bar{\kappa}))^{\Ult(V,\pi_{\bar{\lambda}}(w^*))}$. We shall call the induced filter $\Fil(K^*)$.
\end{definition}

Now given $w^*$ and $K^*$ we define $\pi'_{\bar{\lambda}}(w^*,K^*) = \{g\mid [g]_{\pi_{\bar{\lambda}}(w^*)}\in \Fil(K^*)\}$, which will conclude our definition of $\pi$ for sequences of length up to $k+1$. We note that $\pi'_{\bar{\lambda}}(w^*,K^*)$ consists of all $g$ such that $g\geq b(K^*, A)$ for some $A\in w^*$, where we ensure a pointwise inequality be shrinking the $A$ as necessary. It remains to choose $H^*_{\bar{\lambda}}$, and then once we have done so we will conclude by defining $\mathcal{F}_{\bar{\lambda}} = j(\pi'_{\bar{\lambda}})(u^*_{\bar{\lambda}}, H^*_{\bar{\lambda}})$, which is to say $\mathcal{F}_{\bar{\lambda}}=\{h\mid [h]_{u_{\bar{\lambda}}}\in\Fil(H^*_{\bar{\lambda}})\}$. Care must be taken in selecting $H^*_{\bar{\lambda}}$ because we wish to ensure that the filter $\mathcal{F}_{\bar{\lambda}}$ it induces will be an ultrafilter. The following lemma will be helpful to that end.

\begin{lemma}
  Let $b\in \mathbb{B}(\kappa,j_{_{\bar{\lambda}}}(\kappa))^{\Ult(V,_{\bar{\lambda}})}$ and $K^*$ a $u^*_{\bar{\lambda}}$-constraint. Then there is a $u^*_{\bar{\lambda}}$-constraint $L^*\leq K^*$ such that either $b\in\Fil(L^*)$ or $\neg b\in\Fil(L^*)$.
\end{lemma}

\begin{proof}
  Say $b =: [f]_{_{\bar{\lambda}}}$ and define $A:=\{\vec{x}^*\in \dom K^*\mid \pi(\vec{x})\in \dom f\} \in u^*_{\bar{\lambda}}$. Then for each $\vec{x}^* \in A$ take $L^*(\vec{x}^*)\leq K(\vec{x}^*)$ such that either $L^*(\vec{x}^*)\leq f(\pi(\vec{x}^*))$ or $L^*(\vec{x}^*)\leq \neg f(\pi(\vec{x}^*))$. Define $A^+$ to be the set of places in $A$ where the first case occurs, and $A^-$ to be the set of places where the second does. One of these is in $u^*_{\bar{\lambda}}$ and restricting the domain of $L^*$ to this set will give $L^*$ the required properties.
\end{proof}

For $\bar{\kappa}<\kappa$ the forcing $\mathbb{C}(\bar{\kappa}, \kappa)$ has the $\kappa$-chain condition, so $|\mathbb{B}(\bar{\kappa},\kappa)|=\kappa$. This tells us by elementarity that $|\mathbb{B}(\kappa,j_{_{\bar{\lambda}}}(\kappa))^{\Ult(V,_{\bar{\lambda}})}| = |j_{_{\bar{\lambda}}}(\kappa)| = 2^{\kappa} = \kappa^{+4}$. Now the $u^*_{\bar{\lambda}}$-constraints can be regarded as members of the regular open algebra $\mathbb{B}(\kappa, j_{u^*_{\bar{\lambda}}}(\kappa))^{\Ult(V,u^*_{\bar{\lambda}})}$, in the non-zero part of which the forcing $\mathbb{C}(\kappa, j_{u^*_{\bar{\lambda}}}(\kappa))^{\Ult(V,u^*_{\bar{\lambda}})}=\Coll(\kappa^{+5},<j_{u^*_{\bar{\lambda}}}(\kappa))^{\Ult(V,u^*_{\bar{\lambda}})}$ is dense. The $\kappa^{+4}$-supercompactness of $j_{u^*_{\bar{\lambda}}}$ tells us that the latter forcing is $\kappa^{+5}$-closed, so we can repeatedly apply the above lemma to obtain a $u^*_{\bar{\lambda}}$-constraint $H^*_{\bar{\lambda}}$ such that $\Fil(H^*_{\bar{\lambda}})$ is an ultrafilter. This concludes the inductive construction.

\subsection{Properties of ultrafilter sequences I}

We collect together all save one of the properties that we will want our ultrafilter sequences to possess. The final property is postponed because it requires the definition of $\mathbb{R}_{\vec{u}}$ to state.

Note that for $h\in \mathcal{F}_i$ and $s\in V_{\kappa}$ the $h\downharpoonright s$ notation used here means that the domain of $h$ is restricted to $\{\vec{w}\mid s \in V_{\kappa(\vec{w})}\}$.

\begin{definition}
  We define $\mathcal{U}'$ to be the class of all ultrafilter sequences $\vec{u}=\langle \bar{\kappa}, u_i, \mathcal{F}_i \mid i < \bar{\lambda} \rangle$ that satisfy the following properties:
  \begin{enumerate}
    \item The ultrafilter $u_i$ concentrates on sequences from $\mathcal{U}'$ of length $i$ (so this definition is recursive).
    \item If $h$ is in $\mathcal{F}_i$ and $h$ is equal to $h'$ modulo $u_i$ then $h'$ is also in $\mathcal{F}_i$.
    \item The set of Boolean values represented by the functions $\mathcal{F}_i$ is a $\bar{\kappa}$-complete ultrafilter on $\mathbb{B}(\bar{\kappa},j_{u_i}(\bar{\kappa}))^{\Ult(V,u_i)}$.
    \item (Normality) For all $i<\bar{\lambda}$, given $\langle h^s \mid s \in V_{\bar{\kappa}} \rangle$ with $h^s \in \mathcal{F}_i$ then there is $h \in \mathcal{F}_i$ such that $h \leq h^s \downharpoonright s$ for all $s \in V_{\bar{\kappa}}$.
    \item Let $i'<i''<\bar{\lambda}$ and $e\in \mathcal{F}_{i'}$. Then there is a $u_{i''}$-large set of ultrafilter sequences $\vec{w}=\langle\bar{\kappa}(\vec{w}), w_i, \mathcal{G}_i \mid i<i''\rangle$ such that $e\upharpoonright\bar{\kappa}(\vec{w})$ is in $\mathcal{G}_{i'}$.
  \end{enumerate}
\end{definition}

\begin{lemma}
  Let $\vec{u}$ be constructed from $j$ as above. Then $\vec{u}\in\mathcal{U'}$.
\end{lemma}

\begin{proof}
  The first three clauses are immediate.
  \begin{enumerate} [(1)]
    \setcounter{enumi}{3}
    \item (Normality) We are given $i<\lambda$ and $\langle h^s \mid s \in V_{\kappa} \rangle \subseteq \mathcal{F}_i$. Say $h^s\geq b(H^*_i,A^s)$ with $A^s\in u^*_i$. Take the diagonal intersection of the $A^s$,
  $$A:=\{\vec{w}^* \mid \forall s \in V_{\kappa(\vec{w}^*)}: \vec{w}^*\in A^s\}.$$
  We have $\forall s\in V_{\kappa(\vec{u}^*\upharpoonright i)}: \vec{u}^*\upharpoonright i \in j(A^s)$, which is to say $\vec{u}^*\upharpoonright i\in j(A)$ so $A\in u^*_i$. Then $h:=b(H^*_i,A)$ will be our candidate.

  Given $s\in V_{\kappa}$ we want $h \leq h^s \downharpoonright s$, so given $\vec{w}\in \dom h$ above $s$ we want $h(\vec{w})\leq h^s(\vec{w})$. Now $h(\vec{w})$ is the supremum of $K^*(\vec{w}^*)$ over $\vec{w}^*\in A$ such that $\pi(\vec{w}^*)=\vec{w}$. All of these $\vec{w}^*$ have $\kappa(\vec{w}^*)=\kappa(\vec{w})$ above $s$, so they must also be members of $A^s$. But $h^s(\vec{w})$ is the supremum of $K^*(\vec{w}^*)$ over members of $A^s$, so $h^s(\vec{w})\geq h(\vec{w})$.
    \item We are given $i'<i''<\lambda$ and $e\in \mathcal{F}_{i'}$ and note that $j(e)\upharpoonright\kappa(\vec{u}\upharpoonright i'') = e \in \mathcal{F}_{i'}$. Then by elementarity there is a $u_{i''}$-large set of sequences $\vec{w}=\langle \kappa(\vec{w}), w_i, \mathcal{G}_{i'}\rangle$, as required.
  \end{enumerate}
\end{proof}

\subsection{Definition of the Radin forcing $\mathbb{R}_{\vec{u}}$} \label{defineR}

We are given an ultrafilter sequence $\vec{u}\in\mathcal{U'}$ and define $\bar{\kappa}:=\kappa(\vec{u})$ and $\bar{\lambda}:=\lh\vec{u}$.

For notational convenience, given $\vec{w} =: \langle \kappa(\vec{w}), w_i, \mathcal{F}_i \mid i<\lh\vec{w}\rangle$ we will start writing $\mathcal{F}_{\vec{w},i}$ for $\mathcal{F}_i$ and $\mathcal{F}_{\vec{w}}$ for the set of
functions $e:V_{\kappa(\vec{w})}\rightharpoonup V_{\kappa(\vec{w})}$ such that defining $e_i:=e\upharpoonright\{\vec{v}\mid\lh\vec{v}=i\}$ gives us $e_i\in \mathcal{F}_{\vec{w},i}$ for all $i<\lh\vec{w}$. Note that $\dom e$ is permitted to include sequences that are longer than $\vec{w}$ itself.

\begin{definition}
  Let $\vec{w}\in\mathcal{U}'$. Then an {\em upper part} for $\mathbb{R}_{\vec{w}}$ is a member $e$ of $\mathcal{F}_{\vec{w}}$ such that:
  $$\forall\vec{v}\in\dom e: e\upharpoonright\kappa(\vec{v})\in \mathcal{F}_{\vec{v}}.$$
\end{definition}

We claim that any $e$ in $\mathcal{F}_{\vec{w}}$ can have its domain shrunk to make it into an upper part. Define $e^0:=e$ and then by the final clause of the definition of $\mathcal{U}'$ we have that
$$A^1:=\{\vec{v}\in\dom e\mid e^0\upharpoonright\kappa(\vec{v})\in \mathcal{F}_{\vec{v}}\}\in \bigcap w_i$$
so we can define $e^1:=e^0\upharpoonright A^1 \in \mathcal{F}_{\vec{w}}$. Iterating this process $\omega$ times and intersecting the $A^n$ we reach $e'\leq e$ which has the required property. From now on we shall perform such shrinking without comment when building forcing conditions.

\begin{definition}
  A {\em suitable triple} is $(\vec{w},e,q)$ satisfying the following conditions:
  \begin{itemize}
    \item $\vec{w}\in\mathcal{U}'$.
    \item  $e$ is an upper part for $\mathbb{R}_{\vec{w}}$.
    \item $q \in \mathbb{B}(\kappa(\vec{w}),\kappa)-\{0\}$.
  \end{itemize}
  A {\em direct extension} of $(\vec{w},e,q)$ is a suitable triple $(\vec{w},e',q')$ such that:
  \begin{itemize}
    \item $e'\leq e$ (i.e. $\dom e'\subseteq\dom e$ and $e' \leq e$ pointwise).
    \item $q' \leq q$.
  \end{itemize}
  Another suitable triple $(\vec{v},d,p)$ is {\em addable below} $(\vec{w},e,q)$ if it satisfies the following:
  \begin{itemize}
    \item $\vec{v} \in \dom e$.
    \item $d \leq e\upharpoonright\kappa(\vec{v})$.
    \item $p\leq e(\vec{v})$.
  \end{itemize}
\end{definition}

We observe that for every $\vec{v}\in\dom e$ the definition of ``upper part'' has assured us that $(\vec{v},e\upharpoonright\kappa(\vec{v}),e(\vec{v}))$ is both a suitable triple and addable below $(\vec{w},e,q)$.

\begin{definition}
  A {\em condition} in $\mathbb{R}_{\vec{u}}$ is a finite sequence
  $$s = ((\vec{w}_0,e_0,q_0),...,(\vec{w}_{n-1},e_{n-1},q_{n-1}), (\vec{u}, h))$$
  such that each $(\vec{w}_k,e_k,q_k)$ is a suitable triple, the $\kappa(\vec{w}_k)$ are increasing, $q_k\in\mathbb{B}(\vec{w}_k,\vec{w}_{k+1})$, and $h$ is an upper part for $\mathbb{R}_{\vec{u}}$. We also require that $\kappa(\vec{w}_0)=\omega$, $\lh\vec{w}_0=0$ and $e_0=\phi$. We will call such a $((\vec{w}_0,e_0,q_0),...,(\vec{w}_{n-1},e_{n-1},q_{n-1}))$ a {\em lower part} for the forcing.

  {\em Extension} in $\mathbb{R}_{\vec{u}}$ is given by $s'\leq s$ if
  $$s'=((\vec{v}_0,d_0,p_0),...,(\vec{v}_{m-1},d_{m-1},p_{m-1}),(\vec{u},h'))$$
  such that $h'\leq h$, every $\vec{w}_k$ occurs as some $\vec{v}_l$, and every $(\vec{v}_l,d_l,p_l)$ is either a direct extension of one of the $(\vec{w}_k,e_k,q_k)$ or addable below one of them or addable below $(\vec{u},h)$.

  {\em Direct extension} in $\mathbb{R}_{\vec{u}}$ is given by $s'\leq^* s$ if
  $$s'=((\vec{w}_0,e'_0,q'_0),...,(\vec{w}_{n-1},e'_{n-1},q'_{n-1}), (\vec{u}, h'))$$
  with $h'\leq h$ and $(\vec{w}_k,e'_k,q'_k)$ a direct extension of $(\vec{w}_k,e_k,q_k)$ for $k<n$. \

  For lower parts $r$ and $r'$ we define extension $r'\leq r$ in the same way as for conditions, except that all triples from $r'$ must by either direct extensions of, or addable below, a triple from $r$. Note that this compels $\kappa(\max r') = \kappa(\max r)$. Likewise we have a notion of $\leq^*$ on lower parts, and a {\em $^*$-open} set of lower parts is one that is downward-closed under this relation.
\end{definition}

Observe that a forcing condition is required to have a triple $((\langle\omega\rangle,\phi,p))$ as a member of its stem for some $p$. However we shall write $((\vec{u},h))$ as an abbreviation for $((\langle\omega\rangle,\phi,\phi),(\vec{u},h))$ at times when we are only interested in the upper part of the forcing.

If we force below a condition $((\vec{u},h))$ such that $\dom h$ contains only sequences of length less than $\lambda$ then $\mathbb{R}_{\vec{u}}$ will add a generic sequence of the form $\langle \vec{w}_{\alpha}, g_{\alpha} \mid \alpha<\omega^{\lambda}\rangle$, where $g_i$ is generic in $\mathbb{B}(\vec{w}_{\alpha},\vec{w}_{\alpha+1})$. The $\omega^{\lambda}$ is ordinal exponentiation so as $\lambda$ is regular uncountable we in fact have $\omega^\lambda=\lambda$. This collapses all cardinals in the intervals $(\kappa(\vec{w}_{\alpha})^{+5},\kappa(\vec{w}_{\alpha+1}))$ and we shall see later that it preserves all other cardinals, so it will make $\kappa$ into $\aleph_{\lambda}$.

More generally, forcing with $\mathbb{R}_{\vec{u}}$ will add a generic sequence $\langle \vec{w}_{\alpha}, g_{\alpha} \mid \alpha<\theta+\lambda\rangle$
for some ordinal $\theta$.

\subsection{Properties of ultrafilter sequences II}

We are finally in a position to make the definition that we will use during the main construction.

\begin{definition}
  The class $\mathcal{U}$ is defined recursively to consist of all $\vec{u}\in\mathcal{U}'$ such that the $u_i$ concentrate on members of $\mathcal{U}$, and $\vec{u}$ satisfies the following additional property. (Note that the $h'\upharpoonright\vec{w}$ is given by restricting the domain of $h'$ to sequences $\vec{v}$ such that $\kappa(\vec{v})<\kappa(\vec{w})$ and $\lh\vec{v}<\lh\vec{w}$.)
  \begin{enumerate}[(1)]
    \setcounter{enumi}{5}
    \item (Capturing) Let $h$ be an upper part for $\mathbb{R}_{\vec{u}}$ and $X$ a $^*$-open set of lower parts for $\mathbb{R}_{\vec{u}}$. Then there is an upper part $h' \leq h$ such that for all lower parts $s$ and all $i<\lh\vec{u}$ we have one of:
    \begin{enumerate}[(i)]
      \item For all $\vec{w}\in \dom h'_i$ there do not exist $e$ and $q\leq h'(\vec{w})$ such that $s\frown((\vec{w},e,q)) \in X$.
      \item For all $\vec{w}\in \dom h'_i$ and $\vec{x} \in \dom h'$ such that $\kappa(\vec{w})<\kappa(\vec{x})$ there are densely many $q$ in $\mathbb{B}(\vec{w},\vec{x})$ below $h'_i(\vec{w})$ such that $s\frown((\vec{w},h'\upharpoonright\vec{w},q)) \in X$.
    \end{enumerate}
  \end{enumerate}
\end{definition}

\begin{proposition}
  Let $\vec{u}$ be constructed from a supercompactness embedding $j$ as before. Then $\vec{u}\in\mathcal{U}$.
\end{proposition}

\begin{proof}
  Say $\vec{u}$ is of the form $\langle \kappa, u_i, \mathcal{F}_i \mid i < \lambda \rangle$ and the supercompact ultrafilter sequence used in the construction is $\langle z,u^*_i,H_i\mid i<\lambda\rangle$. We have already established that $\vec{u} \in \mathcal{U}'$ so it remains to prove capturing. For each lower part $s$ begin by defining witnesses $f^s$ such that for all $\vec{w}\in\dom h$, if there are $e$ and $q\leq h(\vec{w})$ such that $s\frown ((\vec{w},e,q))\in X$ then there is $q\leq h(\vec{w})$ such that $s\frown ((\vec{w},f^s(\vec{w}),q))\in X$.

  We may assume that each $h_i$ is of the form $b(H_i, B_i)$ for some $B_i\in u^*_i$. For each lower part $s$ and each $i<\lambda$ choose $H^s_i \leq H_i$ such that for all $\vec{w}^* \in \dom H^s_i$ for which there exists $q \leq H_i(\vec{w}^*)$ with $s\frown ((\pi(\vec{w}^*),f^s(\pi(\vec{w}^*)),q))\in X$ we have $s \frown ((\pi(\vec{w}^*),f^s(\pi(\vec{w}^*)),H^s_i(\vec{w}^*)))\in X$. By normality take $H'_i$ such that for all $i$ and $s$ we have $H'_i \leq H^s_i\downharpoonright s$. For each $i<\lambda$ and lower part $s$ we can choose $C^s_i \subseteq B_i$ a member of $u^*_i$ such that one of the following occurs:
  \begin{enumerate}[(i)]
    \item For every $\vec{w}^*$ in $C^s_i$ there does not exist a $q \leq H_i(\vec{w}^*)$ such that $s\frown ((\pi(\vec{w}^*),f^s(\pi(\vec{w}^*)),q)) \in X$.
    \item For every $\vec{w}^*$ in $C^s_i$ we have $s\frown ((\pi(\vec{w}^*),f^s(\pi(\vec{w}^*)),H'_i(\vec{w}^*)))\in X$.
  \end{enumerate}
  Define $C_i := \triangle_s C^s_i$ and $h'_i:=b(H'_i, C_i)$. Observe that by construction $\Fil(H_i)$ is an ultrafilter and so equal to $\Fil(H'_i)$, whence $h'_i\in\mathcal{F}_i$. We can now prove a weaker version of the desired dichotomy.

  \begin{claim}
    Let $s$ be a lower part and $i<\lambda$. Then we have one of:
      \begin{enumerate}[(i)]
        \item For all $\vec{w}\in \dom h'_i$ there do not exist $e$ and $q\leq h'(\vec{w})$ such that $s\frown((\vec{w},e,q)) \in X$.
        \item For all $\vec{w}\in \dom h'_i$ there are densely many $q$ in $\mathbb{B}(\vec{w},\vec{u})$ below $h'_i(\vec{w})$ such that $s\frown((\vec{w},f^s(\vec{w}),q)) \in X$.
      \end{enumerate}
  \end{claim}

  \begin{proof}
    Suppose (i) is false, so we have $\vec{w}\in\dom h'_i$, $e$ and $q' \leq h'(\vec{w})\leq h(\vec{w})$ such that $s\frown((\vec{w},e,q'))\in X$. The choice of $f^s$ then gives us $q\leq h(\vec{w})$ such that $s\frown((\vec{w},f^s(\vec{w}),q)) \in X$. Now
    $$q \leq h_i(\vec{w}) = b(H_i, B_i)(\vec{w}) = \bigvee_{\pi(\vec{w}^*)=\vec{w}, \vec{w}^*\in B_i} H_i(\vec{w}^*)$$
    so there must be some $\vec{w}^* \in B_i$ with $\pi(\vec{w}^*)=\vec{w}$ such that $q \parallel H_i(\vec{w}^*)$. But $\vec{w}\in\dom h'_i$ so $\vec{w}^*\in C_i\subseteq C^s_i\downharpoonright s$; and $X$ is downwards closed so we cannot have been in the first case when we defined $C^s_i$, and must therefore be in the second case.

    We wish to show that (ii) holds, so we are given some $\vec{w}\in \dom h'_i$ and $r \in \mathbb{B}(\vec{w},\vec{u})$ below $h'_i(\vec{w})$. By similar reasoning we have that $r$ is compatible with $H'_i(\vec{w}^*)$ for some $\vec{w}^* \in C_i$ with $\pi(\vec{w}^*)=\vec{w}$. By the definition of $C^s_i$ we know that $s\frown ((\vec{w},f^s(\vec{w}),H'_i(\vec{w}^*)))\in X$ so it is possible to take $q\leq r$ with $s\frown ((\vec{w},f^s(\vec{w}),q))\in X$.
  \end{proof}

  For each lower part $s$ and each $i<\lambda$ that falls into case (ii) of the claim, and for each $\vec{w}\in \dom h'_i$ we have a dense open set of $q \in \mathbb{B}(\vec{w},\vec{u})$ such that $s\frown ((\vec{w},f^s(\vec{w}),q))\in X$, and we take a maximal antichain contained in both this set and $\mathbb{C}(\vec{w},\vec{u})$. The $\kappa$-chain condition of the forcing tells us that this antichain is bounded, which is to say there is some $\eta_{s, i, \vec{w}} < \kappa$ with the antichain contained in $\mathbb{C}(\vec{w},\eta_{s, i, \vec{w}})$. We now refine $\dom h'$ to contain only $\vec{x}$ such that $\kappa(\vec{x})$ is a closure point of the function $(s, i, \vec{w}) \mapsto \eta_{s, i, \vec{w}}$ and immediately have the following strengthening of the claim.

  For all lower parts $s$ and $i < \lambda$ we have one of:
  \begin{enumerate}[(i)]
    \item For all $\vec{w}\in \dom h'_i$ there do not exist $e$ and $q\leq h'(\vec{w})$ such that $s\frown((\vec{w},e,q)) \in X$.
    \item For all $\vec{w}\in \dom h'_i$ and $\vec{x}\in\dom h'$ such that $\kappa(\vec{w})<\kappa(\vec{x})$ there are densely many $q$ in $\mathbb{B}(\vec{w},\vec{x})$ below $h'_i(\vec{w})$ such that $s\frown((\vec{w},f^s(\vec{w}),q)) \in X$.
  \end{enumerate}

  To conclude the proof we will need to make further reductions of the domains of the $h'_i$. For each lower part $s$ and each $i<k<\lambda$ define $d^s_{i,k}$ to be the function $j(f^s)(\vec{u}\upharpoonright k)$ restricted to lower parts of length $i$. We observe that $j(f^s)(\vec{u}\upharpoonright k)$ is an upper part for $\mathbb{R}_{\vec{u}\upharpoonright k}$ so $d^s_{i,k}$ will have domain in $u_i$ and is a partial function from $V_{\kappa}$ to $V_{\kappa}$. Thus for any $\vec{v}$ in its domain we have $j(d^s_{i,k})(\vec{v}) = j(d^s_{i,k})(j(\vec{v})) = j(d^s_{i,k}(\vec{v}))=d^s_{i,k}(\vec{v})$, giving us
  \begin{align*}
    & \forall\vec{v}\in\dom d^s_{i,k}: j(d^s_{i,k})(\vec{v}) = d^s_{i,k}(\vec{v}) = j(f^s)(\vec{u}\upharpoonright k)(\vec{v}) \\
    \Rightarrow& \forall_{u_k}\vec{w}: \forall\vec{v}\in\dom d^s_{i,k}\cap V_{\kappa(\vec{w})}: d^s_{i,k}(\vec{v}) = f^s(\vec{w})(\vec{v})
  \end{align*}
  Call this $u_k$-large set $X^s_{i,k}$ and take $h''^s\leq h'$ such that for all $j<\lambda$ we have $\dom h''^s_j\subseteq (\bigcap_{k>j}\dom d^s_{j,k})\cap(\bigcap_{i<j}X^s_{i, j})$. Also ensure $h''^s_i\leq j(f^s)(\vec{u}\upharpoonright k)$ for all $i<k<\lambda$; this is possible since all the functions involved are members of the $\kappa$-complete filter $\mathcal{F}_i$. Then by normality take $h''$ such that $h''\downharpoonright s\leq h''^s$ for all $s$. For any lower part $s$, $\vec{w}\in\dom h''$ and $\vec{v}\in\dom h''\upharpoonright\vec{w}$ above $s$ this gives
  $$h''(\vec{v})\leq h''^s(\vec{v})\leq j(f^s)(\vec{u}\upharpoonright (\lh\vec{w}))(\vec{v})=f^s(\vec{w})(\vec{v}).$$
  Hence $h''\upharpoonright\vec{w}\leq f^s(\vec{w})$ and since $X$ is $^*$-open we get $s\frown((\vec{w},h''\upharpoonright\vec{w},q))\in X$ for densely-many $q$ as required.
\end{proof}

This lemma is valuable because it allows us to express the crucial properties of $\vec{u}$ solely in terms of subsets of $V_{\kappa}$, rather than large supercompactness embeddings. When we perform the forcing iteration it will be possible to reflect these properties from the $\vec{u}$ that occurs at the end of the iteration to the $\vec{u}$ at earlier stages.

\section{Properties of the Radin forcing $\mathbb{R}_{\vec{u}}$} \label{properties_of_radin}

\subsection{The Prikry property}

\begin{proposition}
  Let $\vec{u}\in\mathcal{U}$. Then $\mathbb{R}_{\vec{u}}$ has the Prikry property.
\end{proposition}

\begin{proof}
  We are given some condition
  $$p = ((\vec{w}_0,h_0,p_0),...,(\vec{w}_{n-1},h_{n-1},p_{n-1}),(\vec{u},h_n))$$
  from $\mathbb{R}_{\vec{u}}$ and a proposition $\varphi$, and wish to find a direct extension of $p$ that forces either $\varphi$ or $\neg\varphi$. For notational convenience we will deem $\vec{w}_n$ to be $\vec{u}$.

  We define a descending sequence of $p^k$ by induction on $k\leq n$; starting with $p^0\leq^* p$ such that $p^0\parallel\phi$ if possible, or else $p^0:=p$. Given $p^{k-1}$, for each lower part $s\leq p^{k-1}\upharpoonright k$ for $\mathbb{R}_{\vec{w}_k}$, if possible take $((\vec{w}_k,h^s_k,p^s_k))\frown y^s_k \leq^* p^{k-1}\downharpoonright k$ such that $s\frown((\vec{w}_k,h^s_k,p^s_k))\frown y^s_k\parallel\phi$. Then by normality we can form $h'_k\leq h^s_k\downharpoonright s$ for all lower parts $s$, and by closure we can form $p'_k\leq p^s_k$ and $y'_k\leq y^s_k$ for all such $s$. Define $p^k := p^{k-1}\upharpoonright k\frown((\vec{w}_k,h'_k,p'_k))\frown y'_k$. The construction concludes with $p':=p^n$, so
  $$p' = ((\vec{w}_0,h'_0,p'_0),...,(\vec{w}_{n-1},h'_{n-1},p'_{n-1}),(\vec{w}_n,h'_n))$$
  such that for all $k\leq n$ and all lower parts $s \in \mathbb{R}_{\vec{w}_k}$, if there is some direct extension $t$ of $p'\downharpoonright k$ with $s\frown t\parallel\varphi$ then already $s\frown p'\downharpoonright k\parallel\varphi$. Then for each $k\leq n$ define $X^+_k$ to be the set of lower parts $s$ in $\mathbb{R}_{\vec{w}_k}$ such that $s\frown p'\downharpoonright k\forces\varphi$. Similarly define $X^-_k$ with $\neg\varphi$ in place of $\varphi$. Take $h''_k \leq h'_k$ that captures both $X^+_k$ and $X^-_k$.

  We claim that $$p'':=((\vec{w}_0,h''_0,p'_0),...,(\vec{w}_{n-1},h''_{n-1},p'_{n-1}),(\vec{w}_n,h''_n))$$
  decides $\varphi$. Suppose this is not so, and take $t$ of minimal length below $p''$ that decides $\varphi$; without loss of generality we can assume $t\forces\varphi$. Fix $k\leq n$ such that the largest new triple appearing in $t$ lies between $\vec{w}_{k-1}$ and $\vec{w}_k$. Call this triple $(\vec{v},e,q)$ and split $t$ as $r\frown((\vec{v},e,q))\frown s$. Observe that by the construction we actually have $r\frown((\vec{v},e,q))\frown p''\downharpoonright k \forces \varphi$. Observe further that the existence of such a $(\vec{v},e,q)$ tells us that $r$ and $\epsilon:=\lh\vec{v}$ fall into case (ii) of the capture of $X^+_k$. We will show by density that in fact $r\frown p''\downharpoonright k\forces\varphi$, which will contradict the minimality of the length of $t$ and conclude the proof.

  We are given some extension of $r\frown p''\downharpoonright k$, say of the form
  $$r'\frown ((\vec{v}_0,e_0,q_0),...,(\vec{v}_{m-1},e_{m-1},q_{m-1}))\frown y$$
  where $\kappa(\max r') = \kappa(\max r)$ and $\kappa(\min y) = \kappa(\vec{w}_k)$, and we seek an extension that forces $\varphi$. Fix $j$ such that $\lh\vec{v}_j = \epsilon$ and $\lh\vec{v}_i < \epsilon$ for all $i<j$; if there is no such $\vec{v}_j$ then we can easily insert one. We have $\vec{v}_j \in \dom h''_k$ and $q_j\leq h''_k(\vec{v}_j)$, and case (ii) of the capturing of $X^+_k$ occurs for $r$ and $\epsilon$, so we can find $q^*\leq q_j$ such that $r\frown((\vec{v}_j,h''_k\upharpoonright\vec{v}_j,q^*))\in X^+_k$, which is to say
  $$r\frown((\vec{v}_j,h''_k\upharpoonright\vec{v}_j,q^*))\frown p'\downharpoonright k\forces\varphi.$$
  For $i<j$ the fact that $(\vec{v}_i,e_i,q_i)$ could be added below $(\vec{w}_k,h''_k)$ shows us that it can also be added below $(\vec{v}_j,h''_k\upharpoonright\vec{v}_j)$. This establishes that
  $$r'\frown ((\vec{v}_0,e_0,q_0),... (\vec{v}_j,e_j\wedge h''_k\upharpoonright\vec{v}_j, q^*) ,...,(\vec{v}_{m-1},e_{m-1},q_{m-1}))\frown y$$
  is below $r\frown((\vec{v}_j,h''_k\upharpoonright\vec{v}_j,q^*))\frown p'\downharpoonright k$ and hence forces $\varphi$, and it is also an extension of $r'\frown ((\vec{v}_0,e_0,q_0),...,(\vec{v}_{m-1},e_{m-1},q_{m-1}))\frown y$ as required.
\end{proof}

We can use this result to show that $\mathbb{R}_{\vec{u}}$ preserves enough cardinals.

\begin{proposition} \label{preserveCardinals}
  \begin{enumerate}[(a)]
    \item Let $\vec{u}\in\mathcal{U}$ and $\langle\vec{w}_{\alpha},g_{\alpha}\mid\alpha<\theta\rangle$ the generic sequence of ultrafilter sequences and collapses added by $\mathbb{R}_{\vec{u}}$.

    Then for $\alpha<\theta$, $\mathbb{R}_{\vec{u}}$ preserves the cardinals in $[\kappa(\vec{w}_{\alpha}),\kappa(\vec{w}_{\alpha})^{+5}]$.
    \item If we force below $((\vec{u},h))$ such that $\dom h$ contains only sequences of length less than $\lh\vec{u}$ then $\kappa$ becomes $\aleph_{\lh\vec{u}}$.
  \end{enumerate}
\end{proposition}

\begin{proof}
  \begin{enumerate}[(a)]
    \item Our proof is by induction on $\kappa(\vec{u})$. Given $\alpha<\theta$ take a condition $p$ in the generic filter of the form $p_1\frown((\vec{w}_{\alpha},e,q))\frown p_2$. Below $p$, $\mathbb{R}_{\vec{u}}$ splits as
    $$\mathbb{R}_{\vec{w}_{\alpha}}/p_1\frown((\vec{w}_{\alpha},e)) \:\:\times\:\: \mathbb{R}'_{\vec{u}}/((\langle\kappa(\vec{w}_{\alpha})\rangle,\phi,q))\frown p_2$$
    where $\mathbb{R}'_{\vec{u}}$ is the same as $\mathbb{R}_{\vec{u}}$ except with its first collapse starting from $\kappa(\vec{w}_{\alpha})^{+5}$ instead of $\omega^{+5}$. By hypothesis the first of these forcings preserves many cardinals below $\kappa(\vec{w}_{\alpha})$ and hence $\kappa(\vec{w}_{\alpha})$ itself; since it has the $\kappa(\vec{w}_{\alpha})^+$-cc it also preserves all larger cardinals. The second forcing has the Prikry property by a proof identical to the one above, and it is $\kappa(\vec{w}_{\alpha})^{+5}$-closed in the $\leq^*$-ordering, so it will preserve all remaining cardinals up to and including $\kappa(\vec{w}_{\alpha})^{+5}$.
    \item Here we have $\theta=\omega^{\lh\vec{u}}=\lh\vec{u}$. By part (a) we have that cardinals in $[\kappa(\vec{w}_{\alpha}),\kappa(\vec{w}_{\alpha})^{+5}]$ are preserved, and it is clear that all other cardinals below $\kappa(\vec{u})$ are collapsed. So the forcing leaves $\lh\vec{u}$-many cardinals below $\kappa$.
  \end{enumerate}
\end{proof}

\subsection{Analysis of names}

Next we prove a technical lemma allowing us to replace $\mathbb{R}_{\vec{u}}$-names with names in smaller forcings; it will be useful to us in Lemma \ref{Qstrong} when we need to establish tight control over such names.

\begin{lemma} \label{nameAnalysis}
  Let $\vec{u}\in\mathcal{U}$, $\dot{x}$ a Boolean $\mathbb{R}_{\vec{u}}$-name (i.e. a name for a single true/false value), and $s\frown((\vec{u},h))\in\mathbb{R}_{\vec{u}}$.

  Then there is an ordinal $\beta<\kappa$, an upper part $h'\leq h$, and a $\mathbb{R}_{\max s}\times\mathbb{B}(\max s, \beta)$-name $\dot{y}$ such that $\dom h'$ lies above $\beta$ (so below $((\vec{u},h'))$ this $\dot{y}$ can be regarded as a $\mathbb{R}_{\vec{u}}$-name) and such that $s\frown((\vec{u},h'))\forces \dot{x}=\dot{y}$.
\end{lemma}

\begin{proof}
  For each $\beta>\kappa(\max s)$ and each $\mathbb{R}_{\max s}\times\mathbb{B}(\max s,\beta)$-name $\dot{y}$ use normality to take $h_{\dot{y}}\leq h$ such that for every lower part $t$, if there is some $h^*\leq h$ with $t\frown((\vec{u},h^*))\forces\dot{x}=\dot{y}$ then $t\frown((\vec{u},h_{\dot{y}}))\forces\dot{x}=\dot{y}$. Define $X_{\dot{y}}$ to be the set of lower parts $t$ such that
  $$t\frown((\vec{u},h_{\dot{y}}))\forces\dot{x}=\dot{y}$$
  and take $h'_{\dot{y}}\leq h_{\dot{y}}$ capturing $X_{\dot{y}}$. Then use normality again to get $h'$ such that $h'\downharpoonright (\beta+1)\leq h_{\dot{y}}$ for all $\mathbb{R}_{\max s}\times\mathbb{B}(\max s,\beta)$-names $\dot{y}$.

  Take $\vec{w}\in\dom h'$ above $s$ and of length $0$ (i.e. $\vec{w}=\langle\kappa(\vec{w})\rangle$). We can split $\mathbb{R}_{\vec{u}}$ below $s\frown((\vec{w},\phi,0),(\vec{u},h'))$ as
  $$\mathbb{R}_{\vec{w}}/s\frown((\vec{w},\phi))\times \mathbb{R}'_{\vec{u}}/((\vec{w},\phi, 0),(\vec{u},h'\downharpoonright \kappa(\vec{w})))$$
  where $\mathbb{R}'_{\vec{u}}$ is the usual forcing derived from $u$ except that its first collapse starts from $\kappa(\vec{w})^{+5}$ rather than $\omega^{+5}$. Now we can view $\dot{x}$ as being a $\mathbb{R}'_{\vec{u}}$-name for a $\mathbb{R}_{\vec{w}}$-name. We know that $\mathbb{R}_{\vec{w}}$ has the $\kappa(\vec{w})^+$-cc, so the $\mathbb{R}_{\vec{w}}$-name in question consists of at most $\kappa(\vec{w})$-many pieces of information. This allows us to use the Prikry property and closure of $\mathbb{R}'_{\vec{u}}$ to take a direct extension $((\langle\omega\rangle,\phi,q), (\vec{u},h''))$ of $((\langle\omega\rangle,\phi,0), (\vec{u},h'\downharpoonright \kappa(\vec{w})))$ that determines the value of the $\mathbb{R}_{\vec{w}}$-name, say as $\dot{y}$, a $\mathbb{R}_{\vec{w}}$-name in the ground model. So returning to $\mathbb{R}_{\vec{u}}$ we have
  $$s\frown((\vec{w},\phi,q),(\vec{u},h''))\forces\dot{x}=\dot{y}.$$

  Observe that (since $\lh\vec{w}=0$) $\mathbb{R}_{\vec{w}}$ splits below $s\frown((\vec{w},\phi))$ as $\mathbb{R}_{\max s}/s \times \mathbb{B}(\max s, \vec{w})$; this has the $\kappa(\vec{w})$-cc so $\dot{y}$ is in fact a $\mathbb{R}_{\max s}\times\mathbb{B}(\max s, \beta)$-name for some $\beta<\kappa(\vec{w})$. Now by construction $X_{\dot{y}}$ contains $s\frown((\vec{w},\phi,q))$, and $\vec{w}\in\dom h_{\dot{y}}$, so when we captured $X_{\dot{y}}$ we must have been in case (ii) for $s$ and $0$. We can now use the same argument as in the proof of the Prikry condition to show that any extension of $s$ must be extensible to some condition in $X_{\dot{y}}$, so we have that $s\frown((\vec{u},h'))\forces\dot{x}=\dot{y}$ as required.
\end{proof}

\subsection{Characterisation of genericity}

Finally we look for a way to characterise genericity that will allow us to take generic sequences for one Radin forcing and show that they are also generic for other Radin forcings. This characterisation develops similar ideas for simpler forcings found in \cite{Mathias} and \cite{Mitchell}.

\definition
Let $\vec{u}\in\mathcal{U}$. A sequence $\langle \vec{w}_{\alpha}, g_{\alpha} \mid \alpha < \theta\rangle$ in some outer model of set theory is {\em geometric} for $\mathbb{R}_{\vec{u}}$ if it satisfies:
\begin{enumerate}
  \item $\{ \kappa(\vec{w}_{\alpha}) \mid \alpha < \theta\}$ is club in $\kappa(\vec{u})$ with $\kappa(\vec{w}_0)=\omega$.
  \item For all limit ordinals $\alpha < \theta$, $\langle \vec{w}_{\beta}, g_{\beta} \mid \beta < \alpha\rangle$ is generic for $\mathbb{R}_{\vec{w}_{\alpha}}$.
  \item For all $\alpha$, $g_{\alpha}$ is $\mathbb{B}(\vec{w}_{\alpha},\vec{w}_{\alpha+1})$-generic.
  \item For every $X\in V_{\kappa(\vec{u})+1}$; $X \in \bigcap_{i<\lh\vec{u}}u_i$ iff for all large $\alpha$, $\vec{w}_{\alpha}\in X$.
  \item For every upper part $h$ for $\mathbb{R}_{\vec{u}}$, for all large $\alpha$, $h(\vec{w}_{\alpha})\in g_{\alpha}$.
\end{enumerate}

Note (4) implies that for all $i<\lh\vec{u}$ there are unboundedly many $\alpha<\theta$ such that $\lh(\vec{w}_{\alpha})=i$.

It is clear that a generic sequence for $\mathbb{R}_{\vec{u}}$ is geometric; we aim to show the converse.

\begin{definition} \label{defnGenericFilter}
  We have already seen that from a generic filter for $\mathbb{R}_{\vec{u}}$ it is possible to derive a generic sequence $G=\langle \vec{w}_{\alpha}, g_{\alpha} \mid \alpha < \theta\rangle$. Conversely, given such a generic sequence we can rederive the generic filter $F_G$, which will consist of all conditions $((\vec{v}_0,e_0,p_0),...,(\vec{v}_{n-1},e_{n-1},p_{n-1}),(\vec{u},h))$ with $(\vec{v}_n,e_n):=(\vec{u},h)$ such that:
  \begin{itemize}
    \item For all $k<n$ there is $\alpha$ such that $\vec{v}_k=\vec{w}_{\alpha}$ and  $p_k\in g_{\alpha}$
    \item For $\alpha$ and $k\leq n$, if $\kappa(\vec{v}_{k-1})<\kappa(\vec{w}_{\alpha})<\kappa(\vec{v}_k)$, then $\vec{w}_{\alpha}\in\dom e_k$ and $e_k(\vec{w}_{\alpha})\in g_{\alpha}$.
  \end{itemize}
\end{definition}

\begin{definition}
  We will say a sequence $\langle \vec{w}_{\alpha}, g_{\alpha} \mid \alpha < \theta\rangle$ {\em respects} $(\vec{u},h)$ if for all $\alpha < \theta$ we have $\vec{w}_{\alpha} \in \dom h$ and $h(\vec{w}_{\alpha})\in g_{\alpha}$.
\end{definition}

\begin{lemma} \label{geometricLemma}
  Let $\vec{u} \in \mathcal{U}$, $h$ an upper part for $\mathbb{R}_{\vec{u}}$ and $D\subseteq \mathbb{R}_{\vec{u}}$ dense open. Then there is an upper part $h' \leq h$ such that for every geometric sequence $G$ respecting $(\vec{u},h')$ we have $F_G\cap D \neq \phi$.
\end{lemma}

\begin{proof}
  Define $\kappa:=\kappa(\vec{u})$ and $\lambda:=\lh\vec{u}$. Invoke normality to take $h^*\leq h$ such that for every lower part $t$, if there is an $h'$ such that $t\frown((\vec{u},h'))\in D$ then $t\frown((\vec{u},h^*\downharpoonright t))\in D$.  Define $X^{\phi}$ to be the set of lower parts $t$ such that $t\frown((\vec{u},h^*))\in D$. We will inductively define $X^{\eta}$ for $\eta$ any finite sequence of $i<\lambda$. First take $h^{\eta}\leq h^*$ capturing $X^{\eta}$. Then for each $i<\lambda$, the set $X^{\langle i\rangle\frown\eta}$ will consist of all $t$ that with $i$ fall into case (ii) of the capturing of $X^{\eta}$. The number of possible $\eta$ is less than $\kappa$ so by $\kappa$-completeness we can fix $h'$ that is below $h^{\eta}$ for all such $\eta$.

  We note that there are densely-many $r$ in $\mathbb{B}(\omega,\kappa)$ such that $((\langle\omega\rangle,\phi,r))\in X^{\eta}$ for some $\eta$. This is because for any $r\in\mathbb{B}(\omega,\kappa)$ we can extend $((\langle\omega\rangle,\phi,r),(\vec{u},h'))$ to a condition $t\frown((\vec{u},h^*)) \in D$; then $t\in X^{\phi}$ and inductively removing triples of $t$ from the right yields $((\langle\omega\rangle,\phi,r'))\in X^{\eta}$ for some $r'\leq r$ and $\eta$.

  Now we are given a geometric sequence $G=\langle \vec{w}_{\alpha}, g_{\alpha} \mid \alpha < \theta\rangle$ that respects $(\vec{u},h')$ and must show $F_G\cap D \neq \phi$.

  By the density in $\mathbb{B}(\omega,\kappa)$ just noted, and by the genericity of $g_0$, take $q_0\in g_0$ and $\eta$ such that $((\langle\omega\rangle,\phi,q_0))\in X^{\eta}$. Say $\eta=:\langle i_1,...,i_{n-1}\rangle$. Define $\alpha_0=0$ and then inductively take $\alpha_{k+1}>\alpha_k$ minimal such that $\lh\vec{w}_{\alpha_{k+1}}=i_k$. We note that this must be possible by clause (4) in the definition of geometricity. Then for $k\geq 1$ inductively choose $q_k\in g_{\alpha_k}$ such that
  $$s_k:=((\vec{w}_0,h'\upharpoonright\vec{w}_0,q_0),...,(\vec{w}_{\alpha_k},h'\upharpoonright\vec{w}_{\alpha_k},q_k))\in X^{\langle i_{k+1},...i_{n-1}\rangle},$$
  invoking the nature of case (ii) capturing, the genericity of the $g_k$, and the fact $h'(\vec{w}_{\alpha_k})\in g_{\alpha_k}$. This concludes with $s_{n-1}\in X^{\phi}$ so $s_{n-1}\frown((\vec{u},h'))\in D$ (as $h'\leq h^*$), and it remains to show that $s_{n-1}\frown((\vec{u},h'))\in F_G$.

  The first of the two requirements from Definition \ref{defnGenericFilter} is clear from the construction itself. Now we must consider the case of some $\beta < \theta$ and $k\leq n$ such that $\kappa(\vec{w}_{\alpha_{k-1}}) < \kappa(\vec{w}_{\beta}) < \kappa(\vec{w}_{\alpha_k})$, or equivalently $\alpha_{k-1}<\beta<\alpha_k$. Note that the $k=n$ case is taken care of by the respect of $G$ for $(\vec{u},h')$. The minimality of our choice of $\alpha_k$ (together with clauses 2 and 4 in the definition of geometric) tells us that $\lh\vec{w}_{\beta}<\lh\vec{w}_{\alpha_k}$, so $\vec{w}_{\beta}\in\dom h'\upharpoonright\vec{w}_{\alpha_k}$. Then from the respect of $G$ for $(\vec{u},h')$ we have $h'(\vec{w}_{\beta})\in g_{\beta}$ as required.
\end{proof}

\begin{proposition} \label{characteriseGenericity}
  Let $\vec{u}\in\mathcal{U}$. Then a sequence $G$ is generic for $\mathbb{R}_{\vec{u}}$ iff it is geometric for $\mathbb{R}_{\vec{u}}$.
\end{proposition}

\begin{proof}
  We are given $D\subseteq\mathbb{R}_{\vec{u}}$ dense open and wish to show that $D\cap F_G\neq\phi$. We begin by using normality to take an upper part $h$ such that for all lower parts $s$, if there is some $\tilde{h}$ such that $s\frown((\vec{u},\tilde{h}))\in D$ then already $s\frown((\vec{u},h\downharpoonright s))\in D$. For each lower part $s$, use the Prikry property for $\mathbb{R}_{\vec{u}}$ to take $h_s\leq h\downharpoonright s$ such that
  $$((\vec{u},h_s))\parallel\exists t\in \dot{\Gamma}: s\frown t\frown ((\vec{u},h))\in D$$
  where $\dot{\Gamma}$ is the name for the set of all lower parts for $\mathbb{R}_{\vec{u}}$ that appear as the lower part of some condition in the generic filter. We say $s$ is {\em good} if the decision is positive. For good $s$ we have $$D_s:=\{t\frown((\vec{u},h^*))\mid s\frown t\frown((\vec{u},h^*))\in D\}$$
  dense open below $((\vec{u},h_s))$. For these $s$ use Lemma \ref{geometricLemma} to take $h'_s\leq h_s$ such that for all geometric $G$ respecting $(\vec{u},h'_s)$ we have $F_G\cap D \neq\phi$. Then by normality take $h'\leq h\downharpoonright s$ for all $s$, and note that also
  $$((\vec{u},h'))\parallel \exists t\in \dot{\Gamma}: s\frown t\frown ((\vec{u},h))\in D$$
  for all lower parts $s$. Finally take $h''\leq h'$ that captures the set of good lower parts; we say that a lower part $r$ that falls into case (ii) of this capturing is {\em pre-good}.

  Write the geometric $G$ we were given as $\langle \vec{w}_{\alpha}, g_{\alpha} \mid \alpha < \theta\rangle$, and use clause 5 of geometricity to take $\beta<\theta$ a limit ordinal such that
  $$\forall \alpha\geq \beta: \vec{w}_{\alpha}\in\dom h'', h''(\vec{w}_{\alpha})\in g_{\alpha}.$$
  Then take $\gamma>\beta$ also limit such that for all $\beta<\alpha<\gamma$, $\lh\vec{w}_{\alpha}<\lh\vec{w}_{\gamma}$; this is possible as $\cf\theta>\omega$ and all lengths occur cofinally in $\langle\vec{w}_{\alpha}\mid\alpha<\theta\rangle$.

  \begin{claim}
    There are densely-many $r\frown((\vec{w}_{\gamma},e))$ in $\mathbb{R}_{\vec{w}_{\gamma}}$ such that $r$ is pre-good.
  \end{claim}

  \begin{proof}
    We are given a condition $r\frown((\vec{w}_{\gamma},e))\in\mathbb{R}_{\vec{w}_{\gamma}}$ and have that $r\frown((\vec{w}_{\gamma},e,0),(\vec{u},h'')) \in \mathbb{R}_{\vec{u}}$ so we can extend it to
    $$r'\frown((\vec{w}_{\gamma},e',q))\frown t\frown((\vec{u},h''))\in D,$$
    and the choice of $h$ then gives
    $$r'\frown((\vec{w}_{\gamma},e',q))\frown t\frown((\vec{u},h))\in D.$$
    The decision made by $((\vec{u},h'))$ for $r'\frown((\vec{w}_{\gamma},e',q))$ must thus have been positive, which is to say $r'\frown((\vec{w}_{\gamma},e',q))$ is good. The fact that $\vec{w}_{\gamma}\in\dom h''$ and $q\leq h''(\vec{w})$ then gives that $r'$ is pre-good; we are now done because $r'\frown((\vec{w}_{\gamma},e'))\leq r\frown((\vec{w}_{\gamma},e))$ in $\mathbb{R}_{\vec{w}_{\gamma}}$.
  \end{proof}

  The claim allows us to use property 2 of $G$ to find some pre-good $r$ and $e$ with $r\frown((\vec{w}_{\gamma},e))\in F_{G\upharpoonright\gamma}$ and $\kappa(\max(r))>\kappa(\beta)$. Then we use case (ii) of capturing to take $p\in\mathbb{B}(\vec{w}_{\gamma},\vec{w}_{\gamma+1})$ such that $p\in g_{\gamma}$ and $s:=r\frown((\vec{w}_{\gamma},h''\upharpoonright\vec{w}_{\gamma},p))$ is good. For all $\alpha<\gamma$ with $\kappa(\vec{w}_{\alpha})$ above $r$ we have $\alpha>\beta$, so $\lh\vec{w}_{\alpha}<\lh\vec{w}_{\gamma}$ and $\vec{w}_{\alpha}\in\dom h''\upharpoonright\vec{w}_{\gamma}$; also $h''(\vec{w}_{\alpha})\in g_{\alpha}$ by the choice of $\beta$. Combining these shows us that $s\frown((\vec{w}_{\gamma+1},\phi))\in F_{G\upharpoonright(\gamma+1)}$.

  Now $D_s$ is dense and $G\downharpoonright(\gamma+1)$ respects $h'$ so we can take $t\frown((\vec{u},h^*))\leq((\vec{u},h'))$ in $F_{G\downharpoonright(\gamma+1)}\cap D_s$. Then by the definitions of $F_G$ and $D_s$ we have $s\frown t\frown((\vec{u},h^*))\in F_G\cap D$, and are done.
\end{proof}

\section{The preparatory forcing $\mathbb{Q}_{\vec{u}}$} \label{preparatory_forcing}

In this section we work in the following context.

\begin{setting}
  Let $\vec{u}\in\mathcal{U}$ with $\kappa:=\kappa(\vec{u})$ and $\lambda:=\lh\vec{u}$ regular uncountable. Let $\kappa^{<\kappa}=\kappa$ and $2^{\kappa^+}=\kappa^{+3}$. Assume there exists a binary tree $T$ of height and size $\kappa^+$ (i.e. a tree such that each node has two successors on the next level) with $\langle x_{\alpha}\mid \alpha<\kappa^{+3}\rangle$ an enumeration of its branches. Let $\langle \dot{\mathcal{E}}_{\alpha}\mid \alpha<\kappa^{+3}\rangle$ be an enumeration of the $\mathbb{R}_{\vec{u}}$-names for graphs on $\kappa^+$. We note that such an enumeration is possible since $\mathbb{R}_{\vec{u}}$ has size $2^{\kappa}$ and the $\kappa^+$-cc.
\end{setting}

\subsection{Defining the forcing $\mathbb{Q}_{\vec{u}}$}

We will want to perform an iteration that preserves $V_{\kappa}$ and successively expands $V_{\kappa+1}$ and thus the sequences of ultrafilters. With this in mind we consider a member $\vec{u}$ of $\mathcal{U}$, and seek to add a partial function $g$ from $V_{\kappa}$ to $V_{\kappa}$ such that defining $g_i := g\upharpoonright\{\vec{w}\in\mathcal{U}\mid\lh \vec{w} = i\}$ we could potentially expand $\vec{u}$ to some $\vec{u}'$ in the generic extension with $g_i\in \mathcal{F}_{\vec{u}',i}$. In order to accomplish this we will need the $g$ we build to be appropriately compatible with the pre-existing members of $\vec{u}$, motivating the following definition which generalises long Prikry forcing to the case of Radin forcing with collapses.

\begin{definition}
  Let $\vec{u}\in\mathcal{U}$. Then $\mathbb{M}_{\vec{u}}$ is defined to have conditions $(c,h)$ where $h$ is an upper part for $\vec{u}$ and there is $\rho^{(c,h)}:=\rho<\kappa$ such that $c$ is a partial function from $\mathcal{U}\cap V_{\rho}$ to $V_{\kappa}$ such that
  $$\forall \vec{v}\in \dom c: c(\vec{v})\in\mathbb{B}(\vec{v},\vec{\kappa})-\{0\}, c\upharpoonright\kappa(\vec{v}) \in \mathcal{F}_{\vec{v}}.$$
  We also require $\kappa(\vec{v})<\kappa(\vec{w})$ for $\vec{v}\in\dom c$ and $\vec{w}\in\dom h$.

  We define $(c',h')\leq(c,h)$ if $c'\upharpoonright\rho^{(c,h)}=c$, $h'\leq h$ and for each $\vec{w}\in\dom c' - \dom c$ we have $\vec{w}\in\dom h$ and $c(\vec{w})\leq h(\vec{w})$.

  Also define $a^{(c,h)}:=\{\kappa(\vec{w})\mid \vec{w}\in\dom c\}$.
\end{definition}

For any $(c,h)\in\mathbb{M}_{\vec{u}}$ and $\vec{w}\in\dom h$ the definition of upper part gives us that $(c\cup h\upharpoonright(\kappa(\vec{w})+1), h\downharpoonright(\kappa(\vec{w})+1))\leq(c,h)$ so by density $\mathbb{M}_{\vec{u}}$ will add a partial function $g$ from $V_{\kappa}$ to $V_{\kappa}$ such that for all upper parts $h$ there is a $\mu<\kappa$ with $g\downharpoonright\mu\leq h$.

We now augment this definition into one that will help us add a family of universal graphs together with functions witnessing their universality.

\begin{definition}
  Let $\vec{u}\in\mathcal{U}$. Then $\mathbb{Q}^*_{\vec{u}}$ has conditions $p=(c,h,t,f)$ such that:
  \begin{enumerate}
    \item $(c,h)\in\mathbb{M}_{\vec{u}}$. We define $a^p$ to be $a^{(c,h)}$.
    \item $t\in[(a^p\cap\sup a^p)\times \kappa^{+3}]^{<\kappa}$.
    \item $f =: \langle f^{\eta}_{\alpha}\mid (\eta,\alpha)\in t\rangle$ with $\dom f^{\eta}_{\alpha}\in [\kappa^+]^{<\kappa}$.
    \item For each $(\eta,\alpha)\in t$ and $\zeta\in\dom f^{\eta}_{\alpha}$ there is $\gamma < \kappa$ with $f^{\eta}_{\alpha}(\zeta) = (x_{\alpha}\upharpoonright\zeta, \gamma)$.
  \end{enumerate}
  We also write $t^{\eta}:=\{\alpha \mid (\eta,\alpha)\in t\}$.

  We define $(c',h',t',f')\leq(c,h,t,f)$ if $(c',h')\leq(c,h)$ in $\mathbb{M}_{\vec{u}}$, $t'\supseteq t$, and for all $(\eta, \alpha)\in t$ we have $f'^{\eta}_{\alpha}\supseteq f^{\eta}_{\alpha}$.

  Note that this definition is implicitly dependent on the $\langle x_{\alpha}\mid\alpha<\kappa^{+3}\rangle$ and $\langle\dot{\mathcal{E}}_{\alpha}\mid\alpha<\kappa^{+3}\rangle$ from the setting.
\end{definition}

In addition to the function $g$ added by $\mathbb{M}_{\vec{u}}$, this forcing will for each $\vec{w}\in \dom g$ and $\alpha<\kappa^{+3}$ add a function from $\kappa^+$ to $T\times\kappa$, the first co-ordinate of which will run along the branch $x_{\alpha}$. The idea here is that after Radin forcing the $\mathbb{R}_{\vec{u}}$-name $\dot{\mathcal{E}}_{\alpha}$ will be realised as a graph on $\kappa^+$ and then (for some $\eta$ to be selected later) the function $f^{\eta}_{\alpha}$ will map it into $T\times\kappa$. We will then include in our list of jointly-universal graphs the graph on $T\times\kappa$ induced by all these embeddings. This raises the problem that there may be disagreements between the many graphs we are trying to simultaneously embed as to whether or not a particular edge should exist. In order to gain better control of the situation we will add a fifth requirement on forcing conditions, and for this we need a technical definition.

\begin{definition}
  Let $s=\langle (\vec{w}_k,e_k,q_k)\mid k<n\rangle$ be a lower part for $\mathbb{R}_{\vec{u}}$, $c$ the first co-ordinate of a condition from $\mathbb{M}_{\vec{u}}$ and $\eta<\kappa$. Then we say $s$ is {\em harmonious with $c$ past $\eta$} if for all $k<n$ we have one of:
  \begin{itemize}
    \item $\kappa(\vec{w}_k)<\eta$.
    \item $\kappa(\vec{w}_k)=\eta$ and $\lh\vec{w}_k = 0$.
    \item $\kappa(\vec{w}_k)>\eta$, $e_k\leq c\upharpoonright\kappa(\vec{w}_k)$, $\kappa(\vec{v})>\eta$ for all $\vec{v}\in\dom e_k$, $\vec{w}_k\in\dom c$ and $q_k\leq c(\vec{w}_k)$.
  \end{itemize}
\end{definition}

\begin{lemma} \label{openHarmony}
  Let $s$ be harmonious with $c$ past $\eta$ and $s'\leq s$ (so $\kappa(\max s')=\kappa(\max s)$). Then $s'$ is harmonious with $c$ past $\eta$.
\end{lemma}

\begin{proof}
  Consider some element $(\vec{v},d,p)$ from $s'$. If $\vec{v}$ already occurs in $s$ then it is clear that is satisfies the conditions. Otherwise it was added below some element $(\vec{w},e,q)$ from $s$ (because $\kappa(\max s') = \kappa(\max s)$). If $\kappa(\vec{w})\leq\eta$ then $\kappa(\vec{v})<\eta$ so all is well. Otherwise since $\vec{v}\in\dom e$ we get $\kappa(\vec{v})>\eta$. The required conditions in this case follow since $d\leq e\upharpoonright V_{\kappa(\vec{v})}$ and $p\leq e(\vec{v})$.
\end{proof}

We are now ready to define the desired forcing.

\begin{definition}
  The forcing $\mathbb{Q}_{\vec{u}}$ consists of conditions $(c,h,t,f)$ that satisfy the four conditions from the definition of $\mathbb{Q}^*_{\vec{u}}$, which for convenience we repeat here, together with one more:
  \begin{enumerate}
    \item $(c,h)\in\mathbb{M}_{\vec{u}}$.
    \item $t\in[(a\cap\sup a)\times \kappa^{+3}]^{<\kappa}$ where $a:=a^{(c,h)}$.
    \item $f =: \langle f^{\eta}_{\alpha}\mid (\eta,\alpha)\in t\rangle$ with $\dom f^{\eta}_{\alpha}\in [\kappa^+]^{<\kappa}$.
    \item For each relevant $\eta$, $\alpha$ and $\zeta$ there is $\gamma < \kappa$ with $f^{\eta}_{\alpha}(\zeta) = (x_{\alpha}\upharpoonright\zeta, \gamma)$.
    \item Let $\eta\in a\cap\sup a$, $\alpha,\beta\in t^{\eta}$, $s$ a lower part for $\mathbb{R}_{\vec{u}}$ that is harmonious with $c$ past $\eta$, and $\zeta,\zeta'\in\dom f^{\eta}_{\alpha}\cap \dom f^{\eta}_{\beta}$. Let also $f^{\eta}_{\alpha}(\zeta) = f^{\eta}_{\beta}(\zeta)\neq f^{\eta}_{\alpha}(\zeta')=f^{\eta}_{\beta}(\zeta')$. Then
      $$s\frown((\vec{u},h))\forces_{\mathbb{R}_{\vec{u}}} \zeta \dot{\mathcal{E}}_{\alpha}\zeta' \leftrightarrow \zeta\dot{\mathcal{E}}_{\beta}\zeta'.$$
  \end{enumerate}
\end{definition}

We will be able to use this final condition at the end of the argument to ensure that graphs (given by the $\dot{\mathcal{E}}_{\alpha}$) that we wish to map to the same place will agree about which edges should exist. But first we must establish that the right kind of generic object is still added. In doing so we establish a slightly stronger result that tidies up the conditions and will aid some of our later reasoning.

\begin{lemma} \label{squareOff}
  Let $\vec{u}\in\mathcal{U}$, $l$ an upper part for $\mathbb{R}_{\vec{u}}$, $\eta<\mu<\kappa$, $\epsilon,\epsilon'<\kappa^{+3}$ and $\zeta,\zeta'<\kappa^+$. Let $p \in \mathbb{Q}_{\vec{u}}$ with $\eta \in a^p$. Then there are densely many conditions $q=(c,h,t,f)$ below $p$ such that $a^q$ has a maximal element greater than $\mu$, $h\leq l$, and there are $A\in[\kappa^{+3}]^{<\kappa}$ and $B\in[\kappa^+]^{<\kappa}$ with $t = (a^q\cap\sup a^q)\times A\ni(\eta,\epsilon),(\eta,\epsilon')$ and $\dom f^{\theta}_{\beta} = B\ni\zeta,\zeta'$ for each $(\theta, \beta)\in t$.
\end{lemma}

\begin{proof}
  We are given some condition $r=(c,h,t,f)\leq p$ to extend. Choose some $\vec{w}\in\dom h$ such that $\kappa(\vec{w})>\mu$ and define $c' = c\cup h\upharpoonright (\kappa(\vec{w})+1)$ and take $h'\leq l$, $h\downharpoonright\vec{w}$. Then $(c',h')$ will be in $\mathbb{M}_{\vec{u}}$ by the definition of upper part, and $a^{(c',h')}$ will have a maximum element $\kappa(\vec{w})$ as required.

  Now we wish to add new points to $t$ and the domains of the $f$ functions, in order to ensure they contain the required co-ordinates and are ``squared off'' as specified. There are $<\kappa$-many new points needed so we can choose values for the second co-ordinates of the $f^{\theta}_{\beta}(\tau)$ that are all distinct both from pre-existing values and each other. This will avoid creating any new instances of the fifth clause of the definition of $\mathbb{Q}_{\vec{u}}$. Call the resulting condition $q=(c',h',t',f')$.

  Suppose we are given $\theta \in a^{r'}\cap\sup a^{r'}$, $\alpha,\beta \in A$, $s$ a lower part for $\mathbb{R}_{\vec{u}}$ harmonious with $c'$ past $\theta$, and $\tau,\tau'\in B$ such that $f'^{\theta}_{\alpha}(\tau) = f'^{\theta}_{\beta}(\tau)\neq f'^{\theta}_{\alpha}(\tau')=f'^{\theta}_{\beta}(\tau')$. We want
    $$s\frown((\vec{u},h'))\forces \tau\dot{\mathcal{E}}_{\alpha}\tau'\leftrightarrow\tau\dot{\mathcal{E}}_{\beta}\tau'.$$
  Note by the construction of $f'$ that we must have $\theta\in a^r\cap\sup a^r$ with $\alpha,\beta\in t^{\theta}$ and $\tau,\tau'\in\dom f^{\theta}_{\alpha}\cap \dom f^{\theta}_{\beta}$ for these equalities to be possible. We would like to split $s$ as $s_1\frown s_2$ such that $s_1$ is harmonious with $c$ past $\theta$ and $s\frown((\vec{u},h'))\leq s_1\frown((\vec{u},h))$. Unfortunately this may not be possible because the smallest triple of $s_1$ may have a second co-ordinate that includes entries from $c$. So instead we show that $s\frown((\vec{u},h'))$ forces the required statement by a density argument.

  Given any $s^*\frown((\vec{u},h''))\leq s\frown((\vec{u},h'))$ split $s^*$ as $s^*_1\frown s^*_2$ such that $\kappa(\max s^*_1)<\ssup a^r$ and $\kappa(\min s^*_2)\geq\ssup a^r$. By Lemma \ref{openHarmony} we have that $s^*_1\frown s^*_2$ is harmonious with $c'$ past $\theta$, and so also $s^*_1$ is harmonious with $c$ past $\theta$. By the conditionhood of $r$ this gives
  $$s^*_1\frown((\vec{u},h))\forces \tau\dot{\mathcal{E}}_{\alpha}\tau'\leftrightarrow\tau\dot{\mathcal{E}}_{\beta}\tau'.$$
  Strengthen $s^*_2$ to $s^{**}_2$ by shrinking the second co-ordinate of $\min s^*_2$ as necessary to ensure that it lies above $\sup a^r$, so that $s^{**}_2$ can be added below $h$. Defining $s^{**}:=s^*_1\frown s^{**}_2$ this gives us $s^{**}\frown((\vec{u},h''))\leq s^*_1\frown((\vec{u},h))$, so $s^{**}\frown((\vec{u},h''))$ forces the desideratum. We also have $s^{**}\frown((\vec{u},h''))\leq s^*\frown((\vec{u},h''))$ so we are done.
\end{proof}

Being able to perform the argument above is a reason for the second co-ordinate in the definition of the $f^{\eta}_{\alpha}$. From another perspective the second co-ordinate gives us a greater degree of flexibility in our embeddings into the jointly-universal graphs.

\subsection{Properties of $\mathbb{Q}_{\vec{u}}$}

We now prove properties of the $\mathbb{Q}_{\vec{u}}$-forcings that will be valuable when we come to iterate them. First we recall some definitions.

\begin{definition}
  A subset $X$ of a forcing $\mathbb{P}$ is {\em centred} if every finite subset of $X$ has a lower bound. The forcing $\mathbb{P}$ is {\em $\kappa$-compact} if every centred subset of size less than $\kappa$ has a lower bound.
\end{definition}

Note that $\kappa$-compactness implies $\kappa$-directed closure.

\begin{lemma} \label{Qcompact}
  The forcing $\mathbb{Q}_{\vec{u}}$ is $\kappa$-compact.
\end{lemma}

\begin{proof}
  We are given some $X\subseteq\mathbb{Q}_{\vec{u}}$ with $|X|<\kappa$. For each finite subset $x$ of $X$, take a lower bound $(c^x,h^x,t^x,f^x)$ for $x$. It is clear that we can take some $h^*$ that is below $h^x$ for all such $x$ (using the $\kappa$-completeness of $\mathcal{F}_{\vec{u}}$). Also  form $c^*$, $t^*$ and $f^*$ by unions of all the individual $c$, $t$ and $f$ from conditions in $X$. Note we do not use the $c^x$, $t^x$ and $f^x$ as these are not guaranteed to be compatible.

  We can see that $(c^*,h^*,t^*,f^*)$ satisfies the requirements for being a member of $\mathbb{Q}_{\vec{u}}$ except possibly the fifth one. Suppose we are given $\eta,\alpha,\beta,\zeta$ and $\zeta'$ together with $s$ harmonious with $c^*$ past $\eta$ such that $f^{*,\eta}_{\alpha}(\zeta) = f^{*,\eta}_{\beta}(\zeta)\neq f^{*,\eta}_{\alpha}(\zeta')=f^{*,\eta}_{\alpha}(\zeta')$. Then choose a finite set $x \subseteq X$ that contains conditions $(c,h,t,f)$ which between them witness all of the following properties:
  \begin{itemize}
    \item $s$ is harmonious with $c$ past $\eta$.
    \item $(\eta,\alpha)\in t$ and $\zeta\in\dom f^{\eta}_{\alpha}$.
    \item $(\eta,\alpha)\in t$ and $\zeta'\in\dom f^{\eta}_{\alpha}$.
    \item $(\eta,\beta)\in t$ and $\zeta\in\dom f^{\eta}_{\beta}$.
    \item $(\eta,\beta)\in t$ and $\zeta'\in\dom f^{\eta}_{\beta}$.
  \end{itemize}
  Now $(c^x,h^x,t^x,f^x)$ is a condition in $\mathbb{Q}_{\vec{u}}$ so
  $$s\frown((\vec{u},h^x))\forces\zeta\dot{\mathcal{E}}_{\alpha}\zeta'\leftrightarrow\zeta\dot{\mathcal{E}}_{\beta}\zeta'$$
  and we ensured $h^*\leq h^x$ so $s\frown((\vec{u},h^*))$ will force the same thing.
\end{proof}

\begin{definition}
  The forcing $\mathbb{P}$ has the {\em strong $\kappa^+$-chain condition} if for every sequence $\langle p_{\alpha}\mid \alpha<\kappa^+\rangle$ from $\mathbb{P}$ there are a club $C\subseteq\kappa^+$ and a regressive function $f:(C\cap\cof\kappa)\rightarrow\kappa^+$ such that for all $\alpha,\beta\in C\cap\cof\kappa$ if $f(\alpha)=f(\beta)$ then $p_{\alpha}$ and $p_{\beta}$ are compatible.
\end{definition}

Note that by Fodor's theorem this property immediately implies the usual $\kappa^+$-chain condition.

\begin{lemma} \label{Qstrong}
  The forcing $\mathbb{Q}_{\vec{u}}$ has the strong $\kappa^+$-chain condition.
\end{lemma}

\begin{proof}
  We are given $\langle p_i\mid i<\kappa^+\rangle$ a sequence of conditions from $\mathbb{Q}_{\vec{u}}$. Define $(c^i,h^i,t^i,f^i):=p^i$, $\langle f^{i,\eta}_{\alpha}\mid(\eta,\alpha)\in t^i\rangle:=f^i$ and $a^i:=a^{p^i}$. We use our ability to extend conditions as in lemma \ref{squareOff} to assume there are $A^i$ and $B^i$ such that $t^i = (a^i\cap\sup a^i)\times A^i$ and for all $(\eta,\alpha)\in t^i$ that $\dom f^{i,\eta}_{\alpha}=B^i$.

  For each $i<\kappa^+$, $\alpha\in A^i$, $\zeta,\zeta'\in B^i$ and lower part $s$ we have ``$\zeta \dot{\mathcal{E}}_{\alpha}\zeta'$'' a binary name, so by Lemma \ref{nameAnalysis} we can take $h'\leq h^i$ and a $\mathbb{R}_{\max s}\times\mathbb{B}(\max s, \gamma)$-name $\dot{y}^i_{\alpha,s}(\zeta,\zeta')$ for some $\gamma<\kappa$ with $s\frown((\vec{u},h'))\forces \zeta \dot{\mathcal{E}}_{\alpha}\zeta'\leftrightarrow \dot{y}^i_{\alpha,s}(\zeta,\zeta')$. For each $i$ we will use the $\kappa$-closure of upper parts to assume, by shrinking as necessary, that $s\frown((\vec{u},h^i))$ forces this for all such $\alpha$, $\zeta$ and $\zeta'$ and for all lower parts $s$ with $\kappa(\max s)\leq\sup a^i$ and the third co-ordinate of $\max s$ equal to zero.

  Enumerate $\bigcup_{i<\kappa^+}A^i\subseteq\kappa^{+3}$ as $\{\beta(j) \mid j<\kappa^+\}$, and for each $i<\kappa^+$ enumerate $R^i:=\{j<\kappa^+\mid \beta(j)\in A^i\}$ in increasing order as $\{j^i_{\epsilon}\mid\epsilon<\mu^i\}$ for some $\mu^i<\kappa$. Fix $\{ t(k)\mid k<\kappa^+\}$ an enumeration of points in $T$ (the tree from which the branches $x_{\alpha}$ come) and define $T^i$ to be the set of $k<\kappa^+$ such that $f^{i,\eta}_{\alpha}(\zeta)=(t(k),\nu)$ for some $\eta$, $\alpha$, $\zeta$ and $\nu$. Construct functions as follows:
  \begin{itemize}
    \item $F_1(i)=(c^i,\mu^i, R^i\cap i, B^i\cap i,T^i\cap i)$.
    \item $F_2(i)$ is the set of tuples $(\eta, \epsilon,\zeta, k,\nu)$ such that $\eta\in a^i\cap\sup a^i$, $\epsilon<\mu^i$, $\zeta<i$, $k<i$ and $f^{i,\eta}_{\beta(j^i_{\epsilon})}(\zeta)=(t(k),\nu)$.
    \item $F_3(i)$ is the set of tuples $(\epsilon, s, \zeta,\zeta', \dot{y}^i_{\beta(j^i_{\epsilon}),s}(\zeta,\zeta'))$ for $\epsilon<\mu^i$, $s$ a lower part of the form described above, and $\zeta,\zeta'\in B^i\cap i$.
  \end{itemize}
  Then we define $F(i)=(F_1(i),F_2(i),F_3(i))$. Note that $F(i)$ will be a member of
    $$(V_{\kappa}^2\times([i]^{<\kappa})^3)\times(\kappa^2\times i^2\times\kappa)^{<\kappa}\times(\kappa\times V_{\kappa}\times i^2\times V_{\kappa})^{<\kappa}.$$
  Fix an injection $G$ from
    $$(V_{\kappa}^2\times([\kappa^+]^{<\kappa})^3)\times(\kappa^2\times (\kappa^+)^2\times\kappa)^{<\kappa}\times(\kappa\times V_{\kappa}\times (\kappa^+)^2\times V_{\kappa})^{<\kappa}$$
  to $\kappa^+$. We have $\kappa^{<\kappa}=\kappa$ which implies $(\kappa^+)^{<\kappa}=\kappa^+$ so we can find a club $C_0$ such that for all points $i$ in $C_0\cap\cof\kappa$,
    $$G``((V_{\kappa}^2\times([i]^{<\kappa})^3)\times(\kappa^2\times i^2\times\kappa)^{<\kappa}\times(\kappa\times V_{\kappa}\times i^2\times V_{\kappa})^{<\kappa})\subseteq i.$$
  This will make $G\circ F:\kappa^+\rightarrow\kappa^+$ regressive on $C_0$.

  Define $C_1$ to be the club subset of $\kappa^+$ consisting of points $i'$ such that for all $i<i'$ we have that $R^i, B^i, T^i\subseteq i'$ and for all $\alpha\neq\beta$ in $A^i$ that $x_{\alpha}\upharpoonright i'\neq x_{\beta}\upharpoonright i'$.

  We will prove that the regressive function $G\circ F$ and club $C_0\cap C_1$ together serve as witnesses to the strong $\kappa^+$-chain condition. So given $i<i'$ in $C_0\cap C_1$ such that $G(F(i))=G(F(i'))$ we wish to show that $p^i$ is compatible with $p^{i'}$. Note that the properties of $C_1$ plus the fact that $F_1(i)=F_1(i')$ mean that $R^i\cap R^{i'}$, $R^i-R^{i'}$ and $R^{i'}-R^i$ are positioned in increasing order as subsets of $\kappa^+$, and likewise for $B^i$ and $T^i$.

  First consider any $\eta\in a^i\cap\sup a^i=a^{i'}\cap\sup a^{i'}$, $\alpha\in A^i\cap A^{i'}$ and $\zeta\in B^i\cap B^{i'}$. It is clear that the first co-ordinates of $f^{i,\eta}_{\alpha}(\zeta)$ and $f^{i',\eta}_{\alpha}(\zeta)$ agree, since they are just $x_{\alpha}\upharpoonright\zeta$; say this is equal to $t(k)$ and then as $k\in T^i$ and $i'\in C_1$ we get $k<i'$. Let $\alpha=:\beta(j)$, with $j\in R^i\cap R^{i'}$ so by the increasing enumeration, $j=:j^i_{\epsilon}=j^{i'}_{\epsilon}$ for some $\epsilon<\mu^i=\mu^{i'}$. We have $\zeta\in B^i\subseteq i'$. Thus the tuple $(\eta,\epsilon,\zeta,k,\pi_2(f^{i',\eta}_{\beta(j^{i'}_{\epsilon})}(\zeta)))$ will be a member of $F_2(i')$ and so also of $F_2(i)$ and we have $f^{i,\eta}_{\alpha}(\zeta)=f^{i',\eta}_{\alpha}(\zeta)$.

  The preceding argument allows us to define a putative lower bound $p^*=(c^*,h^*,t^*,f^*)$ for $p^i$ and $p^{i'}$ given by $c^*=c^i=c^{i'}$, $h^*$ any upper part below $h^i$ and $h^{i'}$, $t^*=t^i\cup t^{i'}$, and $f^{*,\eta}_{\alpha}$ equal to either $f^{i,\eta}_{\alpha}\cup f^{i',\eta}_{\alpha}$, $f^{i,\eta}_{\alpha}$ or $f^{i',\eta}_{\alpha}$ depending on whether $\alpha$ is in $A^i\cap A^{i'}$, $A^i-A^{i'}$ or $A^{i'}-A^i$ respectively. It is clear that this $p^*$ will satisfy the first four clauses of the definition of $\mathbb{Q}_{\vec{u}}$ so it remains to show the fifth.

  Define $a^*:=a^i=a^{i'}$. We will be given $\eta\in a^*\cap\sup a^*$, $\alpha,\beta\in t^{*,\eta}$, $s$ harmonious with $c^*$ past $\eta$ and $\zeta,\zeta'\in \dom f^{*,\eta}_{\alpha}\cap f^{*,\eta}_{\beta}$ such that $f^{*,\eta}_{\alpha}(\zeta) = f^{*,\eta}_{\beta}(\zeta)\neq f^{*,\eta}_{\alpha}(\zeta')=f^{*,\eta}_{\beta}(\zeta')$. We wish to show that
  $$s\frown((\vec{u},h^*))\forces \zeta\dot{\mathcal{E}}_{\alpha}\zeta'\leftrightarrow\zeta\dot{\mathcal{E}}_{\beta}\zeta'.$$
  We see that
  $$(\alpha,\zeta),(\alpha,\zeta'),(\beta,\zeta),(\beta,\zeta')\in (A^i\times B^i)\cup(A^{i'}\times B^{i'}),$$
  which compels that either all the co-ordinates occur in a single one of $A^i\times B^i$ or $A^{i'}\times B^{i'}$, from which the result is obvious, or (without loss of generality) that we have one of the following two cases.

  \begin{case1}
  $\alpha,\beta \in A^i\cap A^{i'}$, $\zeta\in B^i-B^{i'}$ and $\zeta'\in B^{i'}-B^i$.
  \end{case1}

  We may assume $\alpha\neq\beta$. The definition of $C_1$ and the fact that $\alpha,\beta \in A^i$ ensures that $x_{\alpha}\upharpoonright i'\neq x_{\beta}\upharpoonright i'$. But $B^{i'}\cap i'=B^i\cap i$ so we must have have $\zeta'\geq i'$, giving $x_{\alpha}\upharpoonright \zeta'\neq x_{\beta}\upharpoonright \zeta'$. This contradicts $f^{*,\eta}_{\alpha}(\zeta')=f^{*,\eta}_{\beta}(\zeta')$.

  \begin{case2}
  $\alpha\in A^i-A^{i'}$, $\beta\in A^{i'}-A^i$ and $\zeta,\zeta'\in B^i\cap B^{i'}$.
  \end{case2}

  Take $j$ such that $\beta(j)=\alpha$, $j'$ such that $\beta(j')=\beta$, $\epsilon$ such that $j^i_{\epsilon}=j$ and $\epsilon'$ such that $j^{i'}_{\epsilon'}=j'$. Take $k$ such that $x_{\alpha}\upharpoonright\zeta=x_{\beta}\upharpoonright\zeta=t(k)$ and $k'$ such that $x_{\alpha}\upharpoonright\zeta'=x_{\beta}\upharpoonright\zeta=t(k')$; note that $k,k'\in T^i\cap T^{i'}\subseteq i$. Likewise $\zeta,\zeta'< i$. Combining all this information tells us that the tuples $(\eta, \epsilon, \zeta, k, \nu)$ and $(\eta, \epsilon, \zeta', k', \nu')$ appear in $F_2(i)$ for some $\nu$ and $\nu'$. Hence they also appear in $F_2(i')$ and we have
  $$f^{i',\eta}_{\beta(j^{i'}_{\epsilon})}(\zeta)=f^{i,\eta}_{\beta(j^i_{\epsilon})}(\zeta)=f^{*,\eta}_{\alpha}(\zeta)=f^{*,\eta}_{\beta}(\zeta)=f^{i',\eta}_{\beta}(\zeta)$$
  and similarly for $\zeta'$. These equalities occur entirely inside $p^{i'}$ so we can invoke its conditionhood to get
  $$s\frown((\vec{u},h^{i'}))\forces \zeta\dot{\mathcal{E}}_{\beta(j^{i'}_{\epsilon})}\zeta'\leftrightarrow\zeta\dot{\mathcal{E}}_{\beta}\zeta'.$$

  Define $\tilde{s}$ to be equal to $s$ except that the third co-ordinate of $\max \tilde{s}$ should be trivial. We will have $(\epsilon, \tilde{s}, \zeta,\zeta', \dot{y}^i_{\beta(j^i_{\epsilon}),\tilde{s}}(\zeta,\zeta'))$ in $F_3(i)$ and thus in $F_3(i')$, with $\dot{y}^i_{\alpha,\tilde{s}}(\zeta,\zeta')=\dot{y}^i_{\beta(j^i_{\epsilon}),\tilde{s}}(\zeta,\zeta'))=\dot{y}^{i'}_{\beta(j^{i'}_{\epsilon}),\tilde{s}}(\zeta,\zeta'))$. Since $s$ is below $\tilde{s}$ we know
  $$s\frown((\vec{u},h^i))\forces\zeta\dot{\mathcal{E}}_{\alpha}\zeta'\leftrightarrow\dot{y}^i_{\alpha,\tilde{s}}(\zeta,\zeta')$$
  and
  $$s\frown((\vec{u},h^{i'}))\forces\zeta\dot{\mathcal{E}}_{\beta(j^{i'}_{\epsilon})}\zeta'\leftrightarrow\dot{y}^{i'}_{\beta(j^{i'}_{\epsilon}),\tilde{s}}(\zeta,\zeta').$$
  Putting all these results together yields what we want.
\end{proof}

\section{Construction of the model} \label{construction_of_model}

We now perform an iteration of length $\kappa^{+4}$ of preparatory forcings, under the following assumptions. Note that the behaviour of the power-set function given here can be obtained from any model in which $\kappa$ is supercompact whilst preserving supercompactness.

\begin{setting}
  Let $\kappa$ be supercompact, $2^{\kappa}=\kappa^+$, $2^{\kappa^+}=\kappa^{+3}$, $2^{\kappa^{+3}}=\kappa^{+4}$ and $\lambda<\kappa$ regular uncountable.
\end{setting}

\subsection{The forcing construction}

Fix $\langle x_{\epsilon}\mid \epsilon<\kappa^{+3}\rangle$ an enumeration of the branches of the complete binary tree $T$ on $\kappa^+$.

\begin{lemma} \label{propertiesOfP}
  Let $\mathbb{P}$ be a $<\kappa$-support iteration of length $\kappa^{+4}$ of forcings that are either trivial or of the form $\mathbb{Q}_{\vec{u}}$ for some $\vec{u}$. Then $\mathbb{P}$ is $\kappa$-directed closed and has the $\kappa^+$-chain condition. Also $2^{\kappa}=2^{\kappa^+}=\kappa^{+3}$ at intermediate stages, and $2^{\kappa}=\kappa^{+4}$ at the end of the iteration.
\end{lemma}

\begin{proof}
  We have from lemma \ref{Qcompact} that the $\mathbb{Q}_{\vec{u}}$-forcings are $\kappa$-compact, hence $\kappa$-directed closed. It is clear that a $<\kappa$-support iteration of such forcings will remain $\kappa$-directed closed.

  A forcing is said to be {\em countably parallel closed} if any two descending $\omega$-sequences in it that are pointwise compatible have a common lower bound. It is clear that this property follows from $\kappa$-compactness. We also have that the component forcings are $\kappa$-closed and have the strong $\kappa^+$-chain condition, so we can invoke \cite[Theorem 1.2]{5author} to deduce that $\mathbb{P}$ has the strong $\kappa^+$-chain condition, and hence the usual $\kappa^+$-cc.

  Call the intermediate stages of the forcing $\mathbb{P}_{\gamma}$ for $\gamma<\kappa^{+4}$. We can prove by induction on $\gamma$ that $|\mathbb{P}_{\gamma}|=\kappa^{+3}$ and $(2^{\kappa^+})^{V^{\mathbb{P}_{\gamma}}}=(2^{\kappa^+}\times\kappa^{+3})^{\kappa}=\kappa^{+3}$. The latter follows from the former by the usual analysis of names together with the $\kappa^+$-cc. Conversely, conditions from $\dot{\mathbb{Q}}_{\vec{u}}$ for $\vec{u}\in V^{\mathbb{P}_{\gamma}}$ are members of $(V_{\kappa}\times 2^{\kappa}\times [\kappa\times\kappa^{+3}]^{<\kappa}\times [\kappa\times\kappa^{+3}\times\kappa]^{<\kappa})^{V^{\mathbb{P}_{\gamma}}}$, where we drop the first co-ordinate of the $f^{\eta}_{\epsilon}(\zeta)$ since it can be deduced from $\zeta$ and $\epsilon$. Thus we can use the $\kappa^+$-cc of $\mathbb{P}_{\gamma}$ and the fact that $(2^{\kappa})^{V^{\mathbb{P}_{\gamma}}}=\kappa^{+3}$ to encode them as member of $\kappa^{+3}$. Hence $|\mathbb{P}_{\gamma+1}|=\kappa^{+3}$ and the induction proceeds. Limit stages for $\gamma<\kappa^{+4}$ are immediate by the $<\kappa$-support, and then at the end we get $2^{\kappa}=\kappa^{+4}$ are desired.
\end{proof}

Define $\mathbb{L}$ to be the Laver preparatory forcing to make $\kappa$ indestructible under $\kappa$-directed closed forcing, as given in \cite{Laver}. After this forcing we still have $2^{\kappa^{+3}}=\kappa^{+4}$ so by a result from \cite{ShelahDiamond} we have a $\lozenge_{\kappa^{+4}}(\kappa^{+4}\cap\cof(\kappa^{++}))$-sequence $\langle S_{\gamma}\mid\gamma<\kappa^{+4}\rangle$. We will perform an iteration $\mathbb{P}$ of the type described above
but before doing so we wish to establish a list in $V^{\mathbb{L}}$ of all possible $\mathbb{P}$-names for subsets of $\kappa$, regardless of the sequence of $\vec{u}_{\gamma}$ we end up using to construct $\mathbb{P}$. We can do so by inductively building a list of possible $\mathbb{P}_{\gamma}$-names for subsets of $\kappa$:
\begin{itemize}
  \item For $\gamma=\delta+1$ a $\mathbb{P}_{\gamma}$-name for a subset of $\kappa$ is a $\mathbb{P}_{\delta}$-name for a $\dot{\mathbb{Q}}_{\delta}$-name for a subset of $\kappa$. Such a $\dot{\mathbb{Q}}_{\delta}$-name is, by the $\kappa^+$-cc, a function from $\kappa$ to $\dot{\mathbb{Q}}_{\delta}\times 2$ and as in the proof of \ref{propertiesOfP} members of $\dot{\mathbb{Q}}_{\delta}$ can be encoded as members of $(2^{\kappa})^{V^{\mathbb{L}*\mathbb{P}_{\delta}}}$. We note that this encoding can be done merely be looking at the shape of possible conditions, without knowledge of $\vec{u}_{\delta}$. The list of possible $\mathbb{P}_{\delta}$-names for subsets of $\kappa$ can now be used to list all the possible $\mathbb{P}_{\gamma}$-names for subsets of $\kappa$.
  \item For $\gamma$ limit the listing is straightforward because of the $\kappa^+$-cc.
\end{itemize}
Members of $\mathcal{U}^{V^{\mathbb{P}*\mathbb{L}}}$ are essentially subsets of $(2^{\kappa})^{V^{\mathbb{P}*\mathbb{L}}}$ so our listing allows us to translate between subsets of $\kappa^{+4}$ in $V^{\mathbb{L}}$ and anything that could possibly turn out to be a $\mathbb{P}$-name for a member of $\mathcal{U}$.

We are now ready to define the $<\kappa$-support iteration $\mathbb{P}=\langle \mathbb{P}_{\gamma},\mathbb{Q}_{\delta}\mid \gamma\leq\kappa^{+4},\delta<\kappa^{+4}\rangle$. At stage $\gamma$, apply the translation just established to $S_{\gamma}\subseteq\kappa^{+4}$. If the result is a $\mathbb{P}$-name for a member of $\mathcal{U}$ that is in fact already a $\mathbb{P}_{\gamma}$-name then instantiate this name in $\mathbb{P}_{\gamma}$ and call the result $\vec{u}^{\gamma}$. Use \ref{propertiesOfP} to fix $\langle \dot{\mathcal{E}}^{\gamma}_{\epsilon}\mid \epsilon<\kappa^{+3}\rangle$ an enumeration of the $\mathbb{R}_{\vec{u}^{\gamma}}$-names for graphs on $\kappa^+$. Define $\mathbb{Q}_{\gamma}=\mathbb{Q}_{\vec{u}^{\gamma}}$, working with respect to the sequences $\langle x_{\epsilon}\mid \epsilon<\kappa^{+3}\rangle$ and $\langle \dot{\mathcal{E}}^{\gamma}_{\epsilon}\mid \epsilon<\kappa^{+3}\rangle$. Otherwise take $\mathbb{Q}_{\gamma}$ to be the trivial forcing.

Let $G*H$ be $\mathbb{L}*\mathbb{P}$-generic. If $\mathbb{Q}_{\gamma}$ is non-trivial then $H(\gamma)$ will add a potential upper part which we call $h^{\gamma}$, and a sequence of functions which we call $F^{\gamma}=\langle F^{\gamma,\eta}_{\alpha}\mid \eta =\kappa(\vec{w}), \vec{w}\in\dom h^{\gamma}, \alpha<\kappa^{+3}\rangle$.

\begin{lemma} \label{staty1}
  Let $\vec{u}\in\mathcal{U}^{V[G][H]}$. Then in $V[G][H]$ there is a stationary set of $\gamma<\kappa^{+4}$ of cofinality $\kappa^{++}$ such that $\vec{u}^{\gamma}$ is the restriction of $\vec{u}$ to $V[G][H\upharpoonright \gamma]$ and $\mathbb{Q}_\gamma=\mathbb{Q}_{\vec{u}^{\gamma}}$.
\end{lemma}

\begin{proof}
  There is a club of points $\gamma$ in $\kappa^{+4}$ where the members of $(2^{\kappa})^{V[G][H]}$ listed as above by ordinals below $\gamma$ are exactly $\bigcup_{\delta<\gamma}(2^{\kappa})^{V[G][H\upharpoonright\delta]}$. For such $\gamma$ of cofinality at least $\kappa^+$ the $\kappa^+$-cc of $\mathbb{P}_{\gamma}$ makes this equal to $(2^{\kappa})^{V[G][H\upharpoonright\gamma]}$. Take a $\mathbb{P}$-name for $\vec{u}$ and use the above translation to convert it into a subset of $\kappa^{+4}$; the diamond sequence then gives us a stationary set of $\gamma<\kappa^{+4}$ of cofinality $\kappa^{++}$ such that $\vec{u}^{\gamma}$ is given by restricting $\vec{u}$ to subsets of $\kappa$ that belong to $V[G][H\upharpoonright\gamma]$. Now all the properties in the definition of $\mathcal{U}$ are $\Pi^1_2$ over $V_{\kappa}$, so there a club of $\gamma$ where the restriction of $\vec{u}$ to $V[G][H\upharpoonright\gamma]$ is a member of $\mathcal{U}^{V[G][H\upharpoonright\gamma]}$. Combining these two facts gives a stationary set of $\gamma$ where $\vec{u}$ restricts to $\vec{u}^{\gamma}$ and $\mathbb{Q}_\gamma=\mathbb{Q}_{\vec{u}^{\gamma}}$.
\end{proof}

Observe that by the properties of the Laver preparation and the fact that $\mathbb{P}$ is $\kappa$-directed closed (by lemma \ref{propertiesOfP}) we can take $j:V\rightarrow M$ witnessing that $\kappa$ is highly supercompact and $j(\mathbb{L})(\kappa)=\mathbb{P}$, and then find a master condition allowing us to extend $j$ to an embedding $j:V[G]\rightarrow M[G][H][I]$ where $I$ is generic for a highly closed forcing. We can then use the methods of section 2 to derive $\vec{u}\in\mathcal{U}^{V[G][H][I]}$ from $j$, and observe by the closure that in fact $\vec{u}\in V[G][H]$. It will then be possible to apply the above lemma to $\vec{u}$. However we will actually need to be more careful than this in the construction of our master condition, because we want to ensure that $h^{\gamma}\in \mathcal{F}_{\vec{u}}$ stationarily-often.

\begin{lemma}
  There is $\vec{u}\in\mathcal{U}^{V[G][H]}$ such that in $V[G][H]$ there is a stationary set of $\gamma<\kappa^{+4}$ of cofinality $\kappa^{++}$ such that $\vec{u}^{\gamma}$ is the restriction of $\vec{u}$ to $V[G][H\upharpoonright \gamma]$, $\mathbb{Q}_\gamma=\mathbb{Q}_{\vec{u}^{\gamma}}$, and $h^{\gamma} \in \mathcal{F}_{\vec{u}}$.
\end{lemma}

\begin{proof}
  Take $\mu$ large and $j:V\rightarrow M$ witnessing that $\kappa$ is $\mu$-supercompact with $j(\mathbb{L})(\kappa)=\mathbb{P}$ and $j(\mathbb{L})(\alpha)$ trivial for $\alpha\in(\kappa,\mu)$. We have $j$ fixing $G$ pointwise so we can extend $j$ to $j:V[G]\rightarrow M[G][H][I]$ where $I$ is some $j(\mathbb{L})/(\mathbb{L}*\mathbb{P})$-generic over $M$. We will now build a master condition in $j(\mathbb{P})$ by inductively defining a descending sequence $p_{\gamma}\in j(\mathbb{P}_{\gamma})$ for $\gamma<\kappa^{+4}$ such that $\forces p_{\gamma}\leq j``(H\upharpoonright\gamma)$.

  For $\gamma$ limit take $p_{\gamma}$ to be any lower bound of $\langle p_{\delta}\mid \delta<\gamma\rangle$, using that the forcing is highly closed. We will have $p_{\gamma+1}:=p_{\gamma}\frown(q_{\gamma})$ for $q_{\gamma}$ to be defined. We can force below $p_{\gamma}$ to lift $j$ to $j:V[G][H\upharpoonright\gamma]\rightarrow M[G][H][I][j(H\upharpoonright\gamma)]$. If $\mathbb{Q_{\gamma}}$ is the trivial forcing then so is $j(\mathbb{Q}_{\gamma})$ and we take $q_{\gamma}$ to be its unique member. Otherwise we set $q_{\gamma}=(\tilde{c}^{\gamma},\tilde{h}^{\gamma},\tilde{t}^{\gamma},\tilde{f}^{\gamma})$ with definitions as follows.
  $$\dom \tilde{c}^{\gamma} :=\dom h^{\gamma}\cup \left\{\vec{w} \in \mathcal{U}^{V[G][H]} \middle| \begin{array}{l}
    \exists i<\lambda: \vec{w}\cap V[G][H\upharpoonright\gamma]=\vec{u}^{\gamma}\upharpoonright i, h^{\gamma}_{<i}\in \mathcal{F}_{\vec{w}},\\
    \forall(c,h,t,f)\in H(\gamma): \vec{w}\in\dom j(h), \\
    \bigwedge_{(c,h,t,f)\in H(\gamma)}j(h)(\vec{w})\neq 0
  \end{array}\right\},$$
  with $\tilde{c}^{\gamma}(\vec{w}):=h^{\gamma}(\vec{w})$ for $\vec{w}\in\dom h^{\gamma}$, and
  $$\tilde{c}^{\gamma}(\vec{w}):=\bigwedge_{(c,h,t,f)\in H(\gamma)} j(h)(\vec{w})$$
  for $\vec{w}\in\dom\tilde{c}^{\gamma}-\dom h^{\gamma}$. We set
  $$\tilde{h}^{\gamma}:=\bigwedge_{(c,h,t,f)\in H(\gamma)} j(h),$$
  $\tilde{t}^{\gamma}:=(a^{(\tilde{c}^{\gamma},\tilde{h}^{\gamma})}\cap\kappa) \times j``\kappa^{+3}$, and $\dom (\tilde{f}^{\gamma})^{\eta}_{j(\alpha)}=j``\kappa^+$ and $(\tilde{f}^{\gamma})^{\eta}_{j(\alpha)}(j(\zeta))=j(F^{\gamma,\eta}_{\alpha}(\zeta))$ for all $\eta \in a^{(\tilde{c}^{\gamma},\tilde{h}^{\gamma})}\cap\kappa$, $\alpha\in\kappa^{+3}$ and $\zeta\in\kappa^+$.

  \begin{claim}
    $(\tilde{c}^{\gamma},\tilde{h}^{\gamma},\tilde{t}^{\gamma},\tilde{f}^{\gamma})\in j(\mathbb{Q}_{\gamma})$.
  \end{claim}

  \begin{proof}
    The requirement that $h^{\gamma}_{<i}\in \mathcal{F}_{\vec{w}}$ for those $\vec{w}\in \dom\tilde{c}^{\gamma}$ with $\kappa(\vec{w})=\kappa$ ensures that $\tilde{c}^{\gamma}$ is an acceptable first co-ordinate for a condition in $j(\mathbb{M}_{\vec{u}^{\gamma}})$. The first four clauses of the definition then follow from the fact that $j(\kappa)$ is large. For the fifth we are given $\eta \in a^{(\tilde{c}^{\gamma},\tilde{h}^{\gamma})}\cap\kappa$, $\alpha,\beta\in\kappa^{+3}$, $s$ a lower part for $j(\mathbb{R}_{\vec{u}^{\gamma}})$ that is harmonious with $\tilde{c}^{\gamma}$ past $\eta$, and $\zeta,\zeta'\in\kappa^+$ such that $(\tilde{f}^{\gamma})^{\eta}_{j(\alpha)}(j(\zeta))=(\tilde{f}^{\gamma})^{\eta}_{j(\beta)}(j(\zeta))\neq (\tilde{f}^{\gamma})^{\eta}_{j(\alpha)}(j(\zeta'))=(\tilde{f}^{\gamma})^{\eta}_{j(\beta)}(j(\zeta'))$. By elementarity this last assertion is equivalent to $F^{\gamma,\eta}_{\alpha}(\zeta)=F^{\gamma,\eta}_{\beta}(\zeta)\neq F^{\gamma,\eta}_{\alpha}(\zeta')=F^{\gamma,\eta}_{\beta}(\zeta')$.

    If $\kappa(\max s)<\kappa$ then use Lemma \ref{squareOff} to take a condition $(c,h,t,f)\in H(\gamma)$ with $\eta\in a^{(c,h)}\cap\sup a^{(c,h)}$, $\alpha,\beta \in t^{\eta}$, $s$ harmonious with $c$ past $\eta$, and $\zeta,\zeta'\in \dom f^{\eta}_{\alpha}\cap \dom f^{\eta}_{\beta}$. Then
    $f^{\eta}_{\alpha}(\zeta)=f^{\eta}_{\beta}(\zeta)\neq f^{\eta}_{\alpha}(\zeta')=f^{\eta}_{\beta}(\zeta')$ so we get
    $$s\frown((\vec{u}^{\gamma},h))\forces \zeta\dot{\mathcal{E}}^{\gamma}_{\alpha}\zeta' \leftrightarrow \zeta\dot{\mathcal{E}}^{\gamma}_{\beta}\zeta'.$$
    Now $s\frown((j(\vec{u}^{\gamma}),\tilde{h}^{\gamma}))\leq s\frown((j(\vec{u}^{\gamma}),j(h)))$ so together with elementarity we obtain
    $$s\frown((j(\vec{u}^{\gamma}),\tilde{h}^{\gamma})) \forces j(\zeta) j(\dot{\mathcal{E}}^{\gamma}_{\alpha})j(\zeta') \leftrightarrow j(\zeta) j(\dot{\mathcal{E}}^{\gamma}_{\beta})j(\zeta')$$
    as required.

    Otherwise we can write $s$ as $s_1\frown((\vec{w}, d, p))$ for some $\vec{w}\in\dom \tilde{c}^{\gamma}-\dom h^{\gamma}$. We will show that $s\frown((j(\vec{u}^{\gamma}),\tilde{h}^{\gamma}))$ forces what we want by a density argument. Suppose we are given an extension $s^*\frown((j(\vec{u}^{\gamma}),h^*))$; express $s^*$ as $s^*_1\frown((\vec{w}, d^*,p^*))\frown s^*_2$. Lemma \ref{openHarmony} tells us that $s^*_1\frown((\vec{w}, d^*,p^*))$ remains harmonious with $\tilde{c}^{\gamma}$ past $\eta$, so we can use Lemma \ref{squareOff} to take $(c,h,t,f)\in H(\gamma)$ with $\eta\in a^{(c,h)}\cap\sup a^{(c,h)}$, $\alpha,\beta \in t^{\eta}$, $s^*_1$ harmonious with $c$ past $\eta$, and $\zeta,\zeta'\in \dom f^{\eta}_{\alpha}\cap \dom f^{\eta}_{\beta}$. As before the conditionhood of $(c,h,t,f)$ followed by the elementarity of $j$ give that
    $$s^*_1\frown((j(\vec{u}^{\gamma}),j(h))) \forces j(\zeta) j(\dot{\mathcal{E}}^{\gamma}_{\alpha})j(\zeta') \leftrightarrow j(\zeta) j(\dot{\mathcal{E}}^{\gamma}_{\beta})j(\zeta').$$
    The harmoniousness of $s$ with $\tilde{c}^{\gamma}$ tells us that $d^*\leq d\leq h^{\gamma}\leq c\cup h$ so we can refine $s^*$ to $s^{**}$ by strengthening $d^*$ to $d^{**}\leq h =j(h)\upharpoonright\kappa$. We also have $\vec{w}\in\dom\tilde{c}^{\gamma}\subseteq\dom j(h)$ and $p^*\leq p \leq \tilde{c}^{\gamma}(\vec{w})\leq j(h)(\vec{w})$, so $(\vec{w},d^{**},p^*)$ is addable below $(j(\vec{u}^{\gamma}),j(h))$. So is  $s^*_2$ (because it is addable below $\tilde{h}^{\gamma}\leq j(h)$) yielding
    $$s^{**}\frown((j(\vec{u}^{\gamma}),h^*)) \forces j(\zeta) j(\dot{\mathcal{E}}^{\gamma}_{\alpha})j(\zeta') \leftrightarrow j(\zeta) j(\dot{\mathcal{E}}^{\gamma}_{\beta})j(\zeta').$$
    And $s^{**}\frown((j(\vec{u}^{\gamma}),h^*))$ is also below $s^*\frown((j(\vec{u}^{\gamma}),h^*))$, concluding the proof of the claim.
  \end{proof}

  It is immediate that $\forces q_{\gamma}\leq j``H(\gamma)$ so $\forces p_{\gamma}\leq j``(H\upharpoonright\gamma)$; this finishes the inductive definition. Take $p$ a lower bound of the sequence of $p_{\gamma}$ as our master condition and force below it to obtain a $j(\mathbb{P})$-generic filter. Then we can extend $j$ to $j:V[G][H]\rightarrow M[G][H][I][j(H)]$, where $j(H)$ is the filter for $j(\mathbb{P})$ just obtained. This embedding will witness a high degree of generic supercompactness so as in section 2 we can in $V[G][H][I][j(H)]$ derive an ultrafilter sequence $\vec{u}$ from it, and show $\vec{u}\in\mathcal{U}$; we also get the associated supercompact ultrafilter sequence $\vec{u}^*=\langle z, u^*_i, K^*_i \mid i<\lambda\rangle$ and the associated projection $\pi$. The $\mu$-closure of the $j(\mathbb{L})/(\mathbb{L}*\mathbb{P})*j(\mathbb{P})$-forcing gives us that $\vec{u}\in V[G][H]$. Then we can invoke Lemma \ref{staty1} to see that there are stationarily-many $\gamma<\kappa^{+4}$ where $\vec{u}$ restricts to $\vec{u}^{\gamma}$ and $\mathbb{Q}_{\gamma}=\mathbb{Q}_{\vec{u}^{\gamma}}$. We wish to show that $h^{\gamma}\in \mathcal{F}_{\vec{u}}$ for such $\gamma$ and will do so by proving by induction on $i$ that $h^{\gamma}_i\in \mathcal{F}_{\vec{u},i}$.

  Given any $(c,h,t,f)\in H(\gamma)$ we have $h_i\in \mathcal{F}_{\vec{u}^{\gamma},i}\subseteq \mathcal{F}_{\vec{u},i}$ so $\dom h_i\in u_i$, which gives
  \begin{align*}
    & \pi^{-1}` ` \dom h_i \in u^*_i \\
    \Rightarrow& \forall_{u^*_i}\vec{w}^* \pi(\vec{w}^*)\in\dom h_i \\
    \Rightarrow& j(\pi)(\vec{u}^*\upharpoonright i)\in\dom j(h_i) \\
    \Rightarrow& \vec{u}\upharpoonright i \in\dom j(h_i)
  \end{align*}
  Also by definition of $\Fil(K^*_i)$ there is an $A\in u^*_i$ with $h_i\geq b(K^*_i,A)$, from which
  \begin{align*}
    & \forall\vec{w}\in\dom h_i: h_i(\vec{w}) \geq \bigvee\{K^*_i(\vec{w}^*)\mid \vec{w}^*\in A, \pi(\vec{w}^*)=\vec{w}\}\\
    \Rightarrow& \forall\vec{w}^*\in A: h_i(\pi(\vec{w}^*)) \geq K^*_i(\vec{w}^*)\\
    \Rightarrow& j(h_i)(\vec{u}\upharpoonright i) \geq j(K^*_i)(\vec{u}^*\upharpoonright i),
  \end{align*}
  using that $A\in u^*_i$ and $j(\pi)(\vec{u}^*\upharpoonright i)=\vec{u}\upharpoonright i$.

  Therefore $j(K^*_i)(\vec{u}^*\upharpoonright i)$ witnesses that $\bigwedge_{(c,h,t,f)\in H(\gamma)}j(h)(\vec{u}\upharpoonright i)\neq 0$. By the induction hypothesis we have $h^{\gamma}_{<i}\in \mathcal{F}_{\vec{u}\upharpoonright i}$ so we established have all of the requirements necessary for $\vec{u}\upharpoonright i\in\dom\tilde{c}^{\gamma}$. We have also shown that $\tilde{c}^{\gamma}(\vec{u}\upharpoonright i)\geq j(K^*_i)(\vec{u}^*\upharpoonright i)$. Now forcing below $p_{\gamma+1}$ ensures that $\tilde{c}^{\gamma}$ is an initial segment of $j(h^{\gamma})$ so we have
  \begin{align*}
    & j(h^{\gamma})(\vec{u}^\upharpoonright i)\geq j(K^*)(\vec{u}^*\upharpoonright i)\\
    \Rightarrow& \forall\vec{w}^*\in B: h^{\gamma}(\pi(\vec{w}^*))\geq K^*(\vec{w}^*) \mbox{ for some } B\in u^*_i\\
    \Rightarrow& \forall\vec{w}\in\pi ` `B: h^{\gamma}_i(\vec{w})\geq \bigvee\{K^*_i(\vec{w}^*)\mid \vec{w}^*\in B, \pi(\vec{w}^*)=\vec{w}\}\\
    \Rightarrow& h^{\gamma}_i\geq b(K^*_i,B)
  \end{align*}
  which gives $h^{\gamma}_i\in \mathcal{F}_{\vec{u},i}$ as desired.
\end{proof}

Fix a $\vec{u}$ and $S\subseteq\kappa^{+4}$ stationary as given by this lemma. Take $J$ that is $\mathbb{R}_{\vec{u}}$-generic over $V[G][H]$, forcing below an upper part whose domain is made up of sequences of length less than $\lambda$, so that $J$ generates a generic sequence $\langle \vec{w}_{\alpha}, g_{\alpha}\mid\alpha<\lambda\rangle$ as discussed in sub-section \ref{defineR}. For any $\gamma\in S$ we observe by the characterisation of genericity in Lemma \ref{characteriseGenericity} that $J$ is geometric for $\mathbb{R}_{\vec{u}}$ and hence $\mathbb{R}_{\vec{u}^{\gamma}}$, so it is generic for $\mathbb{R}_{\vec{u}^{\gamma}}$ and we can form the extension $V[G][H\upharpoonright \gamma][J]$.

\subsection{The jointly universal family}

We now fix some $\gamma \in S$ and define a graph $\mathcal{E}^{\gamma}$ on $T\times\kappa$ that is intended to be universal with respect to the graphs in $V[G][H\upharpoonright \gamma][J]$. We have $h^{\gamma}\in \mathcal{F}_{\vec{u}}$ so start by fixing $\beta<\lambda$ such that $\lh\vec{w}_{\beta}=0$ and for all $\alpha>\beta$ we have $\vec{w}_{\alpha} \in \dom h^{\gamma}$ and $h^{\gamma}(\vec{w}_{\alpha})\in g_{\alpha}$. Define $\eta:=\kappa(\vec{w}_{\beta})$. Define $\mathcal{E}^{\gamma}_{\epsilon}$ to be the realisation of $\dot{\mathcal{E}}^{\gamma}_{\epsilon}$ in $V[G][H\upharpoonright \gamma][J]$. For $z, z' \in T\times\kappa$ we define $z \mathcal{E}^{\gamma} z'$ if there exist $\epsilon$, $\zeta$ and $\zeta'$ such that $F^{\gamma,\eta}_{\epsilon}(\zeta)=z$ and $F^{\gamma,\eta}_{\epsilon}(\zeta')=z'$ with $\zeta \mathcal{E}^{\gamma}_{\epsilon} \zeta'$ in $V[G][H\upharpoonright \gamma][J]$.

\begin{lemma}
  Let $\gamma\in S$, $\eta$ as above and $\epsilon<\kappa^{+3}$. Then in $V[G][H\upharpoonright \gamma][J]$ the function $F^{\gamma,\eta}_{\epsilon}$ is an embedding of $\mathcal{E}^{\gamma}_{\epsilon}$ into $\mathcal{E}^{\gamma}$.
\end{lemma}

\begin{proof}
  It is clear from the definition that every edge in $\mathcal{E}^{\gamma}_{\epsilon}$ is mapped to one in $\mathcal{E}^{\gamma}$, so we need to show the converse. Consider $\epsilon$, $\zeta$ and $\zeta'$ such that $F^{\gamma,\eta}_{\epsilon}(\zeta) \mathcal{E}^{\gamma} F^{\gamma,\eta}_{\epsilon}(\zeta)$. Observe first that the values of $\zeta$ and $\zeta$ are deducible from their targets under $F^{\gamma,\eta}_{\epsilon}$, so there must be some $\epsilon'$ with $\zeta \mathcal{E}^{\gamma}_{\epsilon'} \zeta'$ such that
  $$F^{\gamma,\eta}_{\epsilon}(\zeta)=F^{\gamma,\eta}_{\epsilon'}(\zeta)\neq F^{\gamma,\eta}_{\epsilon}(\zeta')=F^{\gamma,\eta}_{\epsilon'}(\zeta').$$
  Take a condition $s\frown((\vec{u}^{\gamma},h)) \in J$ such that $s\frown((\vec{u}^{\gamma},h))\forces\zeta \dot{\mathcal{E}}^{\gamma}_{\epsilon'}\zeta'$, $s$ extends past $\eta$, and $\vec{w}_{\beta}$ occurs in $s$. Use Lemma \ref{squareOff} to take a condition $(c,h',t,f)\in H(\gamma)$ such that $c$ extends past $\max s$, $a^{(c,h')}$ has a maximal element, $h'\leq h$, $(\eta,\epsilon),(\eta,\epsilon') \in t$ and $\zeta,\zeta'\in \dom f^{\eta}_{\epsilon}\cap \dom f^{\eta}_{\epsilon'}$. Our aim is to find a lower part $s'$ such that:
  \begin{enumerate} [(i)]
    \item $s'$ is harmonious with $c$ past $\eta$.
    \item $s'\frown((\vec{u}^{\gamma},h'))\leq s\frown((\vec{u}^{\gamma},h))$.
    \item $s'\frown((\vec{u}^{\gamma},h'))\in J$.
  \end{enumerate}
  Then we will use (i) to invoke the fifth clause of the definition of $\mathbb{Q}_{\vec{u}^{\gamma}}$ for $(c,h',t,f)$ to see that
  $$s'\frown((\vec{u}^{\gamma},h'))\forces \zeta \dot{\mathcal{E}}^{\gamma}_{\epsilon} \zeta' \leftrightarrow \zeta \dot{\mathcal{E}}^{\gamma}_{\epsilon'} \zeta'$$
  which by (ii) will give $s'\frown((\vec{u}^{\gamma},h'))\forces\zeta \dot{\mathcal{E}}^{\gamma}_{\epsilon}\zeta'$ and then by (iii) we will be done.

  We construct $s'$ from $s$ as follows:
  \begin{itemize}
    \item Leave triples $(\vec{w}_{\alpha},d,p)$ with $\kappa(\vec{w}_{\alpha})\leq\eta$ (i.e. $\alpha\leq\beta$) unchanged.
    \item For $(\vec{w}_{\alpha},d,p)\in s$ with $\alpha>\beta$ replace with $(\vec{w}_{\alpha},d\wedge(c\upharpoonright\kappa(\vec{w}_{\alpha})),p\wedge c(\vec{w}_{\alpha}))$.
    \item The set of $\kappa(\vec{w}_{\alpha})$ is a club, and $\ssup a^{(c,h')}$ is a successor. This means we can take $\beta'$ maximal such that $\vec{w}_{\beta'}\in \dom c$. Then add $(\vec{w}_{\beta'},(h\upharpoonright\kappa(\vec{w}_{\beta'}))\wedge(c\upharpoonright\kappa(\vec{w}_{\beta'})),h(\vec{w}_{\beta'})\wedge c(\vec{w}_{\beta'}))$ to the end of $s$.
  \end{itemize}
  For $(\vec{w}_{\alpha},d,p)\in s$ with $\alpha>\beta$ note that $c$ is an initial segment of an upper part $h^{\gamma}$ and $\vec{w}_{\alpha}\in\dom h^{\gamma}$ so we are guaranteed that $c\upharpoonright\kappa(\vec{w}_{\alpha})\in \mathcal{F}_{\vec{w}_{\alpha}}$ for such $\alpha$. Also $p\in g_{\alpha}$, and $c(\vec{w}_{\alpha})=h^{\gamma}(\vec{w}_{\alpha})\in g_{\alpha}$ by choice of $\beta$, so $p$ and $c(\vec{w}_{\alpha})$ are compatible. The same holds for $\beta'$, ensuring $s'$ is a valid lower part. Now we check that it has the required properties.
  \begin{enumerate} [(i)]
    \item This is immediate from the definition (and the reason for the appearance of $c$ in it).
    \item The new triple of $s'$ must be addable to $s\frown((\vec{u}^{\gamma},h))$ on account of its being in $J$, and we have taken care to respect $h$ here.
    \item We will use the requirements from Definition \ref{defnGenericFilter} for a condition to belong to the generic filter $J$ associated with $\langle \vec{w}_{\alpha}, g_{\alpha}\mid\alpha<\lambda\rangle$. The first clause is clear so we consider the second. For $\alpha<\beta$ we have $\vec{w}_{\alpha}$ below a triple of $s$ that is not modified, so all is well. For $\beta<\alpha<\beta'$ we have $\vec{w}_{\alpha}\in h^{\gamma}$ with $h^{\gamma}(\vec{w}_{\alpha})\in g_{\alpha}$. Now $c$ is an initial segment of $h^{\gamma}$ that extends to $\vec{w}_{\beta'}$ so $\vec{w}_{\alpha}\in \dom c$ and $c(\vec{w}_{\alpha})\in g_{\alpha}$. This means that the modifications made to the members of $s$ are unproblematic. (It is for this step that we had to add the extra triple to $s'$.) Finally for $\alpha>\beta'$ we have that $\kappa(\vec{w}_{\alpha})>\kappa(\max \dom c)$ so the fact that $(c,h',t,f)\in H(\gamma)$ and $\vec{w}_{\alpha}\in\dom h^{\gamma}$ tells us that $\vec{w}_{\alpha}\in\dom h'$; likewise $g_{\alpha}\ni h^{\gamma}(\vec{w}_{\alpha})\leq h'(\vec{w}_{\alpha})$.
  \end{enumerate}
\end{proof}

We can now conclude the proof. Take a sequence $\langle\delta_i\mid i<\kappa^{++}\rangle$ of points from $S$ such that $\delta:=\sup \delta_i$ is in $S$. Our final model will be $V[G][H\upharpoonright\delta][J]$ and the family of universal graphs will be $\{\mathcal{E}^{\delta_i}\mid i<\kappa^{++}\}$. Given some graph $\mathcal{E}$ in the model, take a $\mathbb{R}_{\vec{u}^{\delta}}$-name $\dot{\mathcal{E}}$ in $V[G][H\upharpoonright\delta]$ for it. By the $\kappa^+$-cc of $\mathbb{R}_{\vec{u}^{\delta}}$ this name can be coded as a subset of $\kappa^+$ and then by the $\kappa^+$-cc of the forcing iteration we can find some $i<\kappa^{++}$ such that $\dot{\mathcal{E}}$ is in $V[G][H\upharpoonright\delta_i]$. Since $\vec{u}^{\delta_i}$ is a restriction of $\vec{u}^{\delta}$ we see that $\mathbb{R}_{\vec{u}^{\delta_i}}$ will also interpret the name as $\mathcal{E}$, and the lemma above shows that it can be embedded into $\mathcal{E}^{\delta_i}$.

By Lemma \ref{propertiesOfP} $\mathbb{L}\times\mathbb{P}_{\gamma}$ preserves all cardinals, and then by Proposition \ref{preserveCardinals} $\mathbb{R}_{\vec{u}}$ changes $\kappa$ to $\aleph_{\lambda}$ and preserves all larger cardinals. We have proved the following theorem,

\begin{theorem}
  Let $\kappa$ be supercompact and $\lambda<\kappa$ regular uncountable. Then there is a forcing extension in which $\kappa=\aleph_{\lambda}$, $2^{\aleph_{\lambda}}=2^{\aleph_{\lambda+1}}=\aleph_{\lambda+3}$ and there is a jointly universal family of graphs on $\aleph_{\lambda+1}$ of size $\aleph_{\lambda+2}$.
\end{theorem}

\chapter{Set-theoretic geology}

\section{Easton-support iteration of Cohen forcing} \label{easton_generic_mantle}

The following result of set-theoretic geology is well-known, for example it follows from \cite{FHR}[Theorem 65].

\begin{theorem}
  Assume $V=L$. Let $\mathbb{P}$ be a class Easton-support product of $\Add(\kappa,1)$ at $\kappa$ regular with generic $G$. Then $M^{V[G]}=gM^{V[G]}=V$.
\end{theorem}

\begin{proof}
  Consider any $x\in V[G]-V$; we may assume $x\subseteq\kappa$ for some regular $\kappa$. Then split $\mathbb{P}$ as $\mathbb{P}_0\times\mathbb{P}_1$ where $\mathbb{P}_0$ is the product up to and including $\Add(\kappa,1)$, and $\mathbb{P}_1$ is the product from $\Add(\kappa^+,1)$ onwards; split $G$ correspondingly as $G_0\times G_1$. Then the $\kappa^+$-closure of $\mathbb{P}_1$ implies that $x\notin V[G_1]$, but $V[G_1][G_0]=V[G]$ so this is a ground of $V[G]$ omitting $x$. We obtain $x\notin M^{V[G]}$ and since the generic mantle is always a subclass of the mantle this also means that $x\notin gM^{V[G]}$.
\end{proof}

This contrasts with the class Easton support iteration of $\Add(\kappa,1)$ at $\kappa$ regular with generic $G$, which was used by Hamkins, Reitz and Woodin in \cite{HRW} to construct a model in which $V[G]=M^{V[G]}$ but $V[G]\neq\HOD^{V[G]}$. We will now show that $gM^{V[G]}=V[G]$, answering the question posed by Fuchs, Hamkins and Reitz in \cite[Question 69]{FHR} which asks for the generic mantle of this model. To this end we will first prove the following lemma.

For a complete subposet $\mathbb{M}$ of a poset $\mathbb{N}$ and an $\mathbb{N}$-name $\dot{x}$ we will write $\forces_{\mathbb{N}}\dot{x}\in V[\mathbb{M}]$ to mean that for every $\mathbb{N}$-generic $G$ we have that $\dot{x}[G]\in V[G\cap\mathbb{M}]$; this is a weaker notion than $\dot{x}\in V^{\mathbb{M}}$. Similarly $\forces_{\mathbb{N}}\dot{x}\notin V[\mathbb{M}]$ denotes that $\dot{x}[G]\notin V[G\cap\mathbb{M}]$ for every $\mathbb{N}$-generic $G$.

We recall that for a poset $\mathbb{P}$ and $\mathbb{P}$-name for a poset $\dot{\mathbb{Q}}$ the \textit{termspace} forcing $\mathbb{A}(\mathbb{P}, \dot{\mathbb{Q}})$ consists of all $\mathbb{P}$-names $\dot{q}$ such that $\forces\dot{q}\in\dot{\mathbb{Q}}$ and $\dot{q}$ is rank-minimal among all $\mathbb{P}$-names $\dot{q}'$ such that $\forces_{\mathbb{P}}\dot{q}=\dot{q}'$. Note that this is guaranteed to be a set rather than a proper class. Its ordering is given by $\dot{q}\leq\dot{q}'$ iff $\forces\dot{q}\leq\dot{q}'$.

More information about termspace forcing can be found in \cite{saturated_ideals}. In particular it is easy to see that if $\dot{\mathbb{Q}}$ is forced by $\mathbb{P}$ to be $\lambda$-strategically closed then $\mathbb{A}(\mathbb{P},\dot{\mathbb{Q}})$ is $\lambda$-strategically closed. Also whenever we have $(p,\dot{q})\leq(p',\dot{q}')$ in $\mathbb{P}*\dot{\mathbb{Q}}$ we can use the maximum principle to find $\dot{q}^*$ such that $\forces\dot{q}^*\leq q'$ and $p\forces\dot{q}^*=\dot{q}$, from which $(p,\dot{q}^*)$ is equivalent to $(p,\dot{q})$ in $\mathbb{P}*\dot{\mathbb{Q}}$.

\begin{lemma} \label{EastonHelp}
  Let $\mathbb{P}_0$ be a poset with $|\mathbb{P}_0|\leq\kappa$, and let $\dot{\mathbb{R}}$ and $\dot{\mathbb{P}}_1$ be $\mathbb{P}_0$-names for posets such that $\forces_{\mathbb{P}_0}|\dot{\mathbb{R}}|\leq\kappa$ and $\forces_{\mathbb{P}_0}\dot{\mathbb{P}}_1\;\kappa^+$-strategically-closed ($\dot{\mathbb{P}}_1$ may be a name for either a class or a set forcing).

  Let $\dot{x}$ and $\dot{z}$ be names such that $\forces_{\mathbb{P}_0}\dot{x}\subseteq\kappa$ and $\forces_{\mathbb{P}_0*(\dot{\mathbb{R}}\times\dot{\mathbb{P}}_1)}\dot{z}\notin V[\mathbb{P}_0*\dot{\mathbb{R}}]$.

  Then $\forces_{\mathbb{P}_0*(\dot{\mathbb{R}}\times\dot{\mathbb{P}}_1)}\dot{x}\in V[\dot{z}]$.
\end{lemma}

We pause to remark that a simpler version of this lemma in which $\dot{\mathbb{R}}$ is trivial suffices to show that the mantle is equal to the whole model.

\begin{proof}
  We may assume that $\forces\dot{z}\subseteq\theta$ for some $\theta\in\OR$. Define the termspace forcings $\mathbb{A}:=\mathbb{A}(\mathbb{P}_0,\dot{\mathbb{P}}_1)$ and $\mathbb{B}:=\mathbb{A}(\mathbb{P}_0,\dot{\mathbb{R}})$, observing that the former is $\kappa^+$-strategically-closed. We will say that two conditions in $\mathbb{P}_0$ \textit{disagree} on an assertion if one forces that it is true and the other that it is false.

  \begin{claim}
    $$\exists p\in\mathbb{P}_0\;\exists\dot{u}\in\mathbb{B}\;\exists \dot{s}\in\mathbb{A}\;\forall\dot{r}\leq\dot{s}\;\exists j<\theta\;\exists\dot{q},\dot{q}'\leq\dot{r}:(p,\dot{u},\dot{q})\forces j\in\dot{z}, (p,\dot{u},\dot{q}')\forces j\notin\dot{z}.$$
  \end{claim}
  \begin{proof}
    Supposing the claim is false gives us
    \begin{equation} \label{1}
      \forall p\:\forall\dot{u}\:\forall\dot{s}\:\exists\dot{r}\leq\dot{s}\;\forall j<\theta\;\forall\dot{q},\dot{q}'\leq\dot{r}: \neg(p,\dot{u},\dot{q}),(p,\dot{u},\dot{q}') \mbox{ disagree on }j\in\dot{z}.
    \end{equation}
    Enumerate $\mathbb{P}_0$ as $\{p_{\alpha}\mid\alpha<\kappa\}$ (with repetition if necessary) and take $\{\dot{u}_{\alpha}\mid\alpha<\kappa\}\subseteq\mathbb{B}$ that are forced by $\mathbb{P}_0$ to enumerate $\dot{\mathbb{R}}$. Using (\ref{1}), and the $\kappa^+$-strategic-closure of $\mathbb{A}$ build inductively a descending sequence $\langle\dot{s}_{\kappa\alpha+\beta}\mid\alpha,\beta<\kappa\rangle$ such that
      $$\forall j<\theta\;\forall\dot{q},\dot{q}'\leq\dot{s}_{\kappa\alpha+\beta}:\neg(p_{\alpha},\dot{u}_{\beta},\dot{q}),(p_{\alpha},\dot{u}_{\beta},\dot{q}')\mbox{ disagree on }j\in\dot{z}.$$
    Take $\dot{s}$ a lower bound for this sequence, so
    \begin{equation} \label{2}
      \forall p\in\mathbb{P}_0\:\forall\beta<\kappa\:\forall j<\theta\:\forall\dot{q},\dot{q}'\leq\dot{s}: \neg(p,\dot{u}_{\beta},\dot{q}),(p,\dot{u}_{\beta},\dot{q})\mbox{ disagree on }j\in\dot{z}.
    \end{equation}
    Define a $\mathbb{P}_0*\dot{\mathbb{R}}$-name $\dot{w}\subseteq\theta$ by
      $$\dot{w}:=\{((p,\dot{u}),\check{\jmath})\mid \exists\dot{q}\leq\dot{s}:(p,\dot{u},\dot{q})\forces j\in\dot{z}\}.$$
    We were given that $\forces\dot{z}\notin V[\mathbb{P}_0\times\dot{\mathbb{R}}]$, so $\forces_{\mathbb{P}_0*(\dot{\mathbb{P}}_1\times\mathbb{R})}\dot{z}\neq\dot{w}$. Thus we can extend $(0,0,\dot{s})$ to some $(p,\dot{u},\dot{q})\in\mathbb{P}_0*(\dot{\mathbb{R}}\times\dot{\mathbb{P}}_1)$ that decides a point of disagreement, giving $(p,\dot{u},\dot{q})\leq(0,0,\dot{s})$ and $j<\theta$ such that $(p,\dot{u},\dot{q})\forces j\in\dot{z}-\dot{w}$ or $(p,\dot{u},\dot{q})\forces j\in\dot{w}-\dot{z}$. Modify $\dot{q}$ if necessary to ensure $\dot{q}\leq\dot{s}$ in the termspace forcing $\mathbb{A}$.

    One possibility is that $(p,\dot{u},\dot{q})\forces j\in\dot{z}$ and $(p,\dot{u})\forces j\notin\dot{w}$. But by the definition of $\dot{w}$ the first statement gives $(p,\dot{u})\forces j\in\dot{w}$, contradicting the second statement.

    Alternatively we have $(p,\dot{u},\dot{q})\forces j\notin\dot{z}$ and $(p,\dot{u})\forces j\in\dot{w}$. Then from the latter statement we can find $(p^*,\dot{u}^*)\leq(p,\dot{u})$ such that $((p^*,\dot{u}^*),\check{\jmath})\in\dot{w}$ so there is $\dot{q}'\leq\dot{s}$ such that $(p^*,\dot{u}^*,\dot{q}')\forces j\in\dot{z}$. Take $\beta<\kappa$ and $p^{**}\leq p^*$ such that $p^{**}\forces\dot{u}^*=\dot{u}_{\beta}$. Then we see that $(p^{**},\dot{u}_{\beta},\dot{q})$ and $(p^{**},\dot{u}_{\beta},\dot{q}')$ disagree on whether $j\in\dot{z}$, which contradicts (\ref{2}). This concludes the proof of the claim.
  \end{proof}

  Find $p$, $\dot{u}$ and $\dot{s}$ as given by the preceding claim. Temporarily fix any $\dot{r}\leq\dot{s}$ and any $\alpha<\kappa$. The claim gives us $j<\theta$ and $\dot{q},\dot{q}'\leq\dot{r}$ such that $(p,\dot{u},\dot{q})\forces j\in\dot{z}$ and $(p,\dot{u},\dot{q}')\forces j\notin\dot{z}$. Now we can use the maximum principle to construct a new name $\dot{q}''\leq\dot{r}$ such that $\forces_{\mathbb{P}_0*(\dot{\mathbb{R}}\times\dot{\mathbb{P}}_1)}(\alpha\in\dot{x}\rightarrow\dot{q}''=\dot{q})\wedge(\alpha\notin\dot{x}\rightarrow\dot{q}''=\dot{q}')$. This yields
    $$\forall\dot{r}\leq\dot{s}\;\forall\alpha<\kappa\;\exists j<\theta\;\exists\dot{q}''\leq\dot{r}: (p,\dot{u},\dot{q}'')\forces j\in\dot{z}\leftrightarrow\alpha\in\dot{x}.$$
  Next in $V$ use the strategic closure of $\mathbb{A}$ to build $\langle\dot{r}_{\alpha}\mid\alpha<\kappa\rangle$ a descending sequence below $\dot{s}$ in $\mathbb{A}$ according to the strategy, together with $\langle j_{\alpha}\mid\alpha<\kappa\rangle$ from $\theta$ such that
    $$(p,\dot{u},\dot{r}_{\alpha})\forces j_{\alpha}\in\dot{z}\leftrightarrow\alpha\in\dot{x}.$$
  The sequence has a lower bound $\dot{r}$, and then $(p,\dot{u},\dot{r})\forces\forall\alpha<\kappa:j_{\alpha}\in\dot{z}\leftrightarrow\alpha\in\dot{x}$, so $(p,\dot{u},\dot{r})\forces\dot{x}\in V[\dot{z}]$.
\end{proof}

Using the lemma we have the following theorem, which applies for example when $V=L$.

\begin{theorem}
  Assume GCH. Let $\mathbb{P}$ be a class Easton support iteration of $\Add(\kappa,1)$ at $\kappa$ regular with generic $G$ such that $V\subseteq gM^{V[G]}$. Then $gM^{V[G]}=V[G]$.
\end{theorem}

\begin{proof}
  We are given $x\in V[G]$ together with a generic ground $W$ and wish to show that $x\in W$. Take forcings $\mathbb{R}=\dot{\mathbb{R}}[G]\in V[G]$ and $\mathbb{S}\in W$ with generics $I$ and $J$ respectively such that $V[G][I]=W[J]$. Take $\kappa$ successor such that (without loss of generality) $x\subseteq\kappa$ and $\mathbb{R}\in H^{V[G]}_{\kappa}$. Split $\mathbb{P}$ as $\mathbb{P}_0*\dot{\mathbb{P}}_1$ where $\mathbb{P}_0$ is the iteration up to and including $\kappa$, and $\dot{\mathbb{P}}_1$ is the iteration beyond; note by GCH that $|\mathbb{P}_0|\leq\kappa$ and $\forces_{\mathbb{P}_0}\dot{\mathbb{P}}_1\;\kappa^+$-closed. Split $G$ correspondingly as $G_0*G_1$ and observe that $x\in V[G_0]$ so we can take $\dot{x}$ a $\mathbb{P}_0$-name for it. Also $\mathbb{R}\in V[G_0]$ so we can split the forcing $\mathbb{P}*\dot{\mathbb{R}}$ as $\mathbb{P}_0*(\dot{\mathbb{R}}\times\dot{\mathbb{P}}_1)$.

  Since $\mathbb{P}$ is a class forcing we can find $\dot{z}$ such that $\forces_{\mathbb{P}_0*(\dot{\mathbb{R}}\times\dot{\mathbb{P}}_1)}\dot{z}\in W - V[\mathbb{P}_0*\dot{\mathbb{R}}]$. Invoking \ref{EastonHelp} gives us that $\forces\dot{x}\in V[\dot{z}] \subseteq W$, using that $V\subseteq gM^{V[G]}$ so $V\subseteq W$.
\end{proof}

\section{Intersection of set-forcing extensions} \label{set_forcing_intersection}

We are interested in generalising our analysis of the generic mantle to a wider range of class forcing extensions. We are therefore interested in situations where we have a class forcing $\mathbb{P}$ with generic $G$ together with a generic ground of $V[G]$. To aid us we shall first consider the situation in which $\mathbb{P}$ is only a set forcing and (assuming $V\subseteq W$) the intermediate model theorem then gives us that $W$ is a set extension of $V$.

Here and for most of the remaining results we assume Global Choice. This ensures that there is a class of ordinals $X$ with $V=L[X]$. Then for any (set or class) generic extension $V[G]$ and class $Y\subseteq\OR$ thereof we can understand $V[Y]$ to mean $L[X, Y]$.

\begin{theorem} \label{main_sets}
  Assume Global Choice. Let $\mathbb{P}*\dot{\mathbb{R}}$ and $\mathbb{Q}*\dot{\mathbb{S}}$ be set forcings with respective generics $G*I$ and $H*J$ such that $V[G, I] = V[H, J]$ with $\forces_{\mathbb{P}}|\dot{\mathbb{R}}|<\kappa$ and $\forces_{\mathbb{Q}}|\dot{\mathbb{S}}|<\kappa$, where $\kappa$ is a cardinal in $V[G,I]$. Then there is an inner model $V\subseteq U\subseteq V[G]\cap V[H]$ such that $V[G]$ and $V[H]$ are both $\kappa$-cc forcing extensions of $U$ (via forcings that are quotients of $\ro(\mathbb{P})$ and $\ro(\mathbb{Q})$ respectively).
\end{theorem}

\begin{proof}
  We may assume that $\mathbb{P}$ and $\mathbb{Q}$ are forcings whose underlying sets are containing in the ordinals. Fix $\dot{G}$ a $\mathbb{Q}*\dot{\mathbb{S}}$-name such that $\dot{G}[H*J]=G$, and $\dot{H}$ a $\mathbb{P}*\dot{\mathbb{R}}$-name such that $\dot{H}[G*I]=H$. Since $\forces_{\mathbb{P}}|\dot{\mathbb{R}}|<\kappa$ we can choose a set $R$ of size less than $\kappa$ of $\mathbb{P}$-names for members of $\dot{\mathbb{R}}$ such that for any $\mathbb{P}$-name $\dot{r}$ for a member of $\dot{\mathbb{R}}$ there are densely many $p\in\mathbb{P}$ such that there exists an $\dot{r}'\in R$ for which $p\forces\dot{r}=\dot{r}'$. Choose a similar $S\subseteq\dot{\mathbb{S}}$ also of size less than $\kappa$.

  For $p\in\mathbb{P}$ define
    $$S^1(p):=\{q\in\mathbb{Q}\mid\exists \dot{r}\in R:(p,\dot{r})\forces q\in\dot{H}\}$$
  and for $A\subseteq\mathbb{P}$ define $S^1(A):=\bigcup_{p\in A}S^1(p)$. Similarly for $q\in\mathbb{Q}$ define
    $$T^1(q):=\{p\in\mathbb{P}\mid\exists\dot{s}\in S:(q,s)\forces p\in\dot{G}\}$$
  and for $B\subseteq\mathbb{Q}$ define $T^1(B):=\bigcup_{q\in B}T^1(q)$. Now construct inductively $S^{n+1}(p) := \bigcup_{q\in S^1(p)}\bigcup_{p'\in T^1(q)}S^n(p')$ and $S^{n+1}(A) := \bigcup_{q\in S^1(A)}\bigcup_{p'\in T^1(q)}S^n(p')$, followed by $S(p) := \bigcup_{n<\omega}S^n(p)$ and $S(A) := \bigcup_{n<\omega}S^n(A)$. Make equivalent definitions for $T(q)$ and $T(B)$. All of this can be done in $V$.

  For any $q\in H$ we can take $(p,\dot{r})\in G*I$ such that $(p,\dot{r})\forces q\in\dot{H}$. Then by density there will be a $p'\leq p$ with $p'\in G$ together with an $\dot{r}'\in R$ such that $p'\forces\dot{r}=\dot{r}'$. Thus $(p',\dot{r}')$ is below $(p,\dot{r})$ and also forces that $q$ is in $\dot{H}$, so $q\in S^1(p')\subseteq S^1(G)$. Therefore $H\subseteq S^1(G)$ so $S^1(T(H))\subseteq S^1(T(S^1(G))) = S(G)$, and likewise $G\subseteq T^1(H)$ so $S(G)\subseteq S(T^1(H)) = S^1(T(H))$. This gives us that $S(G) = S^1(T(H))$ and so it is a member of both $V[G]$ and $V[H]$.

  For convenience define $E(G)$ to be $T^1(S(G))$ and note that this is contained in $\mathbb{P}$ and is interdefinable with $S(G)$. Furthermore for $\dot{r}\in\dot{\mathbb{R}}$ we have that $\{q\in\mathbb{Q}\mid\exists p\in G:(p,\dot{r})\forces q\in\dot{H}\}$ is a filter, so $S^1(G)$ and hence $E(G)$ is (in $V[G]$) a union of $<\kappa$-many filters. The inner model $U:=V[E(G)]$ will be our candidate. We have $V\subseteq V[E(G)]\subseteq V[G]$ so by the intermediate model theorem there is a complete subalgebra $\mathbb{A}$ of $\mathbb{P}$ such that $V[E(G)] = V[\mathbb{A}\cap G]$, together with $\mathbb{C}$ the quotient forcing $\mathbb{P}/(\mathbb{A}\cap G)$.

  \begin{claim}
    $\mathbb{C}$ has the $\kappa$-cc.
  \end{claim}
  \begin{proof}
    Let $\dot{\Gamma}$ be the standard $\mathbb{P}$-name for the generic of $\mathbb{P}$ and $\dot{\mathcal{E}}$ the $\mathbb{A}$-name for $E(\dot{\Gamma})$; this will exist because $\mathbb{A}$ is generated by the conditions $\{[\![p\in E(\dot{\Gamma})]\!]\mid p\in\mathbb{P}\}$. Take $\dot{\mathbb{C}}$ an $\mathbb{A}$-name for $\mathbb{C}$, so by \cite[Lemma 25.5]{old_jech} there are conditions below which $\mathbb{A}*\dot{\mathbb{C}}$ is isomorphic to $\mathbb{P}$; we can assume these conditions are trivial.

    Given any $\mathbb{A}$-name $\dot{x}$ for a member of $\mathbb{C}$, define $p_{\dot{x}}:=(1, \dot{x})\in\mathbb{A}*\dot{\mathbb{C}}$, which we can regard as a member of $\mathbb{P}$. Take some $\mathbb{C}$-generic $K$ that contains $\dot{x}[G\cap\mathbb{A}]$ and define $\bar{G}:=(\mathbb{A}\cap G)*K$, a $\mathbb{P}$-generic. Now $p_{\dot{x}}$ is in $\bar{G}$ and hence $E(\bar{G})$. Therefore $[\![ p_{\dot{x}}\in \dot{\mathcal{E}}]\!]$ is in $\bar{G}\cap\mathbb{A}$ which equals $G\cap\mathbb{A}$, from which $p_{\dot{x}}\in E(G)$.

    We have in $V[G]$ that $E(G)$ was formed as a union of $<\kappa$-many filters, into one of which each $p_{\dot{x}}$ must fall. Suppose there is $\dot{X}$ an $\mathbb{A}$-name for an antichain in $\mathbb{C}$ of size $\kappa$. For any $\dot{x}$ and $\dot{y}$ such that $\forces_{\mathbb{A}}\dot{x},\dot{y}\in\dot{X}$ we have $\forces_{\mathbb{A}}\dot{x}\perp\dot{y}$ so $p_{\dot{x}}\perp p_{\dot{y}}$, which means that $p_{\dot{x}}$ and $p_{\dot{y}}$ are members of different filters whose union is $E(G)$. But there are $\kappa$-many of these $\dot{x}$ and $<\kappa$-many filters, and $\kappa$ remains a cardinal in $V[G]$. This is a contradiction.
  \end{proof}
  By a similar analysis the forcing between $U$ and $V[H]$ also has the $\kappa$-cc.
\end{proof}

We immediately obtain the following:

\begin{corollary}
  Assume Global Choice. Let $\mathbb{P}*\dot{\mathbb{R}}$ and $\mathbb{Q}*\dot{\mathbb{S}}$ be set forcings with respective generics $G*I$ and $H*J$ such that $V[G, I] = V[H, J]$ and $V[G]\cap V[H]=V$ with $\forces_{\mathbb{P}}|\dot{\mathbb{R}}|<\kappa$ and $\forces_{\mathbb{Q}}|\dot{\mathbb{S}}|<\kappa$, where $\kappa$ is a cardinal in $V[G,I]$. Then there is $p\in G$ such that $\mathbb{P}_p$ is $\kappa$-cc.
\end{corollary}

\begin{proof}
  Theorem \ref{main_sets} gives us a common ground $U$ of $V[G]$ and $V[H]$ that contains $V$, so it must be $V$ itself. Therefore there is a $\kappa$-cc forcing $\mathbb{M}$ with generic $E$ such that $V[E]=V[G]$. Then by \cite[Lemma 25.5]{old_jech} there are $p\in G$ and $m\in E$ such that $\mathbb{P}_p$ equals $\mathbb{M}_m$ and so is $\kappa$-cc.
\end{proof}

To see that this result for chain condition is best possible, consider $\mathbb{P}=\mathbb{Q}=\mathbb{R}=\mathbb{S}=\Add(\lambda, 1)$ with mutually generic filters $G=J$ and $H=I$. Assuming $2^{<\lambda}=\lambda$ the forcings all have size $\lambda$ but $\mathbb{P}$ does not have the $\lambda$-cc.

In the proof of \ref{main_sets} we were not able to simply use the intersection of $V$ and its generic ground as their common ground. The following example that develops a method of Woodin illustrates that this is an unavoidable complexity, because the intersection may not be a model of ZF. An alternative proof (with assuming a strong inaccessible) can be found in \cite[Theorem 34]{FHR}.

\begin{proposition} \label{intersection_not_model_of_ZFC}
  Let $M$ be a countable model of ZFC that believes $V=L$ and that there exists a strongly inaccessible cardinal. Then there exist generics $G$ and $H$ over $M$ such that $M[G]\cap M[H]\nvDash\ZF$.
\end{proposition}

\begin{proof}
  In $M$ fix $\theta$ strongly inaccessible and $\chi$ much larger than $\theta$. Take $x\subseteq\theta$ that encodes a $\Coll(\theta, \chi)^M$-generic over $M$. First, working in $M[x]$, we shall construct $G$ and $H$ which are $\Add(\theta, 1)^M$-generics over $M$ and induce generic functions $g, h:\theta\rightarrow 2$ such that $g^{-1}``\{1\}\cap h^{-1}``\{1\}$ encodes $x$. Enumerate the dense open subsets of $\Add(\theta, 1)^M$ from $M$ ,as $\{D_i\mid i<\theta\}$. We will inductively build $p_i, q_i\in\Add(\theta, 1)^M$ for $i<\theta$ such that $\dom p_i = \dom q_i$ and then take $g = \bigcup p_i$ and $h = \bigcup q_i$. Start with $p_0 = q_0 =\{\}$ and take $p_i = \bigcup_{j<i}p_j$ and $q_i=\bigcup_{j<i}q_j$ at $i$ limit. Given $p_i$ and $q_i$ find $p_i'\leq p_i$ whose domain is even such that $p_i'\in D_i$ and then construct $q_i'\leq q_i$ by defining $q_i(\delta)=0$ for $\delta\in\dom p_i'-\dom q_i$. Similarly find $q_i''\leq q_i'$ such that $q_i''\in D_i$ and $\dom q_i''$ is even, and construct $p_i''\leq p_i'$ by taking $p_i''(\delta)=0$ for $\delta\in\dom q_i''-p_i'$. Observe that $p_i''^{-1}``\{1\}\cap q_i''^{-1}``\{1\} = p_i^{-1}``\{1\}\cap q_i^{-1}``\{1\}$. Now define $\xi:=\dom p_i''=\dom q_i''$ and extend $p_i''$ to $p_{i+1}$ and $q_i''$ to $q_{i+1}$ by defining
  \begin{equation}
    (p_{i+1}(\xi),p_{i+1}(\xi+1))=(q_{i+1}(\xi),q_{i+1}(\xi+1))= \begin{cases}
      (1, 0) \quad\mbox{ if } i\in x\\
      (0, 1) \quad\mbox{ if } i\notin x.\\
    \end{cases}
  \end{equation}

  Define subsets of $\theta$, $\bar{g}:=g^{-1}``\{1\}$ and $\bar{h}:=h^{-1}``\{1\}$. Since we met all of the dense sets from $M$ the associated filters $G$ and $H$ must be generic over $M$, but we have ensured that $i\in x$ iff the $i^{th}$ element of $\bar{g}\cap \bar{h}$ is even, so $\bar{g}\cap \bar{h}$ codes $x$.

  Now take $F$ that is $(\prod_{\alpha<\theta}\Add(\aleph_{\alpha+1}, 1))^{M}$-generic over $M[x]$ where we use an Easton support product. In $M[x]$ there are also the sets $\bar{g}$ and $\bar{h}$ and hence the forcings $(\prod_{\alpha\in \bar{g}}\Add(\aleph_{\alpha+1}, 1))^{M[G]}$ and $(\prod_{\alpha\in \bar{h}}\Add(\aleph_{\alpha+1}, 1))^{M[H]}$. There are natural projections $\pi_0:(\prod_{\alpha<\theta}\Add(\aleph_{\alpha+1}, 1))^M\rightarrow (\prod_{\alpha\in\bar{g}}\Add(\aleph_{\alpha+1}, 1))^{M[G]}$ and $\pi_1:(\prod_{\alpha<\theta}\Add(\aleph_{\alpha+1}, 1))^M\rightarrow (\prod_{\alpha\in\bar{h}}\Add(\aleph_{\alpha+1}, 1))^{M[H]}$ from which we obtain generics $\pi_0``F$ and $\pi_1``F$. Now $\theta$ is strongly inaccessible in $M$ so the constructions of the latter two forcings are unaffected by $\theta$-closed forcing, and they will also be Easton-support products of $\Add(\aleph_{\alpha+1},1)^M$ forcings in $M[G]$ and $M[H]$ respectively. This means we can form the generic extension of $M[G]$ by $(\prod_{\alpha\in \bar{g}}\Add(\aleph_{\alpha+1}, 1))^{M[G]}$ to get $M[G][\pi_0``F]$, and of $M[H]$ by $(\prod_{\alpha\in\bar{h}}\Add(\aleph_{\alpha+1}, 1))^{M[H]}$ to get $M[H][\pi_1``F]$.

  Let us now consider $M[G][\pi_0``F]\cap M[H][\pi_1``F]$. For $\alpha\in\bar{g}\cap\bar{h}$ the generic $F(\alpha)$ for $\Add(\aleph_{\alpha+1},1)^M$ will be a fresh subset of $\aleph_{\alpha+1}^M$ (i.e. a subset all of whose initial segments are in the ground model) that is a member of $M[G][\pi_0``F]\cap M[H][\pi_1``F]$. Conversely if $\alpha\notin\bar{g}\cap\bar{h}$ then say $\alpha\notin\bar{g}$, so in $M[G][\pi_0``F]$ there will be no fresh subset of $\aleph_{\alpha+1}^M$. Therefore $M[G][\pi_0``F]\cap M[H][\pi_1``F]$ contains a fresh subset  of $\aleph_{\alpha+1}^M$ if and only if $\alpha$ is a member of both $\bar{g}$ and $\bar{h}$. If this intersection were a model of ZF then it would be able to form the set of such $\alpha$, and so it would contain $\bar{g}\cap\bar{h}$ and from this $x$. This is impossible because the forcing that gives $G*\pi_0``F$ is too small to have collapsed $\chi$.
\end{proof}

In this example the two models $M[G][\pi_0``F]$ and $M[H][\pi_1``F]$ are re-united by set forcing in the model $M[x][F]$. Take $E$ a generic for some $\chi^+$-closed class forcing in $M$. Then we have a class forcing extension of $M$ given by $M[G][\pi_0``F][E]=M[E][G][\pi_0``F]$ and a generic ground $M[E][H][\pi_1``F]$ thereof such that their intersection is not a model of ZF.

The importance of $\kappa$ remaining a cardinal of $V[G]$ in \ref{main_sets} is illustrated by the following, where the role of $\kappa$ is played by $\lambda^+$.

\begin{proposition} \label{coll_add_example}
  Let $M$ be a countable model of ZFC that believes $\lambda$ to be regular such that $2^{<\lambda}=\lambda$. Then there exist forcings $\mathbb{P}\times\mathbb{R}$ and $\mathbb{Q}\times\mathbb{S}$ in $M$ with generics $G\times I$ and $H\times J$ respectively such that $M[G, I]=M[H, J]$ and $M[G]\cap M[H]=M$ but $|\mathbb{R}|^M=|\mathbb{S}|^M=\lambda$ and for all $p\in\mathbb{P}$, $\mathbb{P}_p$ is not $\lambda^+$-cc in $M$.
\end{proposition}

\begin{proof}
  $\mathbb{P}$ and $\mathbb{Q}$ will be $\Coll^*(\lambda,\lambda^+)^M$, which we define to be the usual $\Coll(\lambda,\lambda^+)^M$ with the additional stipulation that the conditions must be injective partial functions; this means that the generic function it adds will be a bijection. Similarly $\mathbb{R}$ and $\mathbb{S}$ will be $\Add^*(\lambda, 1)^M$, by which we mean the forcing whose conditions are injective partial functions from $\lambda$ to $\lambda$; this is equivalent to the usual $\Add(\lambda, 1)^M$ and will create a bijective generic function from $\lambda$ to $\lambda$.

  Take $G$ and $H$ to be mutually generic filters added by $\mathbb{P}$ and $\mathbb{Q}$, and $g, h:\lambda\rightarrow\lambda^+$ the resultant bijections. Define $i=j=g^{-1}\circ h$ and let $I=J$ be the filters associated with $i$ and $j$ respectively. It is clear that $M[G]\cap M[H]=M$ and $M[G, I]=M[G\times H]=M[H, J]$ so we just need to check that $I$ is in fact generic for $\Add^*(\lambda, 1)$ over $M[G]$ (and similarly for $J$ over $M[H]$).

  Given $D\in M$ a dense subset of $\Add^*(\lambda, 1)^M$ that belongs to $M[G]$ we define $E$ to be $\{q\in\Coll^*(\lambda,\lambda^+)\mid g^{-1}\circ q\in D\}$, also in $M[G]$. Given any $q\in\Coll^*(\lambda,\lambda^+)^M$ then $g^{-1}\circ q\in\Add^*(\lambda,1)$ so there is an $r\in D$ with $r\leq g^{-1}\circ q$ and then $g\circ r\leq q$ is in $E$. Thus $E$ is dense in $\Coll^*(\lambda,\lambda^+)^M$ and we can find $q\in H\cap E$ giving $g^{-1}\circ q\in I\cap D$.
\end{proof}

In the construction used in this proof we could take an antichain $X\subseteq\Coll^*(\lambda,\lambda^+)^M$ of size $\lambda^+$ and apply \ref{main_sets} to it to obtain in $M[G]$ an injection from $X$ into the set of $\lambda$-many filters whose union forms $E(G)$, but this is not a contradiction because $\lambda^+$ has been collapsed by $\mathbb{P}$.

It is well known that for mutually generic filters $G$ and $H$ we must have $V[G]\cap V[H]=V$. We conclude this section by showing that the converse is false.

\begin{proposition}
  Let $M$ be a countable model of ZFC. Then there exist generic filters $G$ and $H$ such that $M[G]\cap M[H]=M$ but $G$ and $H$ are not mutually generic.
\end{proposition}

\begin{proof}
  Fix $\theta$ large and $x\subseteq\omega$ encoding a $\Coll(\omega,\theta)^M$-generic over $M$. Let $\{D_i\mid i<\omega\}$ be an enumeration of all dense open subsets of $\Add(\omega,1)$ from $M$ and $\{(\dot{a}_i,\dot{b}_i)\mid i<\omega\}$ an enumeration of pairs of $\Add(\omega,1)$-names from $M$ that are forced to be new subsets of $\OR$. We will construct $\Add(\omega, 1)$ generics $G$ and $H$ over $M$ by inductively defining $p_i, q_i$ for $i<\omega$ and then taking $g=\bigcup p_i$ and $h=\bigcup q_i$. We have already seen in \ref{intersection_not_model_of_ZFC} how to do this so that $g$ and $h$ induce generic filters $G$ and $H$ in such a way that $\bar{g}\cap\bar{h}$ codes $x$, where $\bar{g}:=g^{-1}``\{1\}$ and $\bar{h}:=h^{-1}``\{1\}$. This means $G$ and $H$ cannot be mutually generic because if they were $G\times H$ would be generic for $\Add(\omega,1)\times\Add(\omega,1)$ whilst encoding a $\Coll(\omega,\theta)^M$-generic.

  To obtain $M[G]\cap M[H]=M$ we add an additional step to the construction of $(p_{i+1},q_{i+1})$. Given $(p_i', q_i')$ which have met $D_i$ and recorded whether $i\in x$ we want to ensure that $\dot{a}_i[G]\neq\dot{b}_i[H]$, so we seek $p_{i+1}\leq p_i'$ and $q_{i+1}\leq q_i'$ and $\eta\in\OR$ such that $p_{i+1}\forces \eta\in\dot{a}_i$ and $q_{i+1}\forces \eta\notin\dot{b}_i$ (or vice versa). We also need $p_{i+1}^{-1}``\{1\}\cap q_{i+1}^{-1}``\{1\}=p_i'^{-1}``\{1\}\cap q_i'^{-1}``\{1\}$ so as not to disrupt the coding of $x$. For $s\in\Add(\omega,1)$ and $n<\omega$ we will write $s\frown\langle 0\rangle^n$ for the extension $s'$ of $s$ with $\dom s'=\dom s+n$ and $s'(m)=0$ for $m\in\dom s'-\dom s$.

  Since $\dot{b}_i$ is a name for a new subset of $\OR$ we can fix $\eta$ such that $q_i'$ does not decide whether $\eta\in\dot{b}_i$. Observe that it is impossible for there to be $k$ and $l$ such that $p_i'\frown\langle 0\rangle^{k}\forces \eta\in\dot{a}_i$ and $p_i'\frown\langle 0\rangle^{l}\forces \eta\notin\dot{a}_i$ since such conditions would be compatible. So without loss of generality say there is no $l<\omega$ with $p_i'\forces \eta\notin\dot{a}_i$. Take $q_i''\leq q_i'$ such that $q_i''\forces \eta\notin\dot{b}_i$ and define $p_i'':=p_i'\frown\langle 0\rangle^{\lh q_i''-\lh q_i'}$. Since $p_i''$ does not force $\eta\notin\dot{a}_i$ we can take $p_{i+1}\leq p_i''$ such that $p_{i+1}\forces \eta\in\dot{a}_i$ and define $q_{i+1}:=q_i''\frown\langle 0\rangle^{\lh p_{i+1}-\lh p_i''}$.

  This has given us that all sets of ordinals from $M[G]\cap M[H]$ are in $M$, and it follows that $M[G]\cap M[H]=M$.
\end{proof}

\section{Intersections with generic grounds I} \label{intersections_part_1}

We first recall some key definitions for class forcing.

\begin{definition}
  Let $\mathbb{P}$ be a class partial order. Then a class $D\subseteq\mathbb{P}$ is \textit{predense} below $p\in\mathbb{P}$ if every $q\leq p$ is compatible with an element of $D$. The forcing $\mathbb{P}$ is \textit{pre-tame} if for every $p\in\mathbb{P}$, $\lambda\in\OR$, and $V$-definable sequence of classes $\langle D_i\mid i<\lambda\rangle$ that are each pre-dense below $p$, there are $q\leq p$ and $\langle d_i\mid i<\lambda\rangle\in V$ such that $d_i\subseteq D_i$ and each $d_i$ is pre-dense below $q$.

  The forcing $\mathbb{P}$ is \textit{tame} if it is both pre-tame and forces the power set axiom to hold.

  The forcing $\mathbb{P}$ is \textit{ZFC-preserving} if for all $\mathbb{P}$-generics $G$ the model $\langle V[G], G\rangle$, in the language of set theory together with a unary predicate, satisfies ZFC, where the axioms of separation and replacement are allowed to make definitions with respect to the unary predicate.
\end{definition}

The fundamental theorems of class forcing show that the forcing relation $\forces$ is definable in $V$, that a $\mathbb{P}$-extension $V[G]$ will satisfy $\varphi(\dot{a}[G])$ iff there is $p\in G$ such that $p\forces\varphi(\dot{a})$, and that a class forcing $\mathbb{P}$ is pre-tame if and only if it preserves ZFC - Power Set (in the extended language with a unary predicate for $G$); for proofs see \cite[Chapter 8, 2.5 and 2.11]{FK}. Thus a class forcing is tame if and only if it is ZFC-preserving. The notion of ZFC-preservation is not first-order expressible but this equivalence allows us to instead speak of tameness, which constitutes a first-order axiom scheme.

\begin{definition}
  We will say a class forcing $\mathbb{P}$ is \textit{forwards} if it is tame and is the result of a an iteration $\langle\mathbb{P}_{\alpha},\dot{\mathbb{Q}}_{\alpha}\mid\alpha\in\OR\rangle$ of set forcing, such that for all cardinals $\kappa$ there is an $\alpha$ such that $\mathbb{P}_{\alpha}$ forces the tail forcing from $\dot{\mathbb{Q}}_{\alpha}$ onwards to be $\kappa$-distributive.

  We note that most interesting class forcings are of this kind but not, for example, Jensen's coding of the universe into a real. The increasing distributivity occurs because any forcing that preserves ZFC can only add set-many subsets of any given ordinal.
\end{definition}

We will need to following result from \cite[Proposition 1.4]{Unger}.

\begin{lemma} \label{kappaCCSubsets}
  Let $\mathbb{P}$ be a non-trivial $\kappa$-cc forcing. Then it adds a new subset of $\kappa$.
\end{lemma}

\begin{proof}
  Let $\dot{B}$ be forced to be a fresh subset of $\lambda$ added by $\mathbb{P}$. Fix $\langle\lambda_i\mid i<\cf(\lambda)\rangle$ cofinal in $\lambda$ and let $T$ be the tree whose $i^{th}$ level consists of all possible values of $\dot{B}\cap\lambda_i$. By the $\kappa$-cc all levels of $T$ have size less than $\kappa$. We have two cases.
  \begin{itemize}
    \item If $\cf(\lambda)\leq\kappa$ then $|T|$ has size at most $\kappa$. The forcing adds a branch through $T$ that determines the realisation of $\dot{B}$, and this branch will be a new subset of $T$.

    \item Otherwise $\cf(\lambda)>\kappa$. For any $i<\cf(\lambda)$ of cofinality $\kappa$ and $A, B\in \Lev_i(T)$ with $A\neq B$ we can find $j<i$ where $A\cap\lambda_j\neq B\cap\lambda_j$. But there are $<\kappa$-many such pairs on level $i$, so we can find a single $f(i)<i$ such that for all $A, B\in\Lev_i(T)$ with $A\neq B$ we have $A\cap\lambda_{f(i)}\neq B\cap\lambda_{f(i)}$. Now $f$ is a regressive function on the stationary set $\cf(\lambda)\cap\cof(\kappa)$ so there is a stationary subset $S$ and a fixed $k<\cf(\lambda)$ with $f(i) = k$ for all $i\in S$. Take $A\in\Lev_k(T)$. Then in any $i\in S$ there is a unique extension of $A$ to $\Lev_i(T)$, so $A$ determines the value of $\dot{B}$, which contradicts its being forced to be new.
  \end{itemize}
\end{proof}

We also need a more careful analysis of possible extensions of the universe by definability with respect to sets of ordinals.

Suppose we have $V\subseteq W$ models of ZFC, $\theta$ such that $\beth^V_{\theta}=\theta$, and $X\subseteq\theta$ a member of $W$. Then we can fix $A\subseteq\theta$ in $V$ such that $L_{\theta}[A]=V_{\theta}$ and interpret $V_{\theta}[X]$ as $L_{\theta}[A, X]$. We pause to show that this will be independent of the choice of $A$. Given another $A'\subseteq\theta$ such that $L_{\theta}[A']=V_{\theta}$ then for any $c\in L_{\theta}[A,X]$ we have $c\in L_{\gamma}[A,X]=L_{\gamma}[A\cap\gamma,X]$ for some $\gamma<\theta$; since $A\cap\gamma\in V_{\theta}=L_{\theta}[A']$ this will give $c\in L_{\theta}[A',X]$. Together with a similar argument in the other direction this yields $L_{\theta}[A,X]=L_{\theta}[A',X]$.

We can also define $V[X]$ to be $V(\{X\})$ as this will automatically be a model of the Axiom of Choice and so of ZFC.

\begin{lemma} \label{reflection_helper_2}
  Let $V\subseteq W$ be models of ZFC, $\theta$ strongly inaccessible in $W$, and $X\subseteq\theta$ a member of $W$. Then $V_{\theta}[X]=V[X]_{\theta}$.
\end{lemma}

\begin{proof}
  The left-to-right containment is clear so we consider the converse. Given $x\in V[X]_{\theta}$, take a strong limit cardinal $\beta>\theta$ such that $x\in V_{\beta}[X]_{\theta}$ and then choose $\chi>>\beta$. Use the strong inaccessibility of $\theta$ to build $M\prec V[X]_{\chi}$ such that $\TC(\{x\})\subseteq M$, $\beta\in M$, $\delta:=M\cap\theta\in\theta$ and $|M|<\theta$. Collapse $\pi:M\rightarrow\bar{M}$ to a transitive model; this gives $\pi(\theta)=\delta$, $\pi(X)=X\cap\delta$ and $\pi(x)=x$. We began with $M\models x\in V_{\beta}[X]_{\theta}$, so we obtain $\bar{M}\models x\in V_{\pi(\beta)}[X\cap\delta]_{\delta}$. But $\pi(\beta)<\theta$ so by absoluteness $x\in V_{\theta}[X\cap\delta]_{\delta}\subseteq V_{\theta}[X]$.
\end{proof}

We recall the following definition, introduced by Hamkins in \cite{hamkins_extensions}.

\begin{definition}
  Let $V\subseteq W$ be models of ZFC and $\lambda$ a cardinal of $W$. We say $V$ has the \textit{$\lambda$-covering} property in $W$ if for every $\sigma\subseteq V$ with $\sigma\in V$ and $|\sigma|^W<\lambda$ there is $\tau\in V$ with $\sigma\subseteq\tau$ and $|\tau|^W<\lambda$. It has the \textit{$\lambda$-approximation} property in $W$, if for any increasing sequence $\langle\tau_{\alpha}\mid\alpha<\delta\rangle\in W$ where $\cf^W(\delta)\geq\lambda$ and each $\tau_{\alpha}\in V$ we have $\bigcup\tau_{\alpha}\in V$.
\end{definition}

It is key to the study of grounds that a universe $V$ has $\lambda$-covering and $\lambda$-approximation within any forcing extension by a forcing of size less than $\lambda$. Furthermore, Hamkins showed that given any universe $V$ and any $S\subseteq\mathcal{P}(\lambda)$ there is at most one inner model $W$ that satisfies $\lambda$-covering and $\lambda$-approximation and has $\mathcal{P}(\lambda)^W=S$. For proofs of these results see \cite[Lemma 9]{FHR}, \cite{Reitz} or \cite[Theorem 5]{HOD}. (This is what renders the grounds of a universe uniformly definable, and so makes the mantle a class, as discussed in \cite{FHR}.)

We can now use our earlier work with set forcing extensions to analyse the intersection of a class forcing extension and one of its generic grounds, in the presence of a sufficient number of weakly compact cardinals.

\begin{theorem} \label{main_class}
  Assume Global Choice. Let $\mathbb{P}$ be a forwards class forcing with generic $G$, formed from $\langle\mathbb{P}_{\alpha}\mid\alpha<\OR\rangle$ with individual generics $G_{\alpha}$. Let there be a stationary class of ordinals $\alpha$ such that the initial segment $\mathbb{P}_{\alpha}$ of $\mathbb{P}$ is formed by a direct limit of the preceding forcings, and such that $\alpha$ is weakly compact in $V[G_{\alpha}]$. Let $W$ be a generic ground of $V[G]$ via forcing of size less than $\kappa$, such that $V\subseteq W$. Then there is a common ground of $V[G]$ and $W$ via forcings with the $\kappa$-cc.
\end{theorem}

\begin{proof}
  Using Global Choice we may assume that the underlying set of $\mathbb{P}$ is contained in the ordinals, so $G\subseteq\OR$. Say $W$ is a generic ground of $V[G]$ via $\mathbb{R}\in V[G]$ with generic $I$ and $\mathbb{S}\in W$ with generic $J$, so $V[G][I]=W[J]$ with $|\mathbb{R}|,|\mathbb{S}|<\kappa$.

  Suppose the result does not hold. Observe that $W$ is definable in $V[G][I]$ as the unique inner model whose powerset of $\kappa$ is $\mathcal{P}(\kappa)^W$ and for which $\kappa$-covering and $\kappa$-approximation hold, so the assertion that no such common ground exists is first-order definable; call this assertion $\varphi(\mathcal{P}(\kappa)^W)$. We may assume that in $V[G]$ it is forced by $\mathbb{R}$ that $\varphi(\mathcal{P}(\kappa)^W)$ holds. Reflecting down the $\langle V[G]_{\alpha}\mid\alpha\in\OR\rangle$ hierarchy, there is a club of $\theta>>\kappa$ such that
    $$V[G]_{\theta}\models\forces_{\mathbb{R}}\varphi(\mathcal{P}(\kappa)^W).$$
  There is a stationary class of $\theta$ such that $\mathbb{P}_{\theta}$ is formed as a direct limit and $\theta$ is weakly compact in $V[G_{\theta}]$. Additionally there is a club of $\theta$ such that the underlying set of $\bigcup_{\alpha<\theta}\mathbb{P}_{\alpha}$ is $\theta$, which if $\mathbb{P}_{\theta}$ is a direct limit means that the underlying set of $\mathbb{P}_{\theta}$ is $\theta$ and $G_{\theta}=G\cap\theta$. Finally for all $\alpha$ the increasing distributivity of $\mathbb{P}$ gives us $\beta$ such that $\forces_{\mathbb{P}}V[\dot{G}]_{\alpha}\subseteq V_{\beta}[\dot{G}]$; this yields (definably in $V$) a club of $\theta$ such that $V[G]_{\theta}=V_{\theta}[G]$. Fix a $\theta$ that has all of these properties.

  The filter $J$ will still be generic over $W_{\theta}$ and, since $|\mathbb{S}|<\theta$, every member of $V[G, I]_{\theta}$ will have a $\mathbb{S}$-name in $W_{\theta}$. This permits the forcing extension $W_{\theta}[J]=V[G,I]_{\theta}$, which means that $W_{\theta}$ is a `class' from the point of view of $V[G, I]_{\theta}$. Specifically it will be the `inner model' of $V[G, I]_{\theta}$ satisfying $\kappa$-covering and $\kappa$-approximation and with the correct $\mathcal{P}(\kappa)$ that we have assumed shares no common ground with $V[G]_{\theta}$ via $\kappa$-cc forcing.

  By the choice of $\theta$ we that have $V[G]_{\theta}$ equals $V_{\theta}[G]$, which clearly equals $V_{\theta}[G\cap\theta]$, which by \ref{reflection_helper_2} equals $V[G\cap\theta]_{\theta}$. Now $|\mathbb{R}|<\theta$ so all sets of size less than $\theta$ added by $\mathbb{R}$ will have names of size less than $\theta$; this allows us to deduce $V[G,I]_{\theta}=V[G]_{\theta}[I]=V[G\cap\theta]_{\theta}[I]=V[G\cap\theta, I]_{\theta}$.

  We have therefore managed to cut down the situation where $W$ was a generic ground of $V[G]$ to one where $W_{\theta}$ is a generic ground of $V[G\cap\theta]_{\theta}$, and by assumption the two have no common ground via $\kappa$-cc forcing. Our aim is now to extend the height of the universes in question back up to $\OR$ while leaving the forcing $\mathbb{P}_{\theta}$ unchanged; then we will be able to apply our understanding of this situation for set forcings to find a common ground and obtain a contradiction. The main difficulty in doing so is extending $W_{\theta}$, since the obvious $V(W_{\theta})$ may lack both a well-ordering of $W_{\theta}$ and an $\mathbb{S}$-name for $G\cap\theta$.

  Since $W_{\theta}$ is definable in $V[G,I]_{\theta}=V[G\cap\theta,I]_{\theta}$ we get that $W_{\theta}\in V[G\cap\theta,I]$. Now $\mathbb{P}_{\theta}$ has underlying set $\theta$ so for any $\alpha\in(\kappa,\theta)$ we may regard $\mathbb{S}$-names for $G\cap\alpha$ as subsets of $\alpha$, and therefore also as members of $W_{\theta}$. Form the tree $T\subseteq W_{\theta}$ whose $\alpha$-level consists of such $\mathbb{S}$-names for $G\cap\alpha$. Now $\theta$ is weakly compact in $V[G\cap\theta]$ and $\mathbb{R}$ is a small forcing so $\theta$ remains weakly compact in $V[G\cap\theta, I]$; since $T\in V[G\cap\theta,I]$ there is a cofinal branch of $T$ in $V[G\cap\theta,I]$. We apply the $\kappa$-approximation between $W$ and $V[G,I]$ to see that this branch will also be a member of $W$.

  Similarly, form the tree $T'\subseteq W_{\theta}$ such that for all $\alpha$ strong limit its $\alpha$-level consists of $x\subseteq\alpha$ such that $W_{\alpha}\subseteq V[x]$. We can again take a branch in this tree that will be a member of both $V[G\cap\theta,I]$ and $W$. Combine both branches into some $B\subseteq\theta$, so that $W_{\theta}\subseteq V[B]$ and $V[B]$ contains an $\mathbb{S}$-name for $G\cap\theta$. Since $B\in W$ we have in fact $W_{\theta}=V[B]_{\theta}$. Also $V[B]\subseteq V[G\cap\theta,I]$, and $V[B]$ contains $\mathbb{S}$-names for both $G\cap\theta$ and $I$ (since the latter name is of size less that $\theta$), so $V[G\cap\theta,I]$ is a forcing extension of $V[B]$ via $\mathbb{S}$ and $J$.

  Since everything has now been reduced to set forcing, we can apply the intermediate model theorem to $V\subseteq V[B]\subseteq V[G\cap\theta,I]$ to get a forcing $\mathbb{Q}$ with generic $H$ such that $V[B]=V[H]$. We can then use \ref{main_sets} to obtain a common ground $U$ of $V[G\cap\theta]$ and $V[B]$, say via $\kappa$-cc forcings $\mathbb{M}$ and $\mathbb{N}$ which are quotients of $\ro(\mathbb{P}_{\theta})$ and $\ro(\mathbb{Q})$ respectively, with corresponding generics $E$ and $F$.

  Take $\alpha<\theta$ such that $\mathcal{P}(\kappa)^{V[G\cap\theta]}=\mathcal{P}(\kappa)^{V[G_{\alpha}]}$ and $\dot{X}$ a $\mathbb{P}_{\alpha}$-name for a subset of $(2^{\kappa})^{V[G\cap\theta]}$ that encodes $\mathcal{P}(\kappa)^{V[G\cap\theta]}$. Let $\mathbb{M}'$ be the complete subalgebra of $\ro(\mathbb{P}_{\alpha})$ generated by $[\![i\in\dot{X}]\!]$ for $i<(2^{\kappa})^{V[G\cap\theta]}$. We can then apply the quotient of $\ro(\mathbb{P}_{\theta})$ to $\mathbb{M}'$ to get $\mathbb{M}^*\leq\mathbb{M}$ that adds $\mathcal{P}(\kappa)^{V[G\cap\theta]}$. This means that $\mathbb{M}/\mathbb{M}^*$ will have no subsets of $\kappa$ left to add, but will still have the $\kappa$-cc so by \ref{kappaCCSubsets} it must be trivial. Therefore we can assume that $\mathbb{M}=\mathbb{M}^*$, so it has size at most $|\mathbb{M}'|\leq |\mathbb{P}_{\alpha}|<\theta$ and is a member of $U_{\theta}$, and similarly for $\mathbb{N}$. Furthermore, since these forcings are small all sets of size less than $\theta$ will have $\mathbb{M}$- or $\mathbb{N}$-names of size less than $\theta$, giving $U_{\theta}[E]=V[G\cap\theta]_{\theta}$ and $U_{\theta}[F]=W_{\theta}$. This makes $U_{\theta}$ a ground of each of these models (it is a class by the usual approximation and covering argument), and so a common ground of them via $\kappa$-cc forcing. This contradicts the non-existence of such a ground that we reflected down to $\theta$.
\end{proof}

The approach of Usuba in \cite{usuba} gives a common ground $U$ of $V[G]$ and $W$ in this scenario, but the forcings in $U$ witnessing this are only guaranteed to be $\kappa^{++}$-cc rather than $\kappa$-cc as shown here.


The requirement of Theorem \ref{main_class} that $V\subseteq W$ is automatic when $V$ is any core model $K$; to see this it is enough to check that $K^{V[G]}=K^V$ and then invoke the standard preservation of $K$ by set forcing. We know that $K\cap H_{\mu}$ is uniformly definable without parameters in $H_{\mu}$ for any cardinal $\mu>\omega$; see \cite{jensen_steel} for details. Given any such $\mu$, fix a set-generic initial segment $G_0$ of $G$ such that $H^{V[G_0]}_{\mu}=H^{V[G]}_{\mu}$. Then $K^{V[G]}\cap H_{\mu}$ will equal $K^{H^{V[G]}_{\mu}}$ by the uniformity of the definition, which equals $K^{H^{V[G_0]}_{\mu}}=K^{V[G_0]}\cap H_{\mu}$ by the uniformity again, which by the preservation of $K$ under set forcing is equal to $K^V\cap H_{\mu}$. Therefore $K^{V[G]}=K^V$.

This requirement $V\subseteq W$ cannot be easily simplified; in particular it is not implied by $V=gM^V$. To illustrate this we construct an example of a universe $V$ and a set $x\in gM^V$ together with a class generic $F$ such that $x\notin gM^{V[F]}$.

\begin{example}
  Take $x$ a Cohen real over $L$. For $\kappa\in\OR$ and $n<\omega$ define $\mathbb{A}_{\kappa, n}$ to be
    $$\Add(\aleph_{\kappa.\omega + n.5 + 1}, \aleph_{\kappa.\omega + n.5 + 3}).$$
  In $L[x]$ define $\mathbb{P}$, $\mathbb{Q}$ and $\mathbb{R}$ to be the Easton support products over $\kappa$ of $\prod_{n<\omega}\mathbb{A}_{\kappa, n}$, $\prod_{n\in x}\mathbb{A}_{\kappa, n}$ and $\prod_{n\notin x}\mathbb{A}_{\kappa, n}$ respectively. It is clear that $\mathbb{P}\cong\mathbb{Q}\times\mathbb{R}$ and that $\mathbb{Q}$ can be regarded as a complete subposet of $\mathbb{P}$. Take $G$ a $\mathbb{P}$-generic over $L[x]$ and split it as $H\times F$ a $(\mathbb{Q}\times\mathbb{R})$-generic. Define $V = L[x][H]$, so $V[F]=L[x][G]$. We see that $x$ is in $gM^{V}$ because it is encoded cofinally in the continuum function of $V$. Now $\mathbb{P}$ is in $L$ so $x$ is Cohen generic over $L[G]$ and we have $L[G]$ a ground of $L[x][G]$ that does not contain $x$, because $\mathbb{P}$ is countably closed and so cannot add new reals. Therefore $x\notin gM^{L[x][G]} = gM^{V[F]}$.
\end{example}

Regardless, we have the following corollary which is a special case of \cite[Corollary 5.5]{usuba}.

\begin{corollary}
  Assume Global Choice. Let $\mathbb{P}$ be a forwards class forcing with generic $G$, formed from $\langle\mathbb{P}_{\alpha}\mid\alpha<\OR\rangle$ with individual generics $G_{\alpha}$. Let there be a stationary class of ordinals $\alpha$ such that $\mathbb{P}_{\alpha}$ is formed by a direct limit of the preceding forcings, and such that $\alpha$ is weakly compact in $V[G_{\alpha}]$. Let $V\subseteq gM^{V[G]}$. Then $gM^{V[G]}=M^{V[G]}$.
\end{corollary}

\section{Characterising the mantle} \label{characterising_the_mantle}

\begin{definition}
  For a forwards class iteration $\mathbb{P}$ with generic $G$ we say $x\in V[G]$ is \textit{caught} by $G$ if
    $$V[G]\models\forall e\;\exists_{unbdd}\theta\exists C\subseteq\theta\;\forall A\subseteq\theta\;(C\in V[e,A]\rightarrow x\in V[A]).$$
\end{definition}

From lemma \ref{EastonHelp} (in the case when $\dot{\mathbb{R}}$ is trivial) we can deduce as follows that if $\mathbb{P}$ is an Easton support iteration of $\Add(\kappa,1)$ at $\kappa$ regular, then all $x\in V[G]$ are caught by $G$. It suffices to consider $x$ that are subsets of some cardinal $\kappa$. Given $e$, take an initial segment $G_{\alpha}$ of the generic such that $x,e\in V[G_{\alpha}]$ and then choose $C$ to be any set of ordinals from $V[G]-V[G_{\alpha}]$; we may regard $C$ as a subset of $\theta$ for unboundedly large $\theta$. Given any $A$ such that $C\in V[e,A]$ then we have $A\notin V[G_{\alpha}]$ so the lemma gives us that $x\in V[A]$ as required.

In some cases we can use this definition to provide a simple characterisation of the mantle, but first we need a lemma.

\begin{lemma} \label{classInts2}
  Let $V$ be a model of ZFC, $\mathbb{P}$ a forwards class forcing with generic $G$, $W$ a model of ZFC such that $V\subseteq W$ and $e\in V[G]$ such that $W[e]=V[G]$. Then $W$ is a class and ground of $V[G]$.
\end{lemma}

\begin{proof}
  Say $\mathbb{P}$ is formed by an iteration of $\mathbb{P}_{\alpha}$ for $\alpha\in\OR$, with associated generics $G_{\alpha}$. We may assume $e\subseteq\kappa$ for some cardinal $\kappa$ and take $\dot{e}$ a $\mathbb{P}$-name for $e$. Choose $\alpha$ such that $\dot{e}$ is a $\mathbb{P}_{\alpha}$-name and let $\mathbb{C}$ be the complete subalgebra of $\ro(\mathbb{P}_{\alpha})$ generated by $[\![i\in \dot{e}]\!]$ for $i<\kappa$. Observe that $|\mathbb{C}|\leq|\ro(\mathbb{P}_{\alpha})|=2^{|\mathbb{P}_{\alpha}|}$.

  Define $\lambda:=((2^{|\mathbb{P}_{\alpha}|})^+)^{V[G]}$ and take $\lhd\in W$ a well-ordering of $H^W_{\lambda}$, so $V(H^W_{\lambda},\lhd)$ will be a model of the axiom of choice and thus of ZFC. Find $\beta>\alpha$ such that $V\subseteq V(H^W_{\lambda},\lhd)\subseteq V[G_{\beta}]$. This gives us
    $$V\subseteq V(H^W_{\lambda},\lhd)\subseteq V(H^W_{\lambda},\lhd)[e]\subseteq V[G_{\beta}].$$
  Apply the intermediate model theorem to get a complete subalgebra $\mathbb{A}$ of $\mathbb{P}_{\beta}$ such that $V[G\cap\mathbb{A}]=V(H^W_{\lambda},\lhd)$. We also have the quotient $\mathbb{P}_{\beta}/(G\cap\mathbb{A})\in V(H^W_{\lambda},\lhd)$ and can form its complete subalgebra $\mathbb{E}:=\mathbb{C}/(G\cap\mathbb{A})\in W$, which we know will have size less than $\lambda$, and the associated generic which we call $F$. Therefore any dense $D\subseteq\mathbb{E}$ from $W$ will be a member of $V(H^W_{\lambda},\lhd)$, so $F$ will be generic over $W$. We have found $\mathbb{E}\in W$ with a generic $F$ such that $W[F]=W[e]=V[G]$.

  Now by the usual arguments that $W$ will have the $\lambda$-covering and $\lambda$-approximation properties in $V[G]$ and so it is a class of $V[G]$, definable from the parameter $\mathcal{P}(\lambda)^W$.
\end{proof}

We pause to note that the assumption that $W$ is within a set distance of $V[G]$ is essential not only to $W$ being a ground of $V[G]$, but also to its being a class.

\begin{example} \label{partial_lottery}
  For $\kappa$ regular let $\mathbb{A}_{\kappa}$ be $\Add(\kappa, 1)$ and let $\mathbb{T}$ be the trivial forcing $\{*\}$. For forcings $\mathbb{P}$ and $\mathbb{Q}$ we write $\mathbb{P}\oplus\mathbb{Q}$ for the lottery sum that chooses one of $\mathbb{P}$ and $\mathbb{Q}$ to force with.

  There is a complete embedding $\mathbb{A}_{\kappa}\oplus\mathbb{T}\rightarrow\mathbb{A}_{\kappa}$ given by sending $p\in\Add(\kappa,1)$ to $\langle 0\rangle\frown p$ and $*$ to $\langle 1\rangle$. We can use this to create a complete embedding from $\mathbb{Q}$ a class Easton support product of $\mathbb{A}_{\kappa}\oplus\mathbb{T}$ into $\mathbb{P}$ a class Easton support product of $\mathbb{A}_{\kappa}$. Thus we may regard $\mathbb{Q}$ as a complete subposet of $\mathbb{P}$. Take $G$ a class generic for $\mathbb{P}$ and let $S$ be the class of ordinals $\kappa$ at which non-trivial forcing is performed by $G\cap\mathbb{Q}$. Then $S$ is a class in $V[G\cap\mathbb{Q}]$; we will show that it is not one in $V[G]$, and so $V[G\cap\mathbb{Q}]$ cannot be a class in $V[G]$.

  Suppose otherwise, so there is a formula $\varphi$ and parameter $a$ such that $\varphi(a,\eta)^{V[G]}\leftrightarrow\eta\in S$ for all ordinals $\eta$. Take $\dot{a}$ a $\mathbb{P}$-name for $a$ and $\dot{S}$ a class $\mathbb{P}$-name for $S$, then find $p\in G$ such that $p\forces\forall\eta:\varphi(\dot{a},\eta)\leftrightarrow\eta\in\dot{S}$. Split $\mathbb{P}$ as $\mathbb{P}_0\times\mathbb{P}_1$ and $G$ as $G_0\times G_1$ so that $p\in\mathbb{P}_0$ and $\dot{a}$ is a $\mathbb{P}_0$-name. Then in $V[G_0]$ we have $\forces_{\mathbb{P}_1}\forall\eta:\varphi(a,\eta)\leftrightarrow\eta\in\dot{S}$. Choose $\kappa$ such that $\mathbb{A}_{\kappa}$ is part of $\mathbb{P}_1$. Then without loss of generality $\kappa\in S=\dot{S}[G]$ so $\varphi(a,\kappa)^{V[G]}$, but we can form $G'$ from $G$ by swapping the first co-ordinate of $G\cap\mathbb{A}_{\kappa}$ while maintaining $V[G']=V[G]$, so $\kappa\notin\dot{S}[G']$ but $\varphi(a,\kappa)^{V[G']}$. This is a contradiction.
\end{example}

We are now ready to give a characterisation of the mantle.

\begin{theorem}
  Assume $V=L$ and that there exists a stationary class of strongly inaccessible cardinals. Let $\mathbb{P}$ be a tame class Easton support iteration of set forcings, $G$ a generic for $\mathbb{P}$, and $x$ be a member of $V[G]$. Then $x$ is caught iff $x\in M^{V[G]}$.
\end{theorem}

\begin{proof}
  For the forwards direction we are given $W$ a ground of $V[G]$, say via $V[G] = W[J]$ for $J$ a $\mathbb{S}$-generic. Define $e$ to be this $J$, and take $\theta>>|\mathbb{S}|$ and $C$ from the definition of `caught'. Take a $\mathbb{S}$-name in $W$ for $C$, and let $A$ be an encoding of this name as a subset of $\theta$; then $C\in L[e, A]$, yielding $x\in V[A]\subseteq W$.

  For the converse, using Global Choice we may assume that all members $\mathbb{P}$ are ordinals. Say $\mathbb{P}$ is formed by the Easton support iteration $\langle\mathbb{P}_i\mid i<\OR\rangle$. There will be a club class of $\theta$ such that $\bigcup_{i<\theta}\mathbb{P}_i=\theta$.  When $\theta$ is also strongly inaccessible the use of Easton support gives that $\mathbb{P}_{\theta}=\bigcup_{i<\theta}\mathbb{P}_i$, so $\mathbb{P}_{\theta}$ will be a forcing on $\theta$ with generic $G\cap\theta$.

  We have that $L[G]\models x\in M$ so by reflection down the $\langle L[G]_{\theta}\mid\theta\in\OR\rangle$ hierarchy there is a club class of $\theta$ such that $L[G]_{\theta}\models x\in M$. Now for each ordinal $\alpha$ there will be some $\beta$ such that for every member $z$ of $L[G]_{\beta}$ a name for $z$ occurs in $L_{\alpha}$; this will give us that $L[G]_{\alpha}\subseteq L_{\beta}[G]$. Therefore there is a club class of cardinals $\theta$ such that $L[G]_{\theta}=L_{\theta}[G]$.

  Suppose $x$ is not caught, which is to say $$L[G]\models\exists e\;\forall_{large}\theta\forall C\subseteq\theta\;\exists A\subseteq\theta\;(C\in L[e,A]\wedge x\notin L[A]).$$ Fix such an $e$ and choose $\theta$ strongly inaccessible and in the club classes considered above such that $e\in L[G]_{\theta}$. Let $C$ be $G\cap\theta$ and take $A\subseteq\theta$ such that $G\cap\theta\in L[e,A]$ and $x\notin L[A]$.

  This gives
    $$L_{\theta}[G\cap\theta]\subseteq L[G\cap\theta]_{\theta}\subseteq L[e,A]_{\theta}\subseteq L[G]_{\theta}=L_{\theta}[G]=L_{\theta}[G\cap\theta]$$
  from which $L_{\theta}[G\cap\theta]=L[e, A]_{\theta}$. Since $A\subseteq\theta$ and $e$ is small, we can use \ref{reflection_helper_2} to see that this is in turn equal to $L_{\theta}[e,A]$, which equals $L_{\theta}[A][e]$.

  We can now apply \ref{classInts2} within $L_{\theta}[G\cap\theta]$ to see that $L_{\theta}[A]$ is a ground thereof; $x\notin L[A]$ so this ground omits $x$. But $L_{\theta}[G\cap\theta]$ equals $L[G]_{\theta}$ and so believes that $x$ is a member of its mantle, which is a contradiction.
\end{proof}



\section{Intersections with generic grounds II} \label{intersections_part_2}

We now present an alternative analysis of the intersection of a universe with one of its generic grounds, which gives a weaker result but avoids the need for large cardinal assumptions. First we recall the following definitions.

\begin{definition}
  Let $V\subseteq W$ be models of ZFC. We say that \textit{Jensen covering} holds between $V$ and $W$ if for any uncountable $X\in W$ with $X\subseteq V$ there is a $Y\in V$ such that $X\subseteq Y$ and $|X|=|Y|$ in $W$. We say \textit{weak covering} holds between $V$ and $W$ if for all singular strong limit cardinals $\lambda$ of $W$ we have $(\lambda^+)^V=(\lambda^+)^W$.
\end{definition}

When $V$ is a small core model such as $L$ or $L[0^{\#}]$ we will have Jensen covering between $V$ and its set-generic extensions. This is not true if $V$ is a larger core model, such as the core model for one Woodin cardinal, but then we will still have weak covering between $V$ and its set-generic extensions, and $V$ will also still satisfy GCH. Therefore we present slightly different arguments for the two situations.

\begin{theorem} \label{weak_main}
  Assume Global Choice and that Jensen covering holds between $V$ and its set-generic extensions. Let $\mathbb{P}$ be a forwards class forcing with generic $G$, and $W$ a generic ground of $V[G]$ such that $V\subseteq W$. Then $V[G]\cap W$ is not contained in $V[c]$ for any $c\in V[G]$.
\end{theorem}

\begin{proof}
  Say $V[H]$ is a generic ground via $V[G][I]=V[H][J]$ where $I$ is $\mathbb{R}$-generic over $V[G]$ and $J$ is $\mathbb{S}$-generic over $V[H]$. Suppose there is $c\in V[G]$ with $V[G]\cap V[H]\subseteq V[c]$.

  Split $\mathbb{P}$ in $V$ as $\mathbb{P}^0*\dot{\mathbb{P}^{>0}}$ and correspondingly $G$ as $G^0*G^{>0}$ such that $\mathbb{R}, c\in V[G^0]$ and $J\in V[G^0,I]$. Observe that $I$ and $G^{>0}$ are mutually generic over $V[G^0]$ so we can regard the extension $V[G,I]$ as $V[G^0,I][G^{>0}]$. Take $\beta$ such that $G^0, I\in W_{\beta}[J]$ and $\lhd$ a well-ordering of $W_{\beta}$. This gives us
    $$V[G^0,I]\subseteq V(W_{\beta},\lhd)[J]\subseteq V[G^0,I][G^{>0}]$$
  and we can apply the intermediate model theorem in $V[G^0,I]$ to split $\mathbb{P}^{>0}:=\dot{\mathbb{P}}^{>0}[G^0]$ as $\mathbb{P}^1*\dot{\mathbb{P}}^{>1}$ and $G^{>0}$ correspondingly as $G^1*G^{>1}$ so that $U:=V[G^0,I][G^1]=V(W_{\beta},\lhd)[J]$. (Technically we must first cut down $\mathbb{P}^{>0}$ to a large enough set initial segment and then apply the theorem to that.)

  Take $\kappa$ an uncountable cardinal in $V[G,I]$ such that $U\models |G^0*I*G^1|\leq\kappa$, then $X\subseteq\OR$ encoding $\mathcal{P}(\kappa)^{V[G,I]}$. Since $V[G]$ is a forwards class forcing extension it is possible to split $\mathbb{P}^{>1}:=\dot{\mathbb{P}}^{>1}[G^0*I*G^1]$ in $V[G^0*I*G^1]$ as $\mathbb{P}^2*\dot{\mathbb{P}}^{>2}$ and $G^{>1}$ as $G^2*G^{>2}$ so that $G^2\notin V[G^0,I,G^1, X]$. Find $\beta'>\beta$ such that $G^2, I\in W_{\beta}[J]$ and $\lhd'$ a well-ordering of $W_{\beta'}$. We get
    $$V\subseteq V(W_{\beta},\lhd)\subseteq V(W_{\beta'},\lhd')\subseteq V[G^0, I, G^{>1}]$$
  so the intermediate models theorem will yield a forcing $\mathbb{Q}\in V$ together with a generic $H$ such that $V[H]=V(W_{\beta'},\lhd')$, then a second application will split them as $\mathbb{Q}=\mathbb{Q}^1*\dot{\mathbb{Q}}^{>1}$ and $H=H^1*H^{>1}$ such that $V[H^1]=V(W_{\beta},\lhd)$. Now we have
    $$U\subseteq U[G^2]= V[G^0,I,G^1,G^2]\subseteq V(W_{\beta}',\lhd')[J]=V(W_{\beta},\lhd)[H^{>1}][J]$$
  where $\mathbb{Q}^{>1}:=\dot{\mathbb{Q}}^{>1}[H^1]$ and $\mathbb{S}$ are both members of $V(W_{\beta},\lhd)$, so $H^{>1}$ and $J$ are mutually generic and can be swapped to get
    $$U\subseteq U[G^2]\subseteq V(W_{\beta},\lhd)[J][H^{>1}]=U[H^{>1}]$$
  allowing us to use the intermediate model theorem in $U$ to split $\mathbb{Q}^{>1}$ as $\mathbb{Q}^2*\dot{\mathbb{Q}}^3$ and $H^{>1}$ as $H^2*H^3$ such that $U':=U[H^2]=U[G^2]$.

  Next by \cite[Lemma 25.5]{old_jech} applied in $U$ there are $p\in G^2$ and $q\in H^2$ together with an isomorphism $\pi:\mathbb{P}^2_p\rightarrow\mathbb{Q}^2_q$ such that $\pi``G^2_p=H^2_q$; we may assume $p$ and $q$ are trivial. (This theorem applies only to set forcing which is one reason why we need to cut down $V[G,I]$ to $U'$ first.)

  Define $\mathbb{A}:=\mathbb{P}^0*\dot{\mathbb{R}}*\dot{\mathbb{P}}^1$ and $F:=G^0*I*G^1$, so $U=V[F]$. Take $\dot{\mathbb{Q}}^2$ an $\mathbb{A}$-name for $\mathbb{Q}^2$ and $\dot{\pi}$ an $\mathbb{A}$-name for $\pi$. For each $a\in\mathbb{A}$ define $\pi_a\in V$ to be the set of pairs $(\dot{p},\dot{q})$ of $\mathbb{A}$-names, the first a member of $\dot{\mathbb{P}}^2$ and the second a member of $\dot{\mathbb{Q}}^2$, such that $a\forces\dot{\pi}(\dot{p})=\dot{q}$. Then $\bigcup_{a\in F}\pi_a$ will be realised as $\pi$ under $F$.

  Now for each $a\in F$ we have $\dom\pi_a\cap G^2\in V[G]$ but also $\dom\pi_a\cap G^2 = \pi^{-1}``(\im\pi_a\cap H^2)\in V[H]$; recalling that $V[G]\cap V[H] \subseteq V[c]$ we get $\dom\pi_a\cap G^* \in V[c]\subseteq V[G_0]$. Note that for $a\notin F$ the partial function $\pi_a$ will not take $G^2$ to $H^2$ so $\dom\pi_a\cap G^2$ may not be in $V[H]$; this means that $U$ is unable to construct a set of the relevant $\dom\pi_a\cap G^2$ and so cannot simply take their union to recover $G^2$.

  However, there is in $U'$ the set $S:=\{\dom\pi_a\cap G^2\mid a\in F\}$ which is a subset of $U$. Covering holds between $V$ and $U'$, and hence between $U$ and $U'$. Thus we can find $T\in U$ of size $\kappa$ with $S\subseteq T$. Fix a bijection $f:\kappa\rightarrow T$ in $U$; then we have $f^{-1}``S\subseteq\kappa$ with $f^{-1}``S\in U'$ and $G^2\in U[f^{-1}``S]\subseteq U[X]=V[G^0,I,G^1, X]$, contradicting the choice of $G^2$.
\end{proof}

We can weaken the assumption of covering to weak covering between $U'$ and $V$ if we also assume that the GCH holds on a tail in $V$, as follows.

\begin{theorem}
  Assume Global Choice, that GCH holds on a tail in $V$, and that weak covering holds between $V$ and its set-generic extensions. Let $\mathbb{P}$ be a forwards class forcing with generic $G$, formed from $\langle\mathbb{P}_{\alpha}\mid\alpha\in\OR\rangle$ such that for all $\kappa$ there is a stationary class of singular cardinals $\theta$ with cofinality greater than $\kappa$ for which $\mathbb{P}_{\theta}$ is a direct limit. Let $W$ be a generic ground of $V[G]$ such that $V\subseteq W$. Then $V[G]\cap W$ is not contained in $V[c]$ for any $c\in V[G]$.
\end{theorem}

\begin{proof}
  We will follow the argument of \ref{weak_main} except that we shall make different choices of $\mathbb{P}^2$ and $G^2$, and our conclusion will be based on weak covering rather than Jensen covering.

  By Global Choice we can assume that $\mathbb{P}$ has the ordinals as its underlying class. Proceed as in \ref{weak_main} up to the choice of $\kappa$. We seek a singular strong limit $\theta$ of cofinality greater than $\kappa$ for which:
  \begin{enumerate}[(a)]
    \item The underlying set of $\mathbb{P}_{\theta}$ is $\theta$.
    \item $\forall x\in V[G,I]_{\theta}\;\exists\alpha<\theta:x\in V[G\cap\alpha,I]$
    \item $\neg\exists\alpha<\theta: G\cap\theta\in V[G\cap\alpha, I]$
    \item $2^{\theta}=\theta^+$ in $V$.
  \end{enumerate}
  There is a stationary class of $\theta$ at which $\mathbb{P}_{\theta}$ is formed as a direct limit, and a club class for which $\bigcup_{\beta<\theta}\mathbb{P}_{\beta}$ has underlying set $\theta$, giving a stationary class where (a) holds. For any $x\in V[G, I]$ there is some ordinal $\alpha$ such that $x\in V[G\cap\alpha,I]$, so there is a club of $\theta$ on which (b) holds. We can assume each step in the iteration forming $\mathbb{P}$ is non-trivial so (c) will hold whenever $\theta$ is a limit ordinal. Finally (d) holds on a tail of $\theta$. Hence finding a $\theta$ which satisfies all of these requirements is possible.

  Split $\mathbb{P}^{>2}$ as $\mathbb{P}^2*\dot{\mathbb{P}}^{>2}$ so that $\mathbb{P}^0*\dot{\mathbb{P}}^1*\dot{\mathbb{P}}^2$ is $\mathbb{P}_{\theta}$. Split $G^{>2}$ as $G^2*G^{>2}$ accordingly.

  We resume the argument of \ref{weak_main} up to the construction of the $\pi_a$ functions. Then in $U$ use (d) to enumerate $\mathcal{P}(\theta)^U$ as $\{ X_i\mid i<\theta^+\}$ and fix an injection $k:F\rightarrow\kappa$. Now $\mathbb{P}^2$ is the quotient of $\mathbb{P}^0*\dot{\mathbb{P}}^1*\dot{\mathbb{P}}^2$ by $G^0*G^1$ so its name $\dot{\mathbb{P}}^2$ has the same underlying set as $\mathbb{P}^0*\dot{\mathbb{P}}^1*\dot{\mathbb{P}}^2$, namely $\theta$. Thus $\dom\dot{\pi}_a\subseteq\theta$ for $a\in F$, and we can recover $G^2$ over $U$ by specifying in $U'$ a function $e:\kappa\rightarrow\theta^+$ such that $\dom\pi_a\cap G^2=X_{e(k(a))}$ for all $a\in F$.

  Take a $\mathbb{P}^2$-name $\dot{e}\in U$ for $e$. By weak covering $(\theta^+)^U$ is preserved in $W$ so $\dot{e}$ must be forced to be bounded in $\theta$, say by $j<\theta^+$. Then $U$ can take the enumeration $\{X_i\mid i<j\}$ and modify it to an enumeration $\{Y_i\mid i<\theta\}$; now $G^2$ is recoverable over $U$ from a similar function $e':\kappa\rightarrow\theta$. However $\cf(\theta)>\kappa$ in $U'$ so again $e'$ must be bounded in $\theta$, and hence a member of $U'_{\theta}=V[G^0,I,G^1,G^2]_{\theta}$. So $G^2\in V[G^0,I,G^1,e']$, contradicting the combination of conditions (b) and (c) for $\theta$, and we are done.
\end{proof}

\section{An intermediate model theorem for class forcing} \label{class_intermediate_models}

We recall that a tame class forcing is always ZFC-preserving, by which we mean that for all generics $G$ the model $\langle V[G],G\rangle$ satisfies ZFC in the language of set theory together with a unary predicate. The presence of this predicate $G$ imposes some difficulties since, as in \ref{partial_lottery}, it may not be a class of $V[G]$. We are interested in building models of theories in the language of set theory without any additional predicates so we must avoid reliance on definability with respect to $G$.

For a class forcing $\mathbb{P}$ the $\mathbb{P}$-names for new sets are always themselves sets. This means that we do not necessarily have the maximum principle, as maximal antichains may still be class-sized. For a class $X$ in the extension $V[G]$ there will be a formula $\varphi$ and parameter $a=\dot{a}[G]$ such that $x\in X$ is equivalent to $\varphi(x, a)^{V[G]}$, so we can understand a `class name' $\dot{X}$ for $X$ as being the class of $(\dot{x}, p)$ such that $\dot{a}$ is a $\mathbb{P}$-name and $p$ is a member of $\mathbb{P}$ such that $p\forces\varphi(\dot{x},\dot{a})$.

We now seek a version of the intermediate model theorem that applies to class forcing. The first step is to form a Boolean algebra version of a given class forcing, as is done in \cite[Lemma 61]{Reitz}. For a detailed treatment of Boolean completions of class forcings see \cite{HKLNS}. Here we say two class forcings $\mathbb{P}$ and $\mathbb{Q}$ are \textit{forcing-equivalent} if from any $\mathbb{P}$-generic $G$ we can define a $\mathbb{Q}$-generic $H$ such that $V[G]=V[H]$, and vice-versa.

\begin{lemma} \label{form_algebra}
  Let $\mathbb{P}$ be a forwards class forcing with generic $G$. Then there is a class Boolean algebra that is complete under set-sized supremums and infimums and is forcing-equivalent to $\mathbb{P}$.
\end{lemma}

\begin{proof}
  Say $\mathbb{P}$ is given by the iteration of $\mathbb{P}_{\alpha}$ for $\alpha\in\OR$. For each $\alpha$ define $\mathbb{B}_{\alpha}$ to be the regular open algebra generated by $\mathbb{P}_{\alpha}$. For $\alpha<\beta$ we can embed as follows:
  \begin{align*}
    i_{\alpha,\beta}:\mathbb{B}_{\alpha}&\hookrightarrow\mathbb{B}_{\beta}\\
    A&\mapsto\dcl(A)
  \end{align*}
  where $\dcl(A)$ is the downwards closure of $A$ in $\mathbb{P}_{\beta}$. Clearly $\dcl(A)$ will be open in $\mathbb{P}_{\beta}$ so we just need to check it is regular, which is to say that for every $p\in\mathbb{P}_{\beta}$ with $\dcl(A)$ dense below $p$ then in fact $p\in\dcl(A)$. Given such a $p$, we we will show that $A$ is dense below $p\upharpoonright\mathbb{P}_{\alpha}$ in $\mathbb{P}_{\alpha}$ and then the regularity of $A$ will tell us that $p\upharpoonright\mathbb{P}_{\alpha}\in A$ and since $p\upharpoonright\mathbb{P}_{\alpha}\leq p$ we will be done.

  Given $u\leq p\upharpoonright\mathbb{P}_{\alpha}$ in $\mathbb{P}_{\alpha}$ take $q\in\mathbb{P}_{\beta}$ such that $q\upharpoonright\mathbb{P}_{\alpha}=u$ and $q\leq p$. Then we can find $r\leq q$ with $r\in\dcl(A)$; this yields $r\upharpoonright\mathbb{P}_{\alpha}\leq u$ and $r\upharpoonright\mathbb{P}_{\alpha}\in A$ so we are done.

  We now check that this embedding is complete. First, given $\{A_k\mid k<\lambda\}\subseteq\mathbb{B}_{\alpha}$ we want $i_{\alpha,\beta}(\bigwedge_k A_k)=\bigwedge_k i_{\alpha,\beta}(A_k)$, which is to say that $\dcl(\bigcap_k A_k)=\bigcap_k\dcl(A_k)$. Now for any $x\in\mathbb{B}_{\beta}$ we have that $x\in\dcl(\bigcap_k A_k)$ is equivalent to $x\upharpoonright\mathbb{P}_{\alpha}\in\bigcap_k A_k$, which is equivalent to $x\upharpoonright\mathbb{P}_{\alpha}\in A_k$ for all $k<\lambda$ and thence to $x\in\bigcap_k\dcl(A_k)$.

  Second, given any $A\in\mathbb{B}_{\alpha}$ we want $i_{\alpha,\beta}(\neg A)=\neg i_{\alpha,\beta}(A)$, which means that $\dcl(\mathring{A^C})=\mathring{\dcl(A)^C}$. Note that any $x\in\mathbb{P}_{\beta}$ can be regarded as a pair $(x_0,\dot{x}_1)$ such that $x_0\in\mathbb{P}_{\alpha}$. Now $x\in\dcl(\mathring{A^C})$ is equivalent $x_0\in\mathring{A^C}$ and thus to $\dcl(\{x_0\})\cap A=\phi$. Meanwhile $x\in\mathring{\dcl{A}^C}$ is equivalent to $\dcl(\{x\})\cap\dcl(A)=\phi$. Given $z\in\dcl(\{x\})\cap\dcl(A)$ then splitting $z$ in the same way as $x$ we have $z_0\in A$ with $z_0\leq x_0$ so $z_0\in\dcl(\{x_0\})\cap A$. Conversely given $a\in\dcl(\{x_0\})\cap A$ then we can form $(a,\dot{x}_1)\in\dcl(\{x\})\cap\dcl(A)$. Therefore $i_{\alpha,\beta}$ is a complete embedding.

  These embeddings allow us to take the direct limit of the $\mathbb{B}_{\alpha}$ to form a class Boolean algebra $\mathbb{B}$ which will be set-complete, though not class-complete. Then any $\mathbb{P}$-name for a set will be a $\mathbb{P}_{\alpha}$-name for some $\alpha$, and so can be converted into a $\mathbb{B}_{\alpha}$-name. Similarly any $\mathbb{B}$-name for a set can be converted into a $\mathbb{P}$-name so they will give the same generic extension.
\end{proof}

Given a ground $W$ of a class forcing extension $V[G]$ such that $V\subseteq W\subseteq V[G]$ we would now like to mirror the proof of the intermediate model theorem for sets by finding a sub-algebra class forcing which generates $W$. Unfortunately there are difficulties in doing so, as illustrated by the next result.

\begin{proposition} \label{lottery_sum}
  There is a forwards class forcing $\mathbb{P}$ together with generic $G$ such that if we construct the Boolean algebra $\mathbb{B}$ from \ref{form_algebra} with associated generic $G^*$ then there is a subalgebra $\mathbb{C}\leq\mathbb{B}$, complete under set-sized operations, such that $G^*\cap\mathbb{C}$ is not generic for $\mathbb{C}$.
\end{proposition}

\begin{proof}
  Consider $\mathbb{P}$ the class forcing that is given by a class-size lottery sum of the individual set forcings $\mathbb{A}_{\kappa}:=\Add(\kappa, 1)$ for $\kappa$ regular, which we form by taking their disjoint union and adding a top element. This forcing will select a single regular cardinal and then add a Cohen subset of it, so it is clear that it is ZFC-preserving and therefore tame.

  \begin{claim}
    For any sub-class $X$ of $\mathbb{P}$ the following are equivalent:
    \begin{itemize}
      \item $X$ is regular and open in $\mathbb{P}$.
      \item $X\cap\mathbb{A}_{\kappa}$ is regular and open for each $\kappa$, and $1_{\mathbb{P}}\in X$ iff $\coprod_{\kappa}\mathbb{A}_{\kappa}\subseteq X$.
    \end{itemize}
  \end{claim}
  \begin{proof}
    For the forwards direction we are told that $X$ is regular and open in $\mathbb{P}$. Clearly $X\cap\mathbb{A}_{\kappa}$ is open in each $\mathbb{A}_{\kappa}$. Given $p\in\mathbb{A}_{\kappa}$ such that $X$ is dense below $p$ in $\mathbb{A}_{\kappa}$ then $X$ is also dense below $X$ in $\mathbb{P}$ and so $p\in X$; therefore $X$ is regular in $\mathbb{A}_{\kappa}$. If $1_{\mathbb{P}}\in X$ then $\coprod_{\kappa}\mathbb{A}_{\kappa}\subseteq X$ by openness, and if $\coprod_{\kappa}\mathbb{A}_{\kappa}\subseteq X$ then $1_{\mathbb{P}}\in X$ by regularity.

    For the reverse direction it is clear that $X$ is open in $\mathbb{P}$. So we consider $p\in\mathbb{P}$ such that $X$ is dense below $p$ in $\mathbb{P}$. If $p\in\mathbb{A}_{\kappa}$ for some $\kappa$ then $X$ is dense below $p$ in $\mathbb{A}_{\kappa}$, and $X\cap\mathbb{A}_{\kappa}$ is regular here so $p\in X$. Otherwise $p=1_{\mathbb{P}}$ and for any $q\in\mathbb{A}_{\kappa}$ we have the $X$ is dense below $q$ in $\mathbb{A}_{\kappa}$, from which $q\in X$. Thus $\coprod_{\kappa}\mathbb{A}_{\kappa}\subseteq X$ so $p=1_{\mathbb{P}}\in X$.
  \end{proof}

  We can therefore form a notion of $\ro(\mathbb{P})$ by taking all regular open classes $X\subseteq\mathbb{P}$ such that either $X\cap\mathbb{A}_{\kappa}=\phi$ for all but set-many $\kappa$ or $\mathbb{A}_{\kappa}\subseteq X$ for all but set-many $\kappa$. Classes of this form are uniformly definable from sets, so we can regard $\ro(\mathbb{P})$ as a class. It will be closed under negations and set-size suprema and infima, though not class-size ones.

  Define $\mathbb{Q}$ in the same way, except that it will be a lottery sum over all $\Add(\kappa, 1)$ for regular $\kappa>\omega$. We can embed $\ro(\mathbb{Q})$ into $\ro(\mathbb{P})$ by sending $X$ to itself if $X\cap\mathbb{A}_{\kappa}=\phi$ for all but set-many $\kappa$, and to $X\cup\mathbb{A}_{\omega}$ if $\mathbb{A}_{\kappa}\subseteq X$ for all but set-many $\kappa$. This embedding will respect all of the set-size Boolean algebra operations in $\mathbb{Q}$ and $\mathbb{P}$. It will not however respect certain class-size operations that it is possible to perform; for example the supremum of $\mathbb{A}_{\kappa}$ for $\kappa>\omega$ in $\mathbb{Q}$ will be $\mathbb{Q}$ which is then embedded as $\mathbb{P}$; the same supremum in $\mathbb{P}$ will be $\mathbb{P}-\mathbb{A}_{\omega}$.

  Consider now a generic $G$ for $\mathbb{P}$ that is a subset of $\mathbb{A}_{\omega}$, and $G^*$ the induced generic for $\ro(\mathbb{P})$, which will consist of all regular open subsets that meet $G\subseteq\mathbb{A}_{\omega}$. The only members of $\ro(\mathbb{Q})$ (as embedded in $\ro(\mathbb{P})$) that meet $\mathbb{A}_{\omega}$ are those $X$ such that $\mathbb{A}_{\kappa}\subseteq X$ for all but set-many $\kappa$, so the filter $G^*\cap\ro(\mathbb{Q})$ on $\ro(\mathbb{Q})$ contains only $X$ of this form. This makes it disjoint from many dense sub-classes of $\ro(\mathbb{Q})$, for example the class of all $X$ such that $X\cap\mathbb{A}_{\kappa}=\phi$ for all but set-many $\kappa$. Therefore $G^*\cap\ro(\mathbb{Q})$ is not generic for the subalgebra $\ro(\mathbb{Q})$.

  It remains to show that the forcing $\mathbb{P}$ can be constructed as a full-support iteration. We build the iteration $\langle\mathbb{P}_{\alpha},\dot{\mathbb{R}}_{\alpha}\mid\alpha\in\OR\rangle$ by defining $\mathbb{P}_1=\mathbb{R}_0 = \Add(\omega, 1)\oplus\{*\}$ the lottery sum of $\Add(\omega, 1)$ and the trivial forcing. If $\aleph_{\alpha}$ is not regular then $\dot{\mathbb{R}}_{\alpha}$ will be trivial forcing. Otherwise use the maximum principle in $\mathbb{P}_{\alpha}$ to define $\dot{\mathbb{R}_{\alpha}}$ as $\{*\}$ if any co-ordinate of $\mathbb{P}_{\alpha}$ has a non-trivial condition, and as $\Add(\aleph_{\alpha},1)\oplus\{*\}$ if all co-ordinates of $\mathbb{P}_{\alpha}$ are performing trivial forcing. Thus will result in $\mathbb{P}$ having a dense subset consisting of sequences which have exactly one non-trivial co-ordinate, which will be a member of some $\Add(\aleph_{\alpha},1)$. The use of full-support iteration means that this $\alpha$ may be arbitrarily large.
\end{proof}

For any forwards class forcing $\mathbb{P}$ that is not $\OR$-cc it is possible to produce a similar generic $G$ and set-complete subalgebra $\mathbb{C}$ such that $G\cap\mathbb{C}$ is not generic for $\mathbb{C}$. To do so fix a maximal antichain $\{p_{\alpha}\mid\alpha\in\OR\}$ and use the downwards cones below the $p_{\alpha}$ in place of the $\Add(\kappa,1)$.

This counterexample means that our attempts to construct an intermediate class subalgebra may result in its having an ultrafilter that is not truly generic. Fortunately it will still meet all dense sets, which is enough for most procedures involving the set-sized names. The proof that a ZFC-preserving forcing is tame depends on the use of genuine generics, so we will also have to renounce tameness for our intermediate class subalgebras. Fortunately the main purpose of tameness is the proving of ZFC-preservation, and we are already guaranteed by definition that an intermediate model will satisfy ZFC. We make the following definition.

\begin{definition}
  A \textit{pseudo-class forcing} is a (not necessarily tame) set-complete class Boolean algebra $\mathbb{B}$. A \textit{pseudo-generic} for $\mathbb{B}$ is a an ultrafilter $G\subseteq\mathbb{B}$ such that if $\bar{\mathbb{B}}$ is a set-size complete subalgebra of $\mathbb{B}$ then $G\cap\bar{\mathbb{B}}$ is a generic ultrafilter for $\bar{\mathbb{B}}$, and such that $V[G]$ is a model of ZFC.
\end{definition}

\begin{proposition}\label{classInts}
  Assume Global Choice. Let $\mathbb{P}$ be a forwards class forcing, $\varphi$ a formula and $\dot{a}$ a $\mathbb{P}$-name such that $\mathbb{P}$ forces that $\varphi(-,-,\dot{a})$ is a global well-ordering of $V[\mathbb{P}]$. Let there be a proper class of cardinals $\theta$ that are forced to remain strongly inaccessible, and such that $\varphi(-,-,\dot{a})$ is forced to reflect to $V[\mathbb{P}]_{\theta}$. Let $G$ be generic for $\mathbb{P}$ and $W$ a ground of $V[G]$ such that $V\subseteq W$. Then there is a pseudo-class forcing $\mathbb{Q}$ with pseudo-generic $H$ such that $V[H]=W$ and $H$ is definable from $G$ in $V[G]$.
\end{proposition}

\begin{proof}
  Say $W$ is a ground of $V[G]$ via $V[G]=W[J]$ where $J$ is generic for a poset $\mathbb{S}\in W$. The global choice in $V[G]$ means there is a class surjection $f:\OR\rightarrow V[G]$ definable there. Fix $\dot{f}$ a class $\mathbb{S}$-name for $f$ that is definable in $W$. Then in $W$ we can define a surjective partial function from $\mathbb{S}\times\OR$ to $W$ by sending $(s,\gamma)$ to $w$ iff $s\forces\dot{f}(\gamma)=w$.

  Therefore global choice holds in $W$, and we can find a formula $\psi$ with parameter $a\in W$ such that $\psi(a,-)$ defines in $W$ a class of ordinals that encodes $W$, and we may assume that this fact is forced by $\mathbb{P}$. Fix $\dot{a}$ a $\mathbb{P}$-name for $a$. Our assumptions about the reflectivity of the global choice in $V[G]$ also mean that we have a proper class of cardinals $\theta$ that are strongly inaccessible in $V[G]$ and for which
    $$\forces_{\mathbb{P}}\forall\eta<\theta:\psi(\eta,\dot{a})^{\dot{W}}\leftrightarrow\psi(\eta,\dot{a})^{W_{\theta}}$$
  where the class name $\dot{W}$ here comes from some formula that, in $V[G]$, defines $W$.
  Take $\kappa$ such that $|\mathbb{S}|<\kappa$ and $\dot{a}\in V_{\kappa}$.

  \begin{claim}
    Let $\eta\in\OR$. Then there is in $V$ a set-sized Boolean $\mathbb{P}$-name for whether $\psi(\dot{a},\eta)^{\dot{W}}$.
  \end{claim}
  \begin{proof}
    Fix $\theta>\kappa,\eta$ a cardinal that is forced by $\mathbb{P}$ to remain strongly inaccessible. We can use the increasing distributivity of the iteration forming $\mathbb{P}$ to split it as $\mathbb{P}_0*\dot{\mathbb{P}}_1$ where $\mathbb{P}_0$ is a set forcing and $\mathbb{P}_0$ forces that $\dot{\mathbb{P}}_1$ is $\theta$-distributive; split $G$ correspondingly as $G_0*G_1$. Now $W_{\theta}$ is definable within $V[G]_{\theta}$ as the unique inner model satisfying $\kappa$-covering and $\kappa$-approximation whose power set of $\kappa$ is $\mathcal{P}(\kappa)^W$. But $V[G]_{\theta}=V[G_0]_{\theta}$ so there is a $\mathbb{P}_0$-name for $V[G]_{\theta}$.

    Now in $\mathbb{P}_0$ we can find a Boolean name for whether $\psi(\dot{a},\eta)^{\dot{W}_{\theta}}$, which is forced to equal $\psi(\dot{a},\eta)^{\dot{W}}$. Since $\mathbb{P}_0$ embeds into $\mathbb{P}$ this gives us a set-sized name in $\mathbb{P}$.
  \end{proof}

  Use \ref{form_algebra} to form a set-complete Boolean algebra $\mathbb{B}$ from $\mathbb{P}$. For any ordinal $\eta$ the set-sized name from the claim allows us to take a set-sized supremum in $\mathbb{B}$ to get a valuation $[\![\psi(\dot{a},\eta)^{\dot{W}}]\!]\in\mathbb{B}$. Define $\mathbb{C}$ to be the subalgebra of $\mathbb{B}$ generated (under set-sized supremums) by these valuations for $\eta\in\OR$. Let $G^*$ be the generic of $\mathbb{B}$ induced by $G$; we claim that $W=V[G^*\cap\mathbb{C}]$.

  For any $x\in W$ there is a $\theta$ such that $x$ is encoded by $\{\eta<\theta\mid\psi(a,\eta)^W\}$, which is equal to $\{\eta<\theta\mid [\![\psi(\dot{a},\eta)^{\dot{W}}]\!]\in G^*\}$, a member of $V[G^*\cap\mathbb{C}]$. Conversely given any $x\in V[G^*\cap\mathbb{C}]$ take $\dot{x}$ a name for $x$ and $\beta$ such that $\dot{x}$ is a $\mathbb{C}\cap\mathbb{B}_{\beta}$-name. Take $\theta$ such that $\mathbb{C}\cap\mathbb{B}_{\beta}$ is contained in the subalgebra generated by $\{[\![\psi(\dot{a},\eta)^{\dot{W}}]\!]\mid\eta<\theta\}$. Now $G^*\cap\mathbb{C}\cap\mathbb{B}_{\beta}$, and hence $x$, can be obtained from $\{\eta<\theta\mid\psi(a,\eta)^W\}\in W$.

  We already know that $W$ is a model of ZFC, so $\mathbb{C}$ is a pseudo-class forcing with pseudo-generic $G^*\cap\mathbb{C}$.
\end{proof}

The precise requirements on the notion of global choice in the forcing extension by $\mathbb{P}$ are important here. We could try to take an arbitrary global well-ordering in $V[G]$ and reflect it down the $\langle V[G]_{\theta}\mid\theta\in\OR\rangle$ hierarchy to find $\theta$ such that the well-ordering of $V[G]_{\theta}$ is definable in $V[G]_{\theta}$. However, we might not be able to find $\theta$ such that the well-ordering is forced to be definable in $V[\mathbb{P}]_{\theta}$. For example in the construction of \ref{lottery_sum}, starting from $V=L$, there is forced to be a $\theta$ such that the formula ``$\exists x\notin L$'' reflects to $V[\mathbb{P}]_{\theta}$, but we cannot in $V$ fix a $\theta$ such that ``$\exists x\notin L$'' is forced to reflect to $V[\mathbb{P}]_{\theta}$.

Unfortunately the need here for global choice in $V[G]$ is problematic, for example it does not hold if $\mathbb{P}$ is an Easton-support iteration of $\Add(\kappa,1)$, as follows. Suppose there was a class well-ordering $\lhd$ of $V[G]$, defined with respect to parameter $a$ and with $\mathbb{P}$-name $\dot{\lhd}$. Split $\mathbb{P}$ as $\mathbb{P}_0*\dot{\mathbb{P}}_1$ and $G$ as $G_0*G_1$ so that $x\in V[G_0]$. Take $\delta$ so that $\mathcal{P}(\delta)^{V[G]}\not\subseteq\mathcal{P}(\delta)^{V[G_0]}$. Then in $V[G]$ we can define ``the $\dot{\lhd}$-least subset of $\delta$ not in $V[G_0]$'' using only parameters from $V[G_0]$, which since $\mathbb{P}_1:=\dot{\mathbb{P}}_1[G_0]$ is weakly homogeneous means it is a member of $V[G_0]$. This is a contradiction.

An additional difficulty when attempting to use this approach to apply the argument of \ref{main_sets} to class forcing is that it is not clear that the $\mathbb{P}$-generic $G$ will be definable from the $\mathbb{C}$ and $\mathbb{S}$-generics $G^*\cap\mathbb{C}$ and $J$, and this definability is necessary in order to form a $\mathbb{C}*\mathbb{S}$-name for $G$. If we are in the special case that $G$ is a class of $V[G]$ then this difficulty disappears, and we also obtain the needed Global Choice in $V[G]$ by taking a surjection $e:\OR\rightarrow V^{\mathbb{P}}$ and defining $e':\OR\rightarrow V[G]$ by $\gamma\mapsto e(\gamma)[G]$.

We leave open the question of the extent to which the global choice requirements for $V[G]$ in \ref{classInts} can be weakened or discarded. If this is possible then we might also hope to generalise it to cover any intermediate model between $V$ and $V[G]$, and not just grounds of $V[G]$.

\bibliographystyle{plain}

\end{document}